\documentclass[11pt]{amsart}

\usepackage[colorlinks,linkcolor={blue}]{hyperref}

\usepackage{amssymb}
 \usepackage{graphicx} 
 \usepackage{amscd}
 \usepackage{enumerate}
 \usepackage{cancel}

\usepackage[none]{hyphenat}

\usepackage[colorinlistoftodos, textsize=footnotesize]{todonotes}
\makeatletter
\providecommand\@dotsep{5}
\makeatother
\makeatletter
\@namedef{subjclassname@2020}{\textup{2020} Mathematics Subject Classification}
\makeatother

 \def\a{\alpha}
 \def\da{{\dot\alpha}}
 
 \def\be{\beta}
 \def\dbe{{\dot\beta}}
 \def\de{\delta}
 
 \def\e{\varepsilon}
 \def\deta{{\dot{\eta}}}

 \def\ga{\gamma}
 \def\dga{{\dot{\gamma}}}

 \def\Ga{\Gamma}

 \def\vr{\varphi}

 \def\ka{\kappa}
 \def\la{\lambda}
 
 \def\La{\Lambda}
 \def\si{\sigma}
 \def\Si{\Sigma}
 \def\bSi{{\mathbf\Sigma}}
 \def\om{\omega}

 \def\th{\theta}

 \def\dzeta{{\dot{\zeta}}}

 \def\re{{\mathbb R}}
 \def\na{{\mathbb N}}

 \def\then{\Longrightarrow}
 \def\ov{\overline}
 \def\Z{{\mathbb Z}}

 \def\A{{\mathbb A}}

 \def\cB{{\mathcal B}}
 
 \def\cC{{\mathcal C}}
 
 \def\D{{\mathbb D}}

 \def\E{{\mathbb E}}
 
 \def\cF{{\mathcal F}}
 \def\F{{\mathbb F}}

 \def\cG{{\mathcal G}}
 \def\H{{\mathbb H}}

 \def\cJ{{\mathcal J}}

 \def\cL{{\mathcal L}}

 \def\cN{{\mathcal N}}

 \def\oom{{\ov{m}}}
 \def\cO{{\mathcal O}}
 \def\P{{\mathbb P}}

 \def\SS{{\mathbb S}}

 \def\cT{{\mathcal T}}
 \def\cU{{\mathcal U}}

 \def\V{{\mathbb V}}
 \def\cW{{\mathcal W}}
 \def\W{{\mathbb W}}

 \def\tq1{{\tilde{q}_1}}
 \def\dw{{\dot{w}}}

 \def\X{{\mathbb X}}

 \def\dz{{\dot{z}}}

 \def \lV{\left\Vert}
 \def \rV{\right\Vert}
 \def \ov{\overline}
 
 \def \then{\Longrightarrow}

 \DeclareMathOperator*{\tsum}{{\textstyle \sum}}

 \DeclareMathOperator{\interior}{int}

 \DeclareMathOperator*{\Fix}{Fix} 
 \DeclareMathOperator{\graph}{graph}
 \DeclareMathOperator{\Emb}{Emb}
 \DeclareMathOperator{\Diff}{Diff}
 \DeclareMathOperator{\inj}{inj}
 \DeclareMathOperator{\trap}{trap}

  \renewcommand{\proofname}{{\bf Proof:}}

 \theoremstyle{plain}

 \newtheorem{MainThm}{Theorem}

  \swapnumbers

 \newtheorem{Thm}{Theorem}[section]
 
 \newtheorem{Lemma}[Thm]{\bf Lemma}
 
 \newtheorem{Corollary}[Thm]{\bf Corollary}
 \newtheorem{Theorem}[Thm]{\bf Theorem}
 \newtheorem{Proposition}[Thm]{\bf Proposition}

\newtheorem{claim}[Thm]{\sc Claim}
 
 \theoremstyle{definition}
 \newtheorem{Definition}[Thm]{\bf Definition}

 \theoremstyle{remark}
 \newtheorem{Remark}[Thm]{\bf Remark}

 \newtheoremstyle{Cl}
  {5pt}
  {3pt}
  {\sl}
  {}
  {\it}
  {:}
  {.5em}
  {}

 \theoremstyle{Cl}

 \def\begincproof{
                  \renewcommand{\proofname}{\it Proof:}
                  \begin{proof}
                 }

 \def\endcproof{
                \renewcommand{\qedsymbol}{$\diamondsuit$}
                \end{proof} 
                \renewcommand{\qedsymbol}{\openbox}
                \renewcommand{\proofname}{\bf Proof:}
               }

 \title{Homoclinic orbits for geodesic flows of surfaces} 

\author{Gonzalo Contreras}  
\address{Gonzalo Contreras\newline\indent 
Centro de Investigaci\'on en Matem\'aticas\newline\indent 
A.P. 402, 36.000, Guanajuato, GTO, Mexico}

\author{Fernando Oliveira}
\address{Fernando Oliveira\newline\indent 
Universidade Federal de Minas Gerais\newline\indent
Av. Ant\^onio Carlos 6627, 
31270-901, Belo Horizonte,
MG, Brasil.}

\thanks{Gonzalo Contreras is partially supported by CONACYT, Mexico, grant A1-S-10145.}


\begin{document}

\maketitle

\parskip +5pt

\begin{abstract}
We prove that the geodesic flow of a
Kupka-Smale riemannian metric on a closed surface
has homoclinic orbits for all of its hyperbolic closed geodesics.
\end{abstract}


\tableofcontents


 \def\Krad{{\color{red}K_{rad}}}
 \def\Krot{{K_{rot}}}
 \def\Kfix{{K_{fix}}}
 \def\Ws{{W^s}}
 \def\Wu{{W^u}}

 \section{Introduction.}

 \hskip 2.2cm  
{ \LARGE L}et $(M,\rho)$ be a closed (i.e. compact, boundaryless) riemannian surface. 
 \linebreak
 Let
 $
 SM=\{(x,v):\rho(v,v)=1\}
 $
 be its unit tangent bundle with projection $\pi:SM\to M$, $\pi(x,v)=x$.
 The {\it geodesic flow}  $\phi_t:SM\to SM$ of $(M,\rho)$ is defined by
 $\phi_t(x,v)=(\ga(t),\dga(t))$, where $\ga$ is the unit speed geodesic 
 with $(\ga(0),\dga(0))=(x,v)$.
  
  A closed orbit for $\phi_t$ is {\it hyperbolic} if its Floquet 
  multipliers  do not have modulus 1. The (strong) {\it stable} and {\it unstable}
  {\it manifolds} of a point $z\in SM$ are
  $$
  W^{s,u}(z):=\{w\in SM: \lim\nolimits_{t\to\pm\infty}d(\phi_t(w),\phi_t(z))=0\},
  $$
  respectively.
  For a subset $A\subset SM$ define
  $$
  W^{s,u}(A)=\textstyle\bigcup_{a\in A}W^{s,u}(a).
  $$
  For a hyperbolic closed geodesic $\ga$ the sets $W^s(\dga)$, $W^u(\dga)$
  are immersed submanifolds of $SM$ either diffeomorphic to a cylinder
  with one boundary $\dga$
  or to a M\"obius band where $\dga$ is its equator, according to wether
  the Floquet multipliers of $\dga$ are positive or negative respectively.
  A {\it homoclinic point} of a hyperbolic closed geodesic $\ga$
  is a point in 
  \linebreak
  $(W^s(\dga)\cap W^u(\dga))\setminus\dga$.
  A {\it heteroclinic} point is a point in 
  $(W^s(\dga)\setminus\dga)\cap (W^u(\deta)\setminus\deta)$, where $\dga$,
  $\deta$ are two hyperbolic  closed orbits of $\phi_t$.
  
  Homoclinic points where first discovered by Henri Poincar\'e in 1889 (cf. Andersson~\cite{Andersson})
and named in Poincar\'e~\cite[\S 395]{PoincareIII}.
It is well known the paragraph of Poincar\'e~\cite[vol. III,  \S 397]{PoincareIII},~\cite[\S 5]{Andersson} 
describing his admiration of the 
complexity of the dynamics implied by the existence of a transversal homoclinic point.

  We say that the riemannian metric $\rho$ is {\it Kupka-Smale} if
  \begin{enumerate}[(i)]
  \item\label{ks1} The Floquet multipliers of every periodic orbit are not roots of unity.
  \item\label{ks2} The heteroclinic intersections of hyperbolic orbits 
  $W^s(\dga)\pitchfork W^u(\deta)$ are transversal.
  \end{enumerate}

 For any $r\in\na$, $r\ge 2$, the set of $C^r$ riemannian metrics whose geodesic flow is 
 Kupka-Smale is residual in the set of $C^r$ riemannian metrics in $M$, see
 Contreras, Paternain \cite[Thm.~2.5]{CP2}. Clarke~\cite{aclarke1} proves that Kupka-Smale
 metrics are also residual in the $C^\om$ topology for analytic hypersurfaces of $\re^n$, $n\ge 3$.
 Here we prove

\begin{MainThm}\label{TA}\quad

For a Kupka-Smale riemannian metric on a closed surface
every hyperbolic closed 
geodesic $\ga$ has homoclinic orbits 
in all the components of 
$W^s(\dga)\setminus \dga$ and of $W^u(\dga)\setminus \dga$
and satisfy $\ov W^s(\dga)=\ov W^u(\dga)$.
\end{MainThm}

The importance of finding homoclinic orbits is that 
in any neighborhood of the 
\linebreak
homoclinic orbit one finds
a horseshoe with complicated dynamics. This dynamics can be 
coded using symbolic dynamics and implies positive (local)
topological entropy, infinitely many periodic orbits shadowing
the homoclinic, infinitely many homoclinics and 
\linebreak
exponential growth
of periodic orbits in a neighborhood of the homoclinic. 
Homoclinics prevent integrability \cite[\S III.6]{Moser},
 and can also be used to obtain Birkhoff sections
 \cite{CM1}. They are also the basic skeleton for Mather 
acceleration theorems in Arnold diffusion 
\cite{Mat15}, \cite{CY1}, \cite{DLS1}, \cite{GLS}.

Also theorem~\ref{TA} may help to prove 
that the closed orbits for the geodesic flow of surfaces are generically dense
in the phase space. A conjecture by Poincar\'e~\cite[vol. I, p.82 \S 36]{PoincareI}
stated  for the three body problem.
By now it is only known that their projection to the surface 
is generically dense,  Irie~\cite{irie1}.

It is well known that $C^2$, $r\ge 2$,
generic riemannian surfaces of genus $g\ge 2$ have homoclinic orbits,
see e.g.~\cite{CP2}.
Contreras and Paternain~\cite{CP2} proved that $C^2$ generic metrics 
on $\SS^2$ or $\re\P^2$ have some orbits with homoclinic orbits.
Knieper and Weiss~\cite{KW2} extended this result to the $C^\infty$ topology
and Clarke~\cite{aclarke1} proved it for analytic convex surfaces in $\re^3$ 
and the $C^\om$ topology. 
Xia and Zhang~\cite{XiaZhang} prove that for a $C^\infty$ generic metric
of positive curvature in $\SS^2$, every hyperbolic periodic orbit has homoclinics.
Contreras~\cite{geod} proves that $C^2$ generic metrics on any closed manifold
have homoclinics.

The $C^\infty$, $C^\om$ results \cite{KW2}, \cite{aclarke1}, \cite{XiaZhang} in the sphere
use the annular Birkhoff section~\cite[\S VI.10, p.180]{Birk2} for the spheres with positive curvature. 
Then they apply the techniques of 
Pixton~\cite{pixton} and Oliveira~\cite{Ol3}  
for area preserving maps on surfaces of genus 0 to obtain the homoclinics.
These techniques extend to genus 1 but not to higher genus.
The problem with riemannian surfaces which are not  spheres of positive curvature
is that they have Birkhoff sections in the Kupka-Smale case~\cite{CM2}, \cite{CKMS}
but their genera is not known. 

Instead we construct what we call a complete system 
of surfaces of section of genera $g\le 1$. With this  we complete
a program initiated by Birkhoff in~\cite[\S 28, p.\,281]{Birk3}
with formal justifications using the curve shortening flow \cite{Grayson}, \cite{gage}.
But now the Poincar\'e maps to these surfaces of section
are not continuous. They are essentially discontinuous\footnote{There are 
arbitrarily small curves whose image under the Poincar\'e map have infinite length
and large diameter.}. And the standard (continuity) arguments of Mather~\cite{Mat9}
and Oliveira~\cite{Ol3} for area preserving homeomorphisms can not be applied.
We show how to take advantage of the discontinuities of the Poincar\'e map to obtain
homoclinic orbits for certain closed orbits. For the remaining hyperbolic orbits we develop
in \cite{OC2} and \cite{OC1} the theories of Mather and Oliveira for partially defined area preserving
homeomorphisms so that they can be applied to our situation.

We also remark that in~\cite{OC1} we show that the usual hypothesis of Moser stability for 
elliptic periodic points in Mather~\cite{Mat9} is not needed. We use instead Theorem~1.2.(4)
from~\cite{OC1} which allows to use only condition~\eqref{ks1} from our Kupka-Smale definition.
\linebreak
Nevertheless, as observed by Xia and Zhang~\cite{XiaZhang}, Fayad and Krikorian~\cite{fakri}
prove that elliptic periodic points are Moser stable if their Floquet multipliers are diophantine,
which is a generic condition for geodesic flows by the Bumpy Metric Theorem.  

The Kupka-Smale condition has been chosen in order to have a unified approach using 
the results from~\cite{CM2}, \cite{OC2}, \cite{OC1}. But the transversality condition~\eqref{ks2}
can be relaxed to asking~\eqref{ks2} only for periodic orbits of small period, in order to 
obtain a Birkhoff section~\cite{CM2}; and a no heteroclinic connections\footnote{A heteroclinic 
connection is the case in which two components of $W^s(\dga)\setminus\dga$ and $W^u(\deta)\setminus\deta$
are equal.}
 condition instead of the transversality~\eqref{ks2}.
 Moreover, since the theorems that we use from~\cite{OC1} on homoclinic points are about fixed points;
 in order to get an homoclinic orbit for an orbit $\dga$ we only need to ask for such generic conditions 
 on periodic orbits of smaller period than $\dga$. We shall not pursue such refinements here.

For an elliptic periodic orbit with Floquet multipliers $\si$ satisfying
$\si^k\ne 1$ for
$1\le k\le 4$, its Poincar\'e map on a 
local transversal section can be written in Birkhoff normal form as
$$
P(z) = z e^{i(\om+\be |z|^2)}+R(z),
$$
with $\om,\be\in\re$ and $R(z)$ with zero 4-jet at $z=0$.
The condition $\be\ne 0$ is residual for 4-jets of $P$.
By theorem~2.5 in~\cite{CP2} this condition on all elliptic orbits 
is residual for $C^r$  riemannian metrics with the $C^r$ topology, $r\ge 5$.
The condition $\be\ne 0$ implies that the Poincar\'e map $P$ is locally
a twist map. Kupka-Smale twist maps have hyperbolic minimizing orbits
with homoclinics for every rational  rotation number in an interval  $[\om,\om+\e[$
or $]\om-\e,\om]$, depending on the sign of $\be$, see \cite{LC2}, \cite{Zehnder1},
\cite{genecand}.
These periodic orbits accumulate on the fixed point $z=0$.
Therefore for $C^r$, $r\ge 5$, generic riemannian metrics on closed surfaces every
closed geodesic is accumulated by homoclinic orbits,
and the closure of the periodic orbits is the same as the closure of 
the homoclinic orbits.

The proof of theorem~\ref{TA} needs 
results in 
dynamics of area preserving maps, Reeb flows and geodesic flows.

A {\it contact 3-manifold} is a pair $(N,\la)$ where $N$ is a closed 3-manifold
and $\la$ is a 1-form in $N$ such that $\la\wedge d\la$ is a volume form.
The {\it Reeb vector field} $X$ of $(N,\la)$ is defined by
$i_X d\la\equiv 0$ and $\la(X)\equiv 1$.
The {\it Reeb flow} $\psi_t:N\to N$ of $(N,\la)$ is the flow of $X$. 
The {\it Liouville form} of a riemannian surface $(M,\rho)$, given by
$$
\la(v)(\xi) :=\rho(v,d\pi(\xi)), \qquad v\in TM, \quad \xi\in T_vTM,
$$
is a contact form on $SM$.
The Reeb flow of $(SM,\la)$ is the geodesic flow of $(M,\rho)$.

A {\it surface of section} for the Reeb flow $\psi_t$ is a compact immersed surface
with boundary $\Si\subset N$, whose interior is embedded and transversal
to the Reeb vector field and whose boundary is a cover of a finite union 
of closed orbits of $\psi_t$.

A {\it Birkhoff section} is a connected inmersed surface $\Si\subset N$
whose interior is embedded and  transversal to the vector field. Its boundary
is a cover of finitely many closed orbits and there is $\ell>0$ such that 
for all $z\in N$, $\psi_{]0,\ell[}(z)\cap \Si\ne \emptyset$ and 
$\psi_{]-\ell,0[}(z)\cap \Si\ne \emptyset$.

Contreras and  Mazzucchelli proved in~\cite[Thm. A]{CM2} that every Kupka-Smale
Reeb flow on a closed contact 3-manifold $(N,\la)$ has a Birkhoff section. The first return map
of the interior of a Birkhoff section is a diffeomorphism which preserves 
the area form $d\la$. 

In order to use the results in \cite{OC2}, \cite{OC1} to obtain homoclinic orbits
we need to have area preserving maps defined on surfaces of genus 0 or 1.
In higher genus, the time one map of an area preserving flow without 
heteroclinic connections is an example of a Kupka-Smale map without 
homoclinics.

In general we don't know the genus of the Birkhoff sections obtained in \cite{CM2} or 
\cite{CKMS}. Instead we use a {\it complete system} of surfaces of section
with genus 0 or 1, (definition~\ref{dcsss}).
 This is a finite collection of surfaces of section which intersect
every orbit and such that  the points which do not return to the collection of surfaces are 
in the stable or unstable manifold of a finite set of hyperbolic closed orbits $\Kfix$,
called {\it non rotating boundary orbits}, 
which are some of the  boundaries of the surfaces of section of the system.
The other closed orbits in the boundaries of the sections
 are called {\it rotating boundary orbits}, their union is denoted $\Krot$.
 They have the property that there is a neighborhood of $\Krot$ where the 
 return times to the system of sections is uniformly bounded.

 If $\ga$ is a hyperbolic orbit of a Reeb flow in a 3-manifold we call {\it separatrices}
the connected components of $W^s(\ga)\setminus \ga$ and of $W^u(\ga)\setminus\ga$.
Since the contact manifold $(N,\la)$ is orientable,
they separate any small tubular neighborhood $U$ of $\ga$ into 2 or 4 connected 
components.
The germs of these components obtained by shrinking 
$U$ are called {\it sectors} of $\ga$. We say that a separatrix {\it accumulates on 
a sector} if it intersects such sector for any tubular neighborhood $U$.
A separatrix  is {\it adjacent to a sector} if both the closure of the sector 
and the separatrix contain a component of
a local invariant manifold $W^{s,u}_\e(\ga)$.
$$
W^{s,u}_\e(\ga)=\textstyle\bigcup_{z\in\ga}W^{s,u}_\e(z),\qquad
W^{s,u}_{\e}(z)=\{w\in W^{s,u}(z): \forall t >0 \;\; d(\psi_{\pm t}(w),\psi_{\pm t}(z))\le \e\,\}.
$$

A {\it Kupka-Smale} contact manifold  means that its Reeb flow satifies
\eqref{ks1} and \eqref{ks2}.

\begin{MainThm}\quad\label{TB}

Let $(N,\la)$ be a Kupka-Smale closed contact 3-manifold. 
\begin{enumerate}
\item \label{B1}
For any hyperbolic closed orbit $\ga$ of $(N,\la)$, all the connected components
         of 
         \linebreak
         $W^s(\ga)\setminus\ga$ and $W^u(\ga)\setminus\ga$ have the same closure
         equal to $\ov W^s(\ga)=\ov W^u(\ga)$.
         
         Moreover, each separatrix of $\ga$ accumulates on both of its adjacent sectors.
        
\item\label{B2}
 If $(N,\la)$ has a Birkhoff section $\Si$ of genus 0 or 1, then every hyperbolic orbit intersecting the
 interior of $\Si$ has homoclinics in all its seperatrices.
 
           A hyperbolic boundary orbit in $\partial \Si$ has homoclinics in all its separatrices
          provided that $\Si$ has genus 0 or if $\Si$ has genus 1 and the union of 
           its local separatrices 
          intersect $\Si$ in at least 4 curves.

\item\label{B3}
  Suppose that  $(N,\la)$ admits a complete system of surfaces of section.
  Then: 
  
  Every non rotating boundary orbit in $\Kfix$ has homoclinics in all its separatrices.
  
  If the system contains a component $S$ of genus $0$ or $1$, then 
  every periodic orbit which intersects the interior of $S$ 
  has homoclinics in all its separatrices.
          
          A hyperbolic rotating boundary orbit in $\Krot\cap\partial S$ 
          has homoclinics in all its separatrices
          provided that $S$ has genus 0 or if $S$ has genus 1 and the
           union of its local separatrices 
          intersect $S$ in at least 4 curves. 
           \end{enumerate}
\end{MainThm}
See also proposition~\ref{PQW} which has no genus restriction.

Observe that the condition of four intersections is satisfied if the hyperbolic boundary orbit
has positive Floquet multipliers. Because in that case the separatrices divide a tubular 
neighborhood of the orbit into four sectors and the trace of the Birkhoff section must turn
around the four sectors.

Recall that Hofer, Wysocki and Zehnder prove in~\cite{HWZ1} corollary~1.8, 
that any 
non degenerate tight contact form on the 3-sphere $\SS^3$  
admits a finite energy foliation whose leaves have genus 0.
The rigid surfaces of the finite energy foliation form a complete system
of surfaces of section.
We check in \S\ref{transvb} that the transversality condition in item~\eqref{css3}
of definition~\ref{dcsss} holds. Therefore we get

\begin{Corollary}\label{Ctight}\quad

 Any Kupka-Smale tight contact form on $\SS^3$ has
homoclinic orbits in all branches of all of its hyperbolic closed orbits.
\end{Corollary}

Since a homoclinic orbit implies the existence of a horseshoe we also obtain

\begin{Corollary}\quad

If a Kupka-Smale tight contact form on $\SS^3$ contains a hyperbolic periodic orbit
then it has infinitely many periodic orbits.
\end{Corollary}

The geodesic flow is the Reeb flow of the Liouville form in the unit tangent bundle.
 By lifting the geodesic flow to a double covering if necessary, in order to obtain homoclinic 
 orbits for geodesic flows it is enough to consider orientable surfaces. 
 Theorem~\ref{TA} follows from theorem~\ref{TB} and the following 
 theorem~\ref{TC} once the conditions on the rotating boundary orbits in $\Krot$
 in item~\eqref{B3} of theorem~\ref{TB} are checked.
 
 \pagebreak
 
\begin{Theorem}[Contreras, Knieper, Mazzucchelli, 
Schulz~{\cite[Thm. E]{CKMS}}]\quad\label{TC}

Let $(M,\rho)$ be a closed connected orientable surface all of  whose
simple contractible closed orbits without conjugate points are non degenerate.
Then there is a complete system of surfaces of section for the 
geodesic flow of $(M,\rho)$ whose components have genus 0 or 1.
\end{Theorem}

The ideas in theorem~\ref{TC} date back to Birkhoff~\cite{Birk3} section~28, 
together with the modern version of the curve shortening lemma by 
Grayson~\cite{Grayson}.
We also provide a proof theorem~\ref{TC} in section~\S\ref{sbb}, theorem~\ref{TCBA}
with a different construction.
And we check the conditions on the rotating boundary orbits
in item~\eqref{B3} of theorem~\ref{TB}.
In our case the system has two embedded surfaces of section of genus 1
and finitely many Birkhoff annuli of disjoint simple closed geodesics.
This proves theorem~\ref{TA}.

For area preserving maps the auto accumulation of invariant manifolds as
in item~\eqref{B1} of theorem~\ref{TB} usually  requires the Kupka-Smale 
condition and also the condition that elliptic periodic orbits are Moser stable.
This is a fundamental step to obtain homoclinics. 
Instead, using our results in~\cite{OC1}, we only use the non-degeneracy
condition~\eqref{ks1} from our Kupka-Smale definition.
In our application the first return map to the complete system of sections
is not globally defined. Special care has been taken in \cite{OC2}, \cite{OC1}
to deal with this case.

In section~\ref{ss1} we prove theorem~\ref{TB} using our results in 
area preserving maps from~\cite{OC2}, \cite{OC1}.
In section~\ref{bdrymap} we show that the return map in a neighborhood
of $\Krot$ extends to the boundary and that the extension of hyperbolic  
rotating boundary orbits give rise to saddle periodic orbits for the return map.
In section~\ref{sbb} we give a proof of theorem~\ref{TC}
 adapted to our application.

\section{Proof of Theorem B.}

\subsection{Auto-accumulation of invariant manifolds.}\label{ss1}

Proof of item~\eqref{B1}.

\begin{Theorem}[Contreras, Mazzuchelli \cite{CM2} Thm. A]\quad\label{CM2}

Any closed contact 3-manifold satisfying the Kupka-Smale 
condition admits a Birkhoff section for its Reeb flow.
\end{Theorem}

\begin{Definition}\label{dhyp}
Let $S$ be a compact  orientable surface with boundary. 
Suppose that $f:S\to S$ is a orientation preserving homeomorphism.
We say that a periodic point $x\in\Fix(f^n)$ is {\it hyperbolic} or of {\it saddle type} if there is an open
neighborhood $V$ and a local chart $h:V\to W$ such that $W=]-1,1[^2$, if $p\in\interior( S)$
or $W=]-1,1[\times [0,1[$, if $p\in\partial S$, $h(p)=(0,0)$ and $h\circ f\circ h^{-1}=g$, where
$g(x,y)=(\la x, \la^{-1} y)$ with $\la\in\re$, $\la\notin\{-1,0,1\}$.
\end{Definition}

 In such coordinates the set $\{(x,y)\in V\,|\, x\ne 0\text{ and }y\ne 0\,\}$
  has two or  four connected components that contain $p$ in their closures.
  We call them {\it sectors of $p$}.
  If $\Si$ is one of these sectors and $\Si'$ is a sector of $p$
   defined by means of another coordinate neighborhood $V'$ of $p$ then
   either $\Si\cap\Si'=\emptyset$
   or $\Si$ and $\Si'$ define the same germ at $p$.
   We say that the set $A$ contains a sector $\Si$ if $A$ contains a set $\Si'$
   germ equivalent to $\Si$ at $p$.
   We say that a set $B$ {\it accumulates on a sector $\Si$ of $p$} 
   if the closure of $B\cap \Si$ contains $p$.
   These definitions do not depend on the choice of $V$
   neither on the choice of the linear map $(x,y)\mapsto (\la x,\la^{-1}y)$.

The stable and unstable manifolds of $p$ are
\begin{align*}
W^s(p,f)&=\{q\in S:\lim_{m\to+\infty}d(f^m(p),f^m(q))=0\,\},
\\
W^u(p,f)&=\{q\in S:\lim_{m\to-\infty}d(f^m(p),f^m(q))=0\,\}.
\end{align*}
The {\it branches} of $p$ are the connected components of $W^s(p,f)\setminus\{p\}$ or of $W^u(p,f)\setminus\{p\}$.
A {\it connection} between two periodic points $p,\,q\in S$ is a branch of $p$ which is also a branch of $q$,
i.e. a whole branch which is contained in $W^s(p,f)\cap W^u(q,f)$ or in $W^u(p,f)\cap W^s(p,f)$.
 We say that a branch $L$ and a sector $\Si$ are {\it adjacent}
   if a local branch of $L$ is contained in the closure of $\Si$ in $S$.
   Two branches are {\it adjacent} if they are adjacent to a single sector.
 
A periodic point $p\in\Fix(f^n)$ is {\it irrationally elliptic} if 
$f$ is $C^1$ in a neighborhood of $p$ and no eigenvalue of $df^n(p)$ is a root of unity.

Let $(N,\la)$ be a Kupka-Smale closed contact 3-manifold
and $\Si$ a Birkhoff section for its Reeb flow $\psi$. 
The first return times $\tau_\pm:\Si\to\re^+$ and 
the first return maps $f^{\pm 1}:\Si\to\Si$ to $\Si$ are defined by
\begin{align*}
\tau_\pm(x) :=\pm \min\{t>0:\psi_{\pm t}(x)\in \Si\},\qquad 
f^{\pm 1}(x) :=\psi_{\tau_\pm(x)}(x).
\end{align*}
We have that $f^{\pm 1}$ are smooth diffeomorphisms of $\Si$  
preserving the area form $d\la$ on $\Si$.
We are going to apply to $f$ the following Theorem:

\begin{Theorem}[Oliveira, Contreras~\cite{OC1} corollary~4.9]\label{OC1}
\quad

Let $S$ be a compact connected orientable surface with boundary 
provided with a finite measure $\mu$ 
which is positive on open sets and 
$f:S\to S$ be an orientation preserving and  area preserving
homeomorphism of $S$.
\begin{enumerate}
\item\label{c11}
Suppose that $L$ is a (periodic) 
branch of $f$ and that all periodic points of $f$ contained in $cl_SL$ 
are of saddle type or irrationally elliptic.
Then either $L$ is a connection or $L$ accumulates on
both adjacent sectors.
In the later alternative $L\subset \om(L)$.

\item\label{c12}
Let $p\in S-\partial S$ be a periodic point of $f$ of saddle type and let
$L_1$ and $L_2$ be adjacent branches of $p$ that are not connections.
If all the periodic points of $f$ contained in $cl_S(L_1\cup L_2)$ are of saddle
type or irrationally elliptic, then $cl_SL_1=cl_SL_2$.

\item\label{c13}
Suppose that $p\in S-\partial S$ is a periodic point of $f$ of saddle type.
Assume that all the periodic points contained in $cl_S(W^u_p\cup W^s_p)$ 
are of saddle type or irrationally elliptic and $p$ has no connections.
Then the branches of $p$ have the same closure
and each branch of $p$ accumulates on all the sectors of $p$.

If in addition $S$ has genus 0 or 1, then the four 
branches of $p$ have homoclinic points.

\item\label{c14}
Let $C$ be a connected component of $\partial S$ and suppose that   
all the periodic points $p_1,\ldots,p_{2n}$ of $f$ in $C$ are of saddle type.
Let $L_i$ be the branch of $p_i$ contained  in $S-\partial S$.
Assume that for every $i$ all the periodic points of $f$ contained 
in $cl_SL_i$ are of saddle type or irrationally elliptic and that $L_i$ is not a connection.
Then for every pair $(i,j)$  the branch $L_i$ accumulates on all the sectors  of $p_j$ and
 $cl_SL_i=cl_SL_j$.

If in addition $S$ has genus 0 then any pair $(L_i,L_j)$
of stable and unstable branches intersect. The same happens 
if the genus of $S$ is 1 provided that there are at least 4 periodic 
points in $C$.
\end{enumerate}
\end{Theorem}

Item~\eqref{c12} of theorem~\ref{OC1} for closed manifolds without boundary and under
the further hypothesis that the elliptic periodic points are Moser stable appears in 
Mather~\cite{Mat9} \linebreak 
theorem~5.2.
It also appears in Franks, Le Calvez~\cite{FLC} theorem~6.2 for $S=\SS^2$, 
the 2-sphere, when the elliptic points are Moser stable. 
The proof of items~\eqref{c13} and~\eqref{c14}  on the existence of homoclinic orbits
using item~\eqref{c12}, appears in the proof of theorem~4.4 of \cite{OC1} and can be 
read independently of the rest of the paper.

In section~\S\ref{bdrymap} 
we prove that if $S$ is a Birkhoff section for the Reeb flow $\psi_t$
of $(N,\la)$, then there is a continuous extension of the return map $f:S\to S$ to the 
boundary $\partial S$ which preserves its boundary components as in figure~\ref{boundary}.
If $\ga$ is a boundary component of $\partial S$ which is an irrationally elliptic closed orbit
then the restriction $f|_\ga$ has no periodic points. If $\ga$ is a hyperbolic closed orbit then
the extension $f|_\ga$ has periodic points which are the limits in $\ga$ of the intersections
$W^s(\ga)\cap S$ and $W^u(\ga)\cap S$. The extension $f|_\ga$ corresponds to the action
of the derivative of the flow $d\psi_t$ on the projective space of the contact structure $\xi$, transversal to
the vector field. Therefore the limits of the intersections $W^s(\ga)\cap S$ are  sources in $f|_\ga$
and the limits of the intersections $W^u(\ga)\cap S$ are sinks in $f|_\ga$. The other points in $\ga$
are  connections among these sources and sinks, i.e. stable manifolds of sinks which
coincide with unstable manifolds of sources inside $\ga$. These periodic points in $\ga$ are 
saddles for $f$ in $\ov S$. The sinks in $\ga$ have an unstable manifold in $S$ which is a 
connected component of $W^u(\ga)\cap S$. Similarly, the sources for $f|_\ga$ are saddles in $S$
with stable manifold a connected component of $W^s(\ga)\cap S$. The Kupka-Smale condition
for the Reeb flow $\psi_t$ implies that the branches in $S -\partial S$ of periodic points in $\partial S$
are not connections. In fact their intersections with other branches of periodic points of $f$ are transversal.

Therefore we can apply the first part of items~\eqref{c13} and~\eqref{c14} of theorem~\ref{OC1} to the
return map $f$ of a Birkhoff section for the Reeb flow $\psi_t$.
This implies item~\eqref{B1} of theorem~\ref{TB}. For periodic points which 
intersect the interior of $S$ we use item~\eqref{c13}  of theorem~\ref{OC1}
and for hyperbolic periodic 
orbits in $\partial S$ we use item~{\eqref{c14}} of theorem~\ref{OC1}.

\subsection{Homoclinics for Birkhoff sections.}\label{pi2B}
Proof of item~\eqref{B2}.

We saw in the proof of item~\eqref{B1} of theorem~\ref{TB}  in subsection \S\ref{ss1}
that we can apply theorem~\ref{OC1} to the first return map of a Birkhoff section for the
Kupka-Smale Reeb flow.
In the case that the Reeb flow admits a Birkhoff section with genus zero or one, we can 
also apply the second part of items~\eqref{c13} and~\eqref{c14} of theorem~\ref{OC1}.
This gives homoclinic orbits in every separatrix of all the hyperbolic closed orbits for the
Reeb flows $\psi_t$ selected in item~\eqref{B2} of theorem~\ref{TB}.

\subsection{Complete system of surfaces of section.}\label{ss3}
\quad

Let $(N,\la)$ be a compact contact 3-manifold
and $\psi_t$ its Reeb flow.
For $Z\subset N$ define the {\it forward trapped set} $\trap_+(Z)$ 
and the {\it backward trapped set} $\trap_-(Z)$ as
$$
\trap_{\pm}(Z)=\{\, z\in N\,:\; \exists \tau\quad  \forall t>\tau\quad \phi_{\pm t}(z)\in Z \,\}.
$$

    \begin{figure}
   \includegraphics[scale=.35]{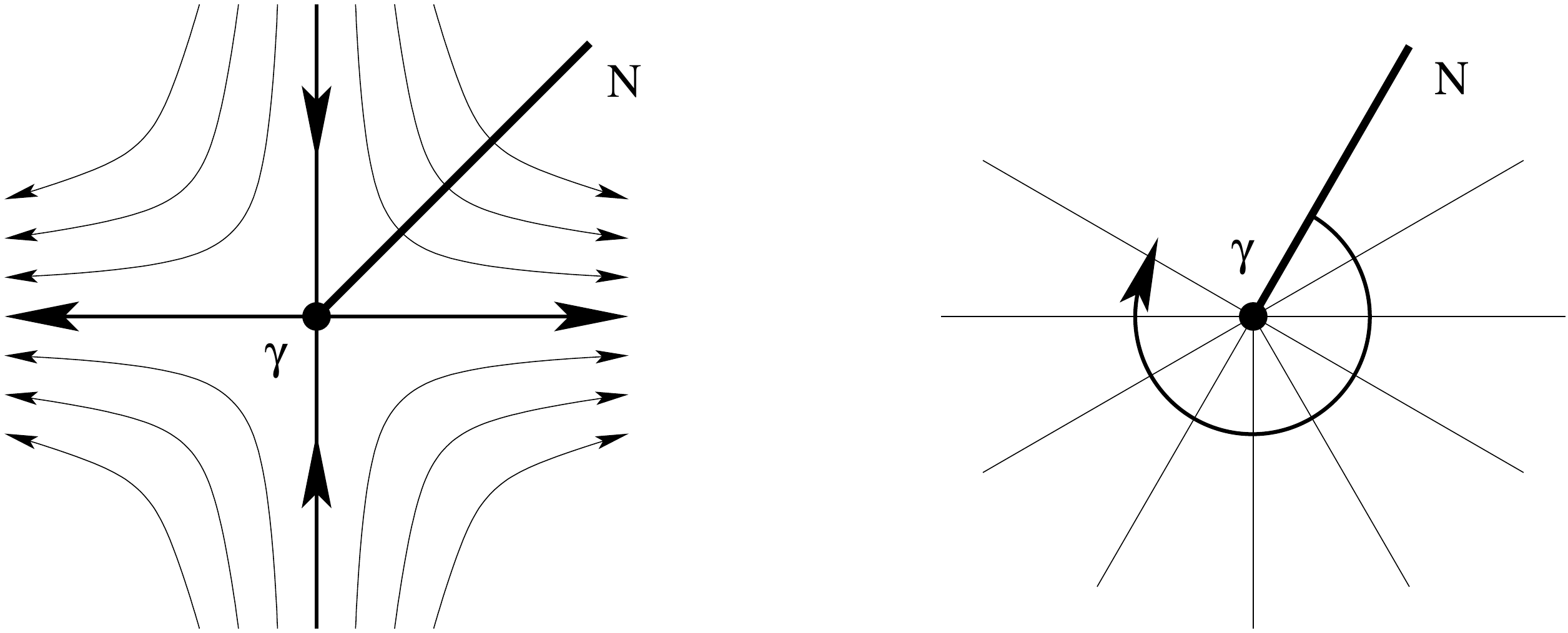}
   \caption{\footnotesize This figure shows in the left how a neighborhood $N$ of a boundary
   closed orbit $\ga\in \Kfix$ arrives to $\ga$ in a local transversal section to the flow.
   The figure in the right shows a neighborhood $N$ of a rotating boundary orbit $\ga\in \Krot$ 
   and some of the interates of $N$ under the Reeb flow.}
   \label{fixed}
   \end{figure}

   \begin{Definition}\label{dcsss}\quad
   
We say that $(\Si_1,\ldots,\Si_n)$ is a complete system 
of surfaces of section for $(N,\la)$ if
\begin{enumerate}[(i)]
\item\label{css1} Each $\Si_i$ is a connected surface of section for $(N,\la)$, i.e.
$\Si_i\subset N$ is a connected  immersed compact surface whose interior $\interior(\Si_i)$ 
is embedded and transversal to the Reeb vector field $X$ and its boundary is a 
cover of a finite collection
of closed orbits of $\psi$.

\item\label{css2} Separate the boundary orbits  in two sets\footnote{This classification is the same as
{\it radial} and {\it broken} binding orbits for broken book decompositions.}
 $\cup_i \partial\Si_i=\Krot\cup \Kfix$. The {\it non rotating} periodic orbits in $\Kfix$
 are hyperbolic and
 have a neighborhood  $N$ in $\Si_i$ 
 which arrives to the boundary inside a sector as in figure~\ref{fixed}.
 For the {\it rotating} boundary orbits\footnote{Rotating boundary orbits can be hyperbolic or elliptic.}
  in $\Krot$
  there is $\ell>0$ such that each $\ga\in \Krot$ has a neighborhood $N_\ga$ in $\Si_i$
  such that  
  $$
  \forall z\in N_\ga \quad \psi_{]0,\ell[}(z)\cap \Si_i\ne \emptyset
  \quad\&\quad \psi_{]-\ell,0[}(z)\cap \Si_i\ne \emptyset.
  $$
  \item\label{css3} At each\footnote{This condition says that the flow rotates more than the surface of section when it approaches its boundary orbit $\gamma$.} rotating boundary orbit $\ga\in \Krot\cap \Si_i$ the extension of $\Si_i$
to the unit normal bundle $\cN(\ga)$ of $\ga$ by 
 blowing up a neighborhood of $\ga$ using polar coordinates, 
  is an embedded collection of closed curves 
transversal to the extension of the Reeb vector field to $B_\ga$.

  \item\label{css4} Every orbit intersects $\bSi=\cup_i\Si_i$.
  \item\label{css5} $\trap_{\pm}(N\setminus\bSi) \subset W^{s,u}(\Kfix)$.
\end{enumerate}
\end{Definition}

Recall that a  {\it Birkhoff section} is a connected embedded surface $\Si\subset N$
whose interior is transversal to the vector field. Its boundary
is a cover of finitely many closed orbits and there is $\ell>0$ such that 
for all $z\in N$, $\psi_{]0,\ell[}(z)\cap \Si\ne \emptyset$ and
$\psi_{]-\ell,0[}(z)\cap \Si\ne \emptyset$.
We use the same notation $\Kfix$, $\Krot$, $\partial K_i$, $\partial\bSi$ for
a collection of periodic orbits or their union. 
Here
$\partial\bSi := \cup_i \;\partial \Si_i$ and also
$\interior(\bSi): = \cup_i \interior \Si_i$.

\begin{Lemma}\label{larrow}\quad

Let $(\Si_i)_{i=1}^n$ be a complete system of surfaces of section for $(N,\la)$.

Let $\ga\in \Kfix$ and  a connected component $L\subset W^{s,u}(\ga)\setminus \ga$, then 
$$
\exists \xi\in \Kfix \qquad L\cap W^{u,s}(\xi)\ne \emptyset.
$$
\end{Lemma}
\begin{proof}\quad

We prove it only for $L\subset W^u(\ga)$, the other case is similar.
Suppose by contradiction that 
\begin{equation}\label{LWs}
L\cap W^s(\Kfix)=\emptyset.
\end{equation}
Let $z\in L$. 
Let $S$ be an essential smooth embedded circle in $L$.
By \eqref{LWs} and \eqref{css5},
the first return time $\tau:S\to\re$ 
$$
\tau(x):=\inf\{\, t>0: \psi_t(x) \in\bSi\,\}
$$
is well defined and finite on $S$.
The return map $f:S\to\interior (\bSi)$, $f(x)=\psi_{\tau(x)}(x)$ is an immersion.
Since $S\subset L$ is connected, compact and disjoint from periodic orbits, 
there is a component $\Si_{i_0}$ of $\bSi$ such that 
$f:S\to f(S)\subset \Si_{i_0}$ is a diffeomorphism. By the intrinsic dynamics of $\psi_t$ on $L$,
we have that  $f(S)$ is an essential  smooth
embedded circle in $L$. Repeating this argument there is a component $\Si_{i_1}$ of $\bSi$
and an infinite collection $\{S_k\}_{k\in\na}$ of disjoint essential smooth embedded
circles in $L$ given by $S_k=f_k(S)$, where $f_k$ is the $k$-th return of $S$ to $\Si_{i_1}$.
Observe that for $i<j$, the circles $S_i$, $S_j$ bound an embedded annulus $A_{ij}$ in $L$.

The circles $S_k$ are disjoint and embedded in $\Si_{i_1}$. 
By Lemma 3.2 or Theorem~3.3 in~\cite{JMM}, there is a free homotopy
class in $\Si_{i_1}$ which contains infinitely many of them $\{S_{k_n}\}_{\in\na}$.

If the circles $S_{n_k}$ are contractible in $\Si_{i_1}$, they bound disjoint disks $D_{k_n}$ with area
$$
\int_{D_{k_n}}d\la=\int_{S_{k_n}}\la=\int_{A(k_n)}d\la +\int_{m\cdot\ga} \la 
= m\cdot\text{period(}\ga)>0,
$$
where $A(k_n)$ is the annulus on $L$ with boundaries $\{S_{k_n},\ga\}$ and
$m\in\{1,2\}$ wether $\partial L$ covers $\ga$ $m$-times.
We have used that $\la(X)\equiv 1$ and that $d\la|_L\equiv 0$ because the Reeb vector field is tangent to $L$.
This contradicts the fact that the area of $\Si_{i_1}$ is finite, because
$$
\text{area}(\Si_{i_1})=\int_{\Si_{i_1}}d\la =\tsum_{\ga\in \partial\Si_{i_1}}m_\ga\int_\ga\la<+\infty,
$$
where $\partial\Si_{i_1}$ is finite and $m_\ga$ is the covering number of $\partial\Si_{i_1}$ over $\ga$.

If the homotopy class of the $S_{k_n}$ in $\Si_{i_1}$ is non trivial then $S_{k_1}$ and $S_{k_2}$
bound an annulus $B_{12}$ in $\Si_{i_1}$.
The annulus $B_{12}$ has positive $d\la$-area because the transversality 
of $\Si_{i_1}$ to the Reeb vector field implies that $d\la$ is non-degenerate on $\Si_{i_1}$.
They also bound the annulus $A_{k_1 k_2}$ in $L$ with zero $d\la$-area, because $d\la\vert_L\equiv 0$. 
Therefore
$$
0=\int_{A_{k_1k_2}}d\la =\int_{S_{k_1}}\la-\int_{S_{k_2}}\la
=\int_{B_{12}}d\la >0.
$$
A contradiction.

\end{proof}
\begin{Remark}
 Using proposition~\ref{lma} instead of proposition~2.1 in \cite{CM1} 
 it is possible to reproduce the proofs of lemma~5.2  and theorem~B 
 in~\cite{CM2} to obtain
 \begin{equation}\label{fixus}
 \forall \ga\in \Kfix \qquad  \ov{W^s(\ga)}=\ov{W^u(\ga)},
 \end{equation}
 whenever $(N,\la)$ is Kupka-Smale and has a complete systems of surfaces of section.
 Then proposition~\ref{p17} and section~3 in~\cite{CM2} give a Birkhoff section for 
 $(N,\la)$ starting from a complete system instead of a broken book decomposition.
 
 Here we will use theorem~\ref{TB}.\eqref{B1},
 proved in subsection~\S\ref{ss1}, to get \eqref{fixus}.
 \end{Remark}

\begin{Lemma}\label{L15}\quad

Let $\a$, $\be$ be hyperbolic periodic orbits of a Kupka-Smale 
Reeb flow of a closed contact 3-manifold $(N,\la)$.
Suppose that $\be$ has homoclinic orbits.
Let $Q$ be a separatrix of $\a$.
Suppose that 
\begin{equation}\label{l150}
cl (W^s(\be)\cup W^u(\be))\subset \ov Q.
\end{equation}
Then 
\begin{align}
Q\subset W^u(\a) \quad\then \quad Q\cap W^s(\be)\ne \emptyset,
\label{l151}\\
Q\subset W^s(\a) \quad\then \quad Q\cap W^u(\be)\ne \emptyset.
\label{l152}
\end{align}
Moreover, all the separatrices of $\a$ have homoclinics.
\end{Lemma}
\begin{proof}

Let $D$ be a small disk transversal to the Reeb flow containing a point $p\in\be\cap D$
and such that 
\begin{equation}\label{dfd}
D\cap \a=\emptyset \quad\text{ and }\quad
\int_D d\la < \int_\a \la.
\end{equation}
 The Kupka-Smale condition implies that the homoclinic intersections
in $W^s(\be)\cap W^u(\be)$ are transversal. 
By the $\la$-lemma there are segments of $W^u(\be)\cap D$ (resp. $W^s(\be)\cap D$) 
accumulating in the $C^1$ topology on the whole local component 
$W^u_\e(\be)\cap D$ (resp. 
$W^s_\e(\be)\cap D$).
These segments form a grid in $D$ nearby $p$ which contains rectangles of arbitrarily small 
diameter. Choose  small rectangles $A$, $B$ with boundaries in 
$W^s(\be)\cup W^u(\be)$ 
such that 
$$
cl(A)\subset int(B)\subset cl(B)\subset int(D).
$$
Since $cl(W^s(\be)\cup W^u(\be))\subset\ov Q$, we have that $Q\cap D$ accumulates 
on the boundary $\partial A$.
Then there is a point  $q\in Q\cap int(B)$. Let $J$ be the connected component
of $Q\cap D$ containing $q$. Since by \eqref{dfd} $D\cap \a=\emptyset$, $J$ is 
either a circle or a curve with endpoints in $\partial D$. 
Suppose first that $J$ is a circle. Since $J$ is transversal 
to the flow inside $Q$ and there are no periodic orbits in $Q$,
by Poincar\'e-Bendixon theorem,  $J$ must be an essential 
embedded circle in~$Q$.   
Let $A\subset D$ be a disk with $\partial A =J$. Since $Q$ is tangent to the 
Reeb vector field, $d\la|_Q\equiv 0$. By Stokes theorem
\begin{equation}\label{ara}
\int_A d\la =\int_J \la = k\cdot \int_\a \la,
\quad k\in\{1,2\},
\end{equation}
with $k=2$ if $\a$ is negative hyperbolic.
But \eqref{ara} contradicts \eqref{dfd} because $A\subset D$.
Therefore $J$ is a curve with endpoints in $\partial D$.
Since $q\in J\cap int(B)\ne\emptyset$ and $B\cap \partial D=\emptyset$, we
have that $\exists r\in J\cap \partial B\ne \emptyset$.
Then $r\in Q\cap W^{u,s}(\be)$ if $Q\subset W^{s,u}(\a)$.
This proves \eqref{l151} and \eqref{l152}.

Suppose now that $Q\subset W^u(\a)$, the case $Q\subset W^s(\a)$ is similar.
By \eqref{l151}
\begin{equation}\label{qcap}
Q\cap W^s(\be)\ne\emptyset.
\end{equation}
But by \ref{TB}.\eqref{B1}, $\ov W^s(\a) = \ov W^u(\a)=\ov Q$.
Hence $cl(W^s(\be)\cup W^u(\be))\subset \ov Q=\ov W^s(\a)$.
By \eqref{l152} applied to a separatrix in $W^s(\a)$
we have that 
\begin{equation}\label{sguk}
W^u(\be)\cap W^s(\a)\ne\emptyset.
\end{equation}
Since by the Kupka-Smale condition the heteroclinic  intersections are transversal, 
\linebreak
equations
\eqref{qcap}, \eqref{sguk} and the  $\la$-lemma imply that 
$Q\cap W^s(\a)\ne \emptyset$.
Thus the separatrix $Q$ has homoclinics.
Now observe that by \ref{TB}.\eqref{B1} the condition~\eqref{l150}
is satisfied by all the separatrices of $\a$.

\end{proof}

\begin{Proposition}\label{p17}
\quad

\hskip 1cm 
Let $(N,\la)$ be a closed contact 3-manifold satisfying the Kupka-Smale condition.
\linebreak
Let $(\Si_1,\ldots,\Si_m)$ be a complete system of surfaces of section
 with boundary components $K=K_\text{rot}\cup\Kfix$.
Then every component of $W^{s,u}(\ga)\setminus \ga$ of every non rotating boundary
orbit $\ga\in \Kfix$
has homoclinics and $\ov{W^s(\ga)}=\ov{W^u(\ga)}$.

\end{Proposition}

\pagebreak

\begin{proof}\quad

Write 
$(\ka_1,\ldots,\ka_n)\in \Ga$
if  $\forall i$ $\ka_i\in \Kfix$ and
$W^u(\ka_i)\cap W^s(\ka_{i+1})\ne\emptyset$ for $1\le i<n$.
The definition of $\Ga$ implies that 
\begin{equation}\label{pGa1}
(\ka_1,\ka_2)\in \Ga, \quad (\ka_2,\ka_3)\in \Ga
\quad \then \quad (\ka_1,\ka_2,\ka_3)\in\Ga.
\end{equation}
The $\la$-lemma implies that 
\begin{equation}\label{pGa2}
(\ka_1,\ldots,\ka_n)\in \Ga \quad \then \quad (\ka_1,\ka_n)\in \Ga.
\end{equation}

By lemma~\ref{larrow} 
\begin{equation}\label{pGa3}
\forall \be\in \Kfix \quad \exists \a,\ga\in \Kfix
\qquad
\{(\a,\be), (\be,\ga)\}\subset \Ga.
\end{equation}

By \eqref{pGa3} for any $\a\in \Kfix$ there is an infinite sequence
$(\a,\ka_1,\ka_2,\ldots)\in \Ga$.
Since $\Kfix$ is finite, there are $n\ne m$ such that $k_n=k_m=:\be$.
By properties~\eqref{pGa1} and \eqref{pGa2}, $(\a,\be,\be)\in\Ga$.
Thus
\begin{equation}\label{eab}
\forall \a\in \Kfix \quad \exists \be\in \Kfix \qquad (\a,\be,\be)\in \Ga.
\end{equation}

Let $\a\in \Kfix$ and let $Q$ be a component of $W^{s,u}(\a)\setminus\a$.
By theorem~\ref{TB}.\eqref{B1}, proved in~\S\ref{ss1}, $\ov Q=\ov{W^s(\a)}=\ov{W^u(\a)}$.

Let $\be\in\Kfix$ be given by \eqref{eab}.
Since the intersection $W^u(\a)\cap W^s(\be)$ is transversal, 
by the $\la$-lemma $W^u(\be)\subset \ov{W^u(\a)}$.
Then by theorem~\ref{TB}.\eqref{B1},
$$
cl(W^s(\be)\cup W^u(\be))=\ov{W^u(\be)}\subset \ov{W^u(\a)}=\ov Q.
$$
Then lemma~\ref{L15} implies that $Q$ has homoclinics.

\end{proof}

\subsection{Homoclinics for complete systems.}\label{pi3B}
Proof of item~\ref{TB}.\eqref{B3}.

We shall use the following

\begin{Theorem}[The accumulation lemma]\label{ACL}\quad

 Let $S$ be a  connected surface with compact boundary provided with 
a Borel measure $\mu$ such that open non-empty subsets have positive
measure and compact subsets have finite measure.
Let $S_0\subset S$ be an open subset with   $fr_S S_0$  compact.

Let $f,f^{-1}:S_0\to S$ be an area preserving homeomorphism of $S_0$
onto  open subsets $f(S_0)$, $f^{-1}(S_0)$ of $S$.
Let $K\subset S_0$ be a compact connected invariant subset of $S_0$.

If $L\subset S_0$ is a branch of $f$ and $L\cap K\ne \emptyset$,
then $L\subset K$.
\end{Theorem}

This version of theorem~\ref{ACL} is proved in \cite[Thm.~4.3]{OC2},
its proof also applies to branches $L$ of saddle points in the boundary $\partial S_0$.
Theorem~\ref{ACL} was originally proved in Mather~\cite[corollary~8.3]{Mat9} 
for surfaces without boundary and global maps ($S_0=S$).
It is also proved in Franks, Le Calvez~\cite[lemma~6.1]{FLC}
for $S=S_0=\SS^2$, the 2-sphere.
This version is needed to prove theorem~\ref{toc1} in ~\cite{OC1}.
In proposition~\ref{PQW} we only use its global version $S_0=S$,
but in corollary~\ref{c115} we use this version for partially defined maps.

\begin{Proposition}\quad\label{PQW}

Let $(N,\la)$ be a closed contact 3-manifold satisfying the Kupka-Smale condition
with a given complete system of surfaces of section. 
Let $\Kfix$ be the set of non rotating boundary orbits
let 
$$
\W=(W^s(\Kfix)\cup W^u(\Kfix))\setminus \Kfix.
$$ 
\indent 
Let $\ga$ be a hyperbolic closed orbit of the Reeb flow of $(N,\la)$.
Let $Q$ be a separatrix of $\ga$. 
If 
$\ov Q\cap \W\ne \emptyset$, then
all the separatrices of  $\ga$ have homoclinics. 
\end{Proposition}

\begin{proof}\quad

Assume that $Q\subset W^u(\ga)$, the case $Q\subset W^s(\ga)$ is similar.

By theorem~\ref{CM2} there is a Birkhoff section $\cB$ 
for the Kupka-Smale Reeb flow of $(N,\la)$.
Let $f$ be the first return map of the Reeb flow to $\cB$.

In section~\ref{bdrymap} we show that 
in case $\ga\subset \partial \cB$ is a hyperbolic periodic orbit, 
then the connected components of $Q\cap\cB$ are interior branches of
saddle points in $\ga$ for the return map $f$ to $\cB$.
Choose a connected component $L_p$ of $Q\cap \cB$.
Then $L_p$ is a branch of a periodic point $p\in\cB$ of $f$,
possibly at the boundary $p\in\partial \cB$ if $\ga\subset\partial \cB$.
Let $n$ be the minimal period of $L_p$, $f^n(L_p)=L_p$.
In particular $f^n(p)=p$.
Then $L_p$ is a branch of $W^u(p,f^n)$.
And $K=\ov L_p$ is a compact $f^n$-invariant subset of $\cB$.

Let $\ka\in \Kfix$ be a non rotating boundary
orbit\footnote{In lemma~3.3 in \cite{CM2} an argument of Fried is used to show that, since 
by proposition~\ref{p17}
an  orbit $\ka\in\Kfix$ has 
homoclinics in all its branches, one can obtain a Birkhoff section $\cB$ which intersects $\ka$ in its interior.} and let $q\in\ka\cap\cB$ be a (saddle) periodic point for $f$.
 Choose a multiple $m$ of $n$ such that 
$f^m(q)=q$, then $\{p,q\}\subset\Fix(f^m)$ and $f^m(K)=K$.
We will apply  the  accumulation lemma~\ref{ACL} to  $f^m$  and the compact $f^m$-invariant set $K$.
Let $L_q\subset \interior(\cB)$ be an interior branch of $q$, i.e. a connected component of $W^\tau(q,f^m)\setminus\{q\}$,
$\tau\in\{s,u\}$, which is also a connected component of $W^\tau(\ka)\cap \cB$.

By the accumulation lemma~\ref{ACL},
$$
L_q\cap K\ne\emptyset \qquad\then\qquad L_q\subset K.
$$
This implies in the Reeb flow that
$$
(W^\tau(\ka)\setminus \ka)\cap \ov Q\ne \emptyset \qquad \then \qquad 
(\text{a component of }W^\tau(\ka)\setminus\ka)\subset \ov Q.
$$
By item~\eqref{B1} of theorem~\ref{TB} we have that $\ov W^s(\ka)=\ov W^u(\ka)=\ov W^\tau(\ka)$ is also the
closure of any component  of $W^\tau(\ka)\setminus\ka$,
therefore we get
\begin{equation}\label{chkt}
\exists \tau\in\{s,u\} \qquad
\ov Q \cap (W^\tau(\ka)\setminus \ka)\ne\emptyset \qquad\then\qquad
cl(W^s(\ka)\cup W^u(\ka))\subset \ov Q.
\end{equation}

Suppose that $\ov Q\cap \W\ne\emptyset$.
Then there is $\ka\in \Kfix$ and $\tau\in\{s,u\}$ such that 
\linebreak
$\ov Q\cap (W^\tau(\ka)\setminus\ka)\ne\emptyset$.
By~\eqref{chkt}, 
$$
cl(W^s(\ka)\cup W^u(\ka))\subset \ov Q.
$$ 
By proposition~\ref{p17} we have that $\ka$ has transversal homoclinics.
Then by lemma~\ref{L15}
all the separatrices of $\ga$ have homoclinics.

\end{proof}

Let $S$ be a component of a complete system of surfaces of section
 for a closed contact 3-manifold $(N,\la)$. 
 Define the first return times $\tau_\pm$ to $S$ 
 and the first return maps $f^{\pm 1}$ as
 \begin{gather}
 \tau_+:\interior(S)\to ]0,+\infty[\cup\{+\infty\} 
 \quad\text{ and }\quad
 \tau_-:\interior(S)\to ]-\!\infty,0[\cup\{-\infty\}
 \quad\text{ by}
 \notag
 \\
 \label{etaupm}
 \tau_{\pm}(z) :=
 \pm \inf\{t>0: \phi_{\pm t}(x)\in S\,\}. 
 \\
\label{eff}
 f(x) :=\phi_{\tau_+(x)}(x),\qquad  f^{-1}(x):=\phi_{\tau_-(x)}(x).
 \end{gather}
 By the implicit function theorem $f$ and $f^{-1}$ are defined in the 
 open subsets $[\tau_+<+\infty]$ and $[\tau_->-\infty]$  of $\interior(S)$ 
 respectively.

In section~\ref{boundary} we show that $f$ and $f^{-1}$ extend 
to a neighborhood of $\Krot\cap \partial S$ as in figure~\ref{boundary}. 
All the periodic points for $f^{\pm 1}$ in any 
$\ga\in\Krot\cap\partial S$ are of saddle type,
 their  invariant manifolds for $f^{\pm 1}$ are 
 either the intersections $W^{s,u}(\ga)\cap S$ or heteroclinic 
 connections in $\ga$. Irrationally elliptic orbits in $\partial S$
 are in  $\Krot\cap\partial S$, 
 but they have no periodic orbits for $f^{\pm 1}$.
 And $f^{\pm 1}$ are not 
 defined\footnote{In fact the natural extension of $f$ to a point $x\in\Kfix$ would be
 the whole circle of a first intersection of a  component of $W^u(\ga)\setminus\ga$, $\ga=\psi_\re(x)$, 
 with $S$. See figure~\ref{fixed}.}
 on $\Kfix\cap\partial S$.
 Let 
 \begin{equation}\label{eso1}
 S_0 = ([\tau_+<+\infty]\cap [\tau_->-\infty])\cup(\Krot\cap\partial S).
 \end{equation}
 By condition~\ref{dcsss}.\eqref{css2} the maps $\tau_\pm$ are finite in a neighborhood 
 in  $\interior(S)$ of $\Krot\cap \partial S$. 
 Then $S_0$ is an open submanifold of $S$ with compact boundary
 $\partial S_0\subset \partial S$ 
 and $f^{\pm 1}:S_0\to S$ are differentiable, area preserving
 and $f(\partial S_0)\subset \partial S_0$.

Observe that condition ~\ref{dcsss}.\eqref{css2} implies that 
$\Kfix\subset \cap_{n\in\na}\ov{[|\tau_\pm|>n]}$.
In section~\ref{bdrymap} we see that the functions $\tau_\pm$ 
can be extended to $\Krot$.
So  we use the notation
\begin{gather}
\Kfix\subset [\tau_\pm=\pm\infty],
\qquad
\Krot\subset [|\tau_\pm|<\infty],
\label{tpm2}\\
S_0=[\tau_+<+\infty]\cap[\tau_->-\infty].
\label{eso}
\end{gather}

\pagebreak

\begin{Lemma}\label{Lnti}
\begin{gather}
\forall \e>0\quad \exists N\in\na\quad
x\in\interior(S)\;\;\&\;\;
N\le |\tau_\pm(x)|<\infty
\quad\then\quad
d(x, [|\tau_\pm|=\infty])<\e.
\end{gather}
\end{Lemma}
\begin{proof}\quad

We only prove it for $\tau_+$.
For $n\in\na\cup\{+\infty\}$ let  $A_n:=[\tau_+<n]$.
Then $A_n$ is an increasing family of open sets in the closure $cl(S)$
with $A_\infty=\cup_{n\in\na}A_n$. 
For $\de>0$ let
$B(\partial A_\infty,\de):=\{\, x\in S:d(x,\partial A_\infty)<\de\,\}$.
Observe that $\partial A_\infty =[\tau_+=\infty]$ in $cl(S)$.
Indeed, by Poincar\'e recurrence theorem, $A_\infty$
has total measure in $S$, then $\Kfix\subset \partial A_\infty$.
It is enough to prove that 
\begin{equation}\label{eke}
\forall \e>0 \quad \exists N\in\na\qquad
A_\infty\subset A_N\cup B(\partial A_\infty,\e).
\end{equation}
Let $K_\e:= A_\infty\setminus B(\partial A_\infty,\frac \e2)$.
Then $K_\e$ is compact and $\{A_n\}_{n\in\na}$ is an open cover of $K_\e$.
Since the family $\{A_n\}$ is increasing, there is $N\in\na$
such that $K_\e\subset A_N$.
Then $A_\infty\subset K_\e\cup B(\partial A_\infty,\e)\subset A_N\cup B(\partial A_\infty,\e)$.

\end{proof}

\begin{Proposition}[M. Mazzucchelli]\label{lma}\quad

 Let $N$ be a compact 3-manifold with a flow $\psi_t$.
 Let $\bSi$ be a finite union of connected surfaces of section
 and $K$ a finite collection of
 hyperbolic periodic orbits in $\partial \bSi$.
  Suppose that 
 \begin{enumerate}[\quad(a)]
 \item\label{lma1} Every orbit of $\psi$ intersects $\bSi$.
 \item\label{lma2}   $z\in N\quad\&\quad \psi_{[0,+\infty[}(z)\cap \bSi=\emptyset \qquad\then\qquad z\in W^s(K)$.
 \end{enumerate}
 
 Let $\Si_1$ be a connected component of $\interior( \bSi)$.
 Let $\tau:\Si_1\to ]0,+\infty[\cup\{+\infty\}$ be the first return time to the component $\Si_1$,
 i.e.
$$
 \tau(z):=\inf\{\, t>0\;|\;\psi_t(z)\in \Si_1\}. 
 $$
 Let $\a:[0,1[\to \Si_1$ be continuous and  suppose that 
 \begin{enumerate}[(i)\quad]
 \item $\forall s\in[0,1[\quad \tau(\a(s))<+\infty$.
 \item $\exists \{s_n\}_{n\in\na}\subset [0,1[ \quad
 \lim\a(s_n)=w\in\interior(\Si_1), \quad \tau(w)=+\infty.$
 \end{enumerate}
 Then $w\in W^s(K)$.
\end{Proposition}
\begin{proof}
Let 
$$
k(s)=\#\{\, t\in[0,\tau(\a(s))]: \psi_t(\a(s))\in\bSi\,\}.
$$
By the implicit function theorem $k(s)=k_0$ is constant in $s\in [0,1[$.

Suppose by contradiction that $w\notin W^s(K)$.
Then $\psi_t(w)\notin W^s(K)$ for all $t>0$.
Hypothesis \eqref{lma2} then implies that 
$\psi_{[0,+\infty[}(w)\cap \bSi$ is infinite.
Let $T>0$ be such that $\psi_{[0,T[}(w)$ intersects $(k_0+1)$ times the surface $\bSi$. 
By the implicit function theorem there 
is a neighborhood $U$ of $w$ in $\interior(\Si_1)$
such that for each $x\in U$, the curve 
$\psi_{[0,T[}(x)$ intersects $(k_0+1)$-times
 $\bSi$.
 Therefore 
 $\tau(\a(s))\le T$ whenever $\a(s)\in U$.
 
 Thus $\tau(\a(s_n))\le T$ for $n$ large enough.
 There is a subsequence $s_{n_k}$ such that 
 \linebreak
 $0<\tau_1:=\lim_k\tau(\a(s_{n_k}))\le T$ exists.
 Then
 $$
 \psi_{\tau_1}(w)=\lim_k\psi_{\tau(\a(s_{n_k}))}(\a(s_{n_k}))\in\Si_1.
 $$
 Therefore $\tau(w)\le \tau_1<+\infty$. A contradiction.

\end{proof}

The previous results will allow us to obtain homoclinics 
for branches of periodic points whose closure is not included in $S_0$.
For the remaining case we will use theorem~\ref{toc1}.

We remark that the proof of existence of homoclinic orbits in~\cite{OC1}, once the auto accumulation
of invariant manifolds is known, only uses the dynamics of the map in a neighborhood of the
 invariant manifolds.
We will use the following

\begin{Theorem}[Oliveira, Contreras~\cite{OC1}, corollary~4.10]
\label{toc1}\quad

Let $S$ be a compact connected orientable surface with boundary. 
Let $S_0\subset S$ be a submanifold with compact boundary 
$\partial S_0\subset \partial S$
 and let $f,f^{-1}:S_0\to S$
be an orientation preserving and area preserving homeomorphism of $S_0$ onto
open subsets $fS_0$, $f^{-1}S_0$ of $S$ with $f(\partial S_0)\subset \partial S_0$.

\begin{enumerate}
\item 
Let $p\in S_0-\partial S$ be a periodic point of $f$ of saddle type.
Assume that the branches of $p$ have closure
included in $S_0$.
Assume also that each branch of $p$ accumulates on both
of its adjacent sectors and that all the branches of $p$
have the same closure in $S$.
If in addition $S$ has genus 0 or 1,
then the four branches of $p$ have homoclinic points.

\item
Let $C$ be a connected component of $\partial S_0$ and suppose that 
all the periodic points $p_1,\ldots,p_{2n}$ of $f$ in $C$
are of saddle type.
Let $L_i$ be the branch of $p_i$ contained in
$S-\partial S$.
Assume that for every $i$, $L_i$ is not a connection and
$cl_SL_i=cl_SL_j\subset S_0$ for every pair $(i,j)$.

If in addition $S$ has genus 0, then every pair $L_i$, $L_j$ of stable and 
unstable branches intersect. 
The same happens if the genus of $S$
is 1 provided that there are at least $4$ periodic points in $C$.
\end{enumerate}

\end{Theorem}

Theorem~\ref{toc1} is the version for periodic points of theorem~4.4 in~\cite{OC1}.
The proof of theorem~4.4 in~\cite{OC1} can be read independently
of the rest of the paper.

\bigskip

\noindent
{\bf Proof of item~\eqref{B3} of theorem~\ref{TB}:}

By proposition~\ref{p17} every orbit $\ga\in\Kfix$ has homoclinics in all its separatrices.

By proposition~\ref{PQW} the same happens for a periodic orbit $\ga$ if $\ov{W}^s(\ga)\cup{\ov W^u(\ga)}$ intersects 
$$
\W:=W^s(\Kfix)\cup W^u(\Kfix)\setminus \Kfix.
$$
 So assume that $\ga$ is a hyperbolic periodic orbit with $\ga\cap \ov S\ne\emptyset$ and 
\begin{equation}\label{iww}
\big ( \ov W^s(\ga)\cup \ov W^u(\ga)\big)\cap \W =\emptyset.
\end{equation}

Let $S$ be a component of the complete system.
Let $\tau_\pm:S\to \re\cup\{-\infty,+\infty\}$ be the first return times to $S$,
defined in~\eqref{etaupm},~\eqref{tpm2},
 let $S_0=[\tau_+<+\infty]\cap [\tau_->-\infty]$ be as in~\eqref{eso1}, \eqref{eso},
 and let $f,\, f^{-1}:S_0\to S$ be the extensions of the  first return maps as in \eqref{eff} 
 and \S\ref{bdrymap}.

   Since $\ga$ is a periodic orbit, $|\tau_\pm|$ are finite on $\ga\cap S$, bounded by the period of $\ga$.
   Thus $\ga\cap S\subset S_0$.
   Let $p\in S_0=\interior(S_0)\cup\partial S_0$, $p\in\ga\cap \ov S$,  be a saddle point for $f$ 
  and let $L\subset\interior(S)$ be an 
  interior branch of $p$.  Let $Q$ be the separatrix of $\ga$ which contains $L$.
  
  Suppose that $\tau_+$ is unbounded on $L$. 
  Let $L_1\subset L$ be the connected component
  of $L\cap[\tau_+<\infty]$ with $p\in L_1$.
  Then $\tau_+$ is unbounded on $L_1$.
  Let $\a:[0,1[\to L_1$ be a parametrization of $L_1$.
 Then there is a sequence $s_n\in[0,1[$
 with $\lim_n\tau_+(\a(s_n))=+\infty$. Extracting a subsequence
 we can assume that $z_0=\lim_n \a(s_n)$ exists.
 Since $[\tau_+=\infty]$ is compact, lemma~\ref{Lnti} implies that 
 $\tau_+(z_0)=\infty$. 
 But condition~\ref{dcsss}.\eqref{css2}  implies that 
 $\tau_+$ is bounded in a neighborhood of $\Krot\cap\partial S$.
 Thus $z_0\in \interior(S)\cup \Kfix$.
 If $z_0\in\interior(S)$
 then proposition~\ref{lma} implies that 
 $z_0\in W^s(\Kfix)$.
 Therefore $z_0\in\ov L\cap (W^s(\Kfix)\setminus\Kfix)\ne \emptyset$.
 This contradicts~\eqref{iww}.

      \begin{figure}
   \includegraphics[scale=.55]{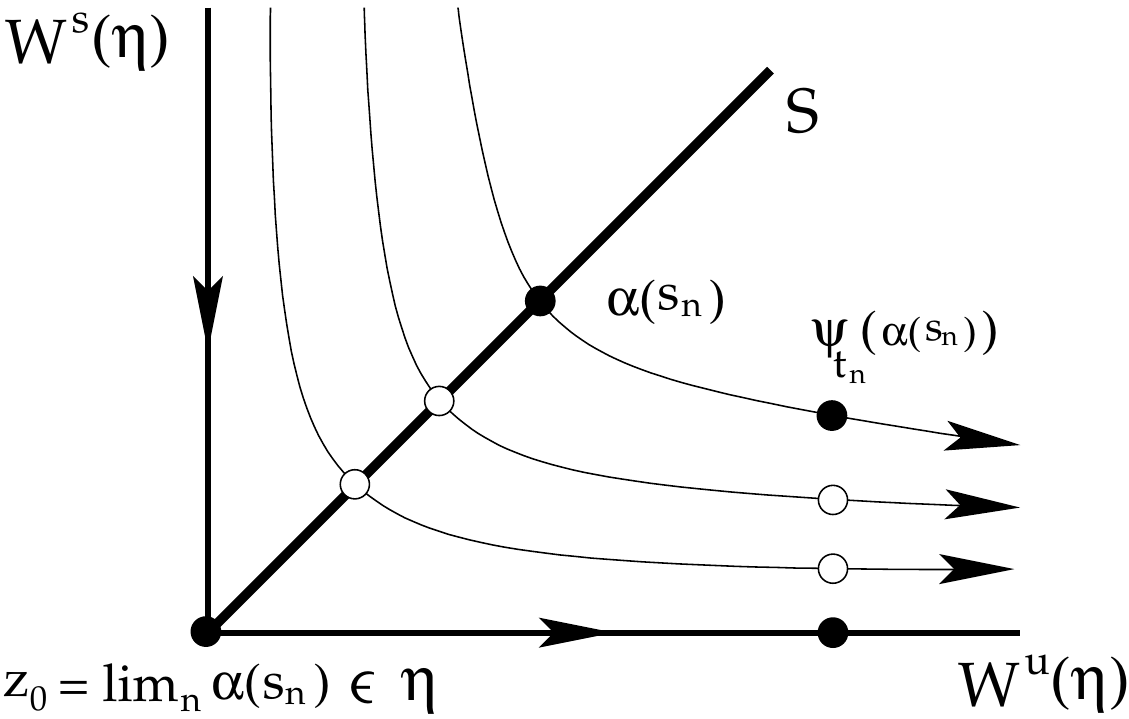}
   \caption{\footnotesize
  \newline
        When $\lim_n\a(s_n)\in\eta\subset \Kfix$,   
   the forward orbits of the $\a(s_n)$ approach $W^u(\eta)\setminus\eta$.     }
   \label{ang}
   \end{figure}

 Then $z_0=\lim_n\a(s_n)\in \Kfix\cap\partial S$.
 Let $\eta=\psi_\re(z_0)\in\Kfix$.
 The surface $S$ approaches the non rotating boundary orbit
  $\eta\subset\partial S\cap\Kfix$
 through a quadrant of $\eta$ as in figure~\ref{ang}.
  There are $t_n\ge 0$ such that $\lim_n(\psi_{t_n}(\a(s_n)))\in W^u(\eta)\setminus\eta$.
   Since $W^u(\eta)\setminus\eta$ does not contain periodic orbits, this limit is in
   $W^u(\eta)\setminus \Kfix$.
  Since $L_1\subset Q$ and $Q$ is invariant, $\psi_{t_n}(\a(s_n))\in Q$.
  Therefore $\ov{Q}\cap (W^u(\eta)\setminus\Kfix)\ne\emptyset$. This contradicts~\eqref{iww}.

 This proves that $\tau_+$ is bounded on $L$. 
 A similar\footnote{For the boundedness of $\tau_-$ we apply proposition~\ref{lma} to the inverse flow $\psi_{-t}$.} 
 proof shows that $\tau_-$ is bounded on $L$.
  
  Now assume that $\tau_\pm$ are bounded on $L$ and $\text{genus}(S)\le 1$.
  Then there is $N>0$ such that 
   $$
    \ov L\subset[\tau_+\le N]\cap[\tau_-\le N]\subset S_0,
   $$
as required in theorem~\ref{toc1}.
In order to apply theorem~\ref{toc1} we need to show
that 
\begin{enumerate}[(a)]
\item
If $p\in\interior(S)$ then each branch of $p$
accumulates on both of its adjacent sectors and
 all branches of $p$ in $\interior(S)$ have the same closure.
\item If $p\in\partial S$ (and hence $\ga=\psi_\re(p)\subset \Krot\cap \partial S$),
then all the components of $(W^s(\ga)\setminus\ga)\cap S$ and of $(W^u(\ga)\setminus\ga)\cap S$ have the same closure.
\end{enumerate}
Then corollary~\ref{c115} finishes the proof of item~\ref{B3} of theorem~\ref{TB}.

\qed

\begin{Lemma}\label{ltb3}\quad

Let $(N,\la)$ be a Kupka-Smale closed contact 3-manifold.
Let $S$ be a component of 
\linebreak
a complete systems of surfaces of section for $(N,\la)$.
Let $S_0$ be as in \eqref{eso} and let 
\linebreak
$f:S_0\cup\partial S_0\to S$ be the 
extension of the first return map to $S$ made in section~\ref{bdrymap}.
\linebreak
 Let $\ga$ be a hyperbolic closed orbit for $(N,\la)$,
$\ga\notin \Kfix$, such that 
\begin{equation}\label{swws}
S\cap(\ov W^s(\ga)\cup \ov W^u(\ga))\subset S_0.
\end{equation}
Let $p\in \ga\cap (S\cup\partial S)$ be a periodic point for $f$ 
and let $L_p\subset S\setminus\partial S$ be an interior branch of $p$.

Then $L_p$ accumulates on both of its adjacent sectors. 
\end{Lemma}

 \begin{proof}\quad
 
  Let $\tau_+$ and $f$ be from~\eqref{etaupm} and~\eqref{eff}.
    The branches in $S\setminus\partial S$ 
of $p$ for $f$ are the connected components of the 
intersection of the separatrices of $\ga$ with $S$ that contain $p$ as an endpoint. 
Let $Q$ be the separatrix of $\ga$ containing the branch $L_p$.
 By item~\eqref{B1} of theorem~\ref{TB}
we know that $Q$ accumulates on both of  its adjacent sectors in $(N,\la)$.

By hypothesis~\eqref{swws}, all the branches of the $f$-orbit of $p$ in $S$
have closure included in $S_0$. The map $f:S_0\to S$ is well defined in $S_0$
and is an injective immersion. In particular $f$ is continuous in a neighborhood in $S$ of
the branches of the orbit of $p$. And every connected component of
$S\cap (W^s(\ga)\cup W^u(\ga))$ has an endpoint in an element
 of the $f$-orbit of $p$.

Let $A$ be a sector for $(f,p)$ in $S$  adjacent to $L_p$. 
Let $n$ be the minimal period of $p$, $f^n(p)=p$.
There are at most $2n$ connected components of $Q\cap S$
and they  are branches of the iterates $f^i(p)$.
At least one of these components accumulates on the sector $A$.

Suppose first that $\ga$ is a  positive hyperbolic orbit.
Then for every $i\in\Z$, $f^i(L_p)$ is the unique connected
component of $Q\cap S$ with endpoint $f^i(p)$.
Let $L_k$ be a connected component of $Q\cap S$ which accumulates
on the sector $A$.
Then $L_k$ is a branch of an $f$-periodic point $p_k\in\ga\cap S$.
There is $0\le k<n$ such that $f^k(p)=p_k$.
If $k=0$ then the lemma holds. Assume $k\ge 1$.
Since $L_k$ accumulates on the sector $A$ adjacent to $L_p$,
we have that $L_p\cap \ov{L_k}\ne\emptyset$.
The compact set $\ov{L_k}$ is invariant under $f^n$ and $\ov{L_k}\subset S_0$.
By the accumulation lemma~\ref{ACL} applied to $f^n$, $L_p\subset \ov{L_k}$.

Observe that  $f^k(L_k)=:L_{2k}$ is a connected component of $Q\cap S$
with endpoint $p_{2k}=f^k(p_k)$ and it accumulates on the sector 
$f^k(A)$ of $p_k$. Similarly $L_k\subset \ov{L_{2k}}$. And then
$L_p\subset \ov {L_k}\subset\ov{L_{2k}}$.
Inductively $L_{nk} = f^{nk}(L_p)$ is a component of $Q\cap S$ with endpoint
$p_{nk}=f^{nk}(p)=p$. Thus $L_{nk}=L_p$.
Moreover $\ov{L_k}\subset \ov{L_{2k}}\subset\cdots \subset \ov{L_{nk}}=\ov{L_p}$.
Therefore $L_p$ accumulates on the sector $A$.

Suppose  now that $p\in\ga\subset \partial S$ is a boundary $f$-periodic point.
The return map $f$ preserves the area form of $S$ and hence it preserves orientation.
This implies that $p$ is a positive hyperbolic orbit for $f$.
The orbit $\ga$ for the flow may be negative hyperbolic 
but the return map $f$ permutes the interior components of $Q\cap S$.
The previous proof of the positive hyperbolic case applies here.

Now suppose that $p\in S\setminus\partial S$ is a negative hyperbolic periodic point for $f$.
Let $n>0$ be its minimal period, $f^n(p)=p$. Let $J_k$ be a connected 
component of $Q\cap S$ which accumulates on the sector $A$. Let $p_k$ 
be the endpoint of $J_k$ and let $k\in\na$ be such that $f^k(p)=p_k$.
In the case  $L_k:=f^k(L_p)=J_k$ the same proof as in the positive hyperbolic case follows,
with $L_j:=f^j(L_p)$ and $\ov{L_k}\subset \ov{L_{2k}}\subset \cdots \subset \ov{L_{2nk}}=\ov{L_p}$.
The accumulation lemma~\ref{ACL} is applied to $f^{2n}$ which leaves the 
branches $L_{ik}$ invariant.
We iterate $2n$ times $f^k$, 
because the map $f^{2nk}$ fixes the branch $L_p$.

Suppose then that $f^k(L_p)\ne J_k$.
For $j\in\Z$, let  $L_j:=f^j(L_p)$ and $A_j:=f^j(A)$.
Let $K_j$ be the other component of $Q\cap S$ with 
endpoint $f^j(p)$. In local coordinates $K_j$ is the
branch $-L_j$ adjacent to the sector $-A_j$.
Then $J_k=K_k$ and $K_j=f^j(K_0)$.
 The branch $K_k$ accumulates on the sector $A$.
Then the branch $L_k=f^n(K_k)$ accumulates on the sector $f^n(A)=-A$,
adjacent to the branch $K_0$.
By the accumulation lemma~\ref{ACL} applied to $f^{2n}$, 
for which the branches are invariant, $L_p\subset \ov{K_k}$ and $K_p:=K_0\subset \ov{L_k}$.

Since the branch $K_k$ accumulates on the sector $A$, we have that 
the branch $K_{2k}=f^k(K_k)$ accumulates on the sector $A_{k}=f^k(A)$, adjacent to $L_k$.
And using $f^n$, the branch $L_{2k}:=f^{2k}(L_p)=f^n(K_{2k})$ accumulates on the 
sector $-A_k=f^n(A_k)$, adjacent to $K_k$. Using the accumulation lemma~\ref{ACL} we get
that
\begin{equation}\label{lkl}
L_k\subset \ov{K_{2k}}\qquad\text{ and }\qquad K_k\subset \ov{L_{2k}}.
\end{equation}
Applying $f^j$ to the inclusions in~\eqref{lkl} and using that $L_j=f^j(L_0)$,
$K_j=f^j(K_0)$ we have that 
$$
L_p\subset \ov{K_k}\subset \ov{L_{2k}}\subset \ov{K_{3k}}\subset \ov{L_{4k}}\subset\cdots
$$
In the iterate $2nk$ we have that
$$
K_k\subset \ov{L_{2nk}}=f^{2nk}(\ov {L_p})=\ov{L_p},
$$
 then $L_p$ accumulates on the sector $A$.
 
  \end{proof}

 The hypothesis in theorem~\ref{toc1} asks for more than lemma~\ref{ltb3}, namely

 \begin{Corollary}\label{c115}\quad

Let $p\in S_0\cup\partial S_0$ 
be a saddle point for the extension  $f:S_0\to S$  of the return map.
Assume that 
\begin{equation}\label{c115e}
\ga:=\psi_\re(p), \qquad
\ov{S\cap W^s(\ga)}\subset S_0
\qquad\&\qquad
\ov{S\cap W^u(\ga)}\subset S_0.
\end{equation}
Then
\begin{enumerate}
\item\label{c1151} If $p\in S\setminus\partial S$, all the branches  of $p$
have the same closure and 
accumulate on all the sectors of $p$.
\item\label{c1152}
If $p\in\partial S$, $\ga=\psi_\re(p)$, $p_1,\ldots,p_{2m}$ are the periodic points 
of $f$ in $\ga\subset\partial S_0$ and $L_i\subset S\setminus \partial S$ 
is the interior branch of $p_i$, then $\ov{L_i}=\ov{L_j}$ for every $(i,j)$.
\end{enumerate}
\end{Corollary}

\begin{proof}\quad

\eqref{c1151}.
By lemma~\ref{ltb3}, it is enough to prove that the branches of $p$ in 
$S\setminus\partial S$ have the same closure.
Let $n$ be such that $f^n(p)=p$.
Let $L_0$, $L_1$ be two branches of $p$ adjacent to the same sector $A$.
Since by lemma~\ref{ltb3}, $L_0$ accumulates on the sector $A$; we have that 
$L_1\cap \ov{L_0}\ne\emptyset$.
By hypothesis $\ov{L_0}\subset S_0$.
Also $f^{2n}(L_0)=L_0$.
By the accumulation lemma~\ref{ACL} applied to $f^{2n}$, we have that 
$L_1\subset\ov{L_0}$. Iterating this argument we have that 
all branches of $p$ have the same closure.

   \begin{figure}
      \includegraphics[scale=.5]{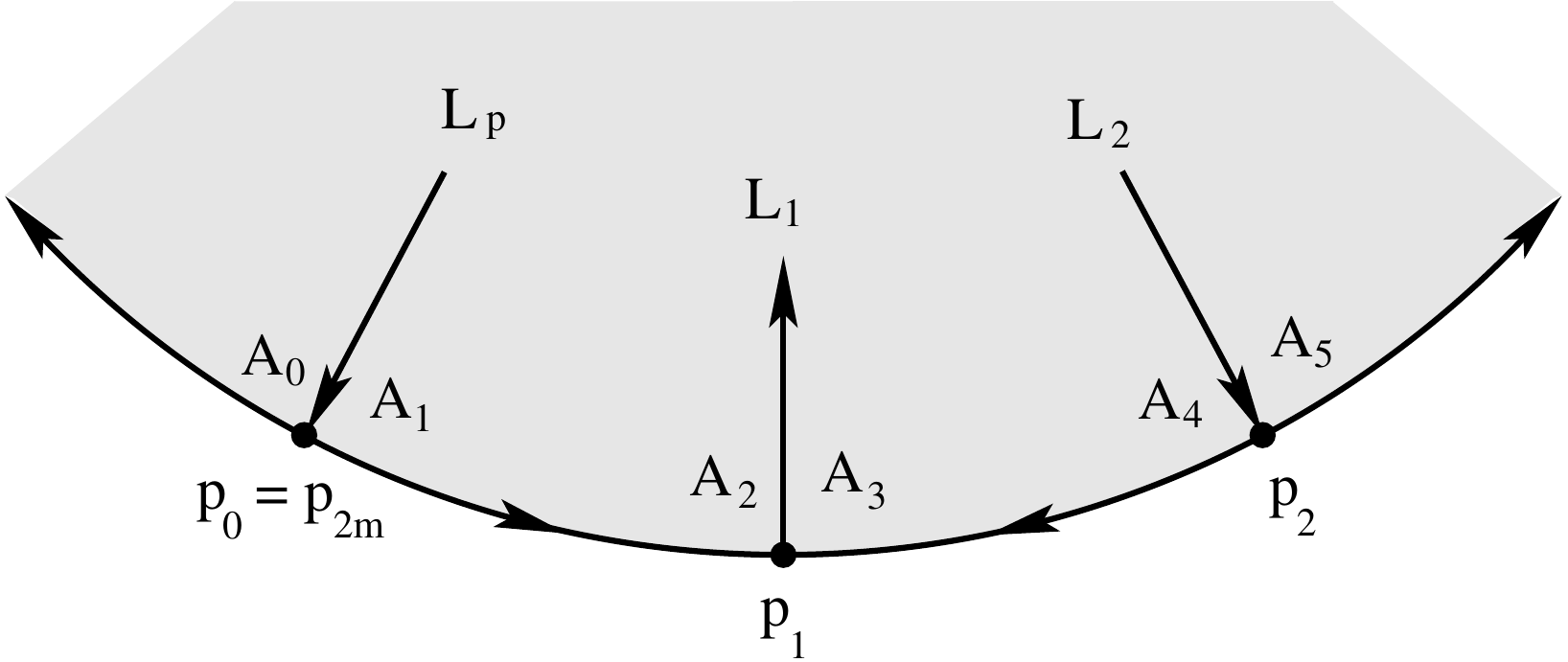}
   \caption{\footnotesize Saddle periodic points and sectors at a hyperbolic rotating boundary orbit.}
   \label{BD}
   \end{figure}

\eqref{c1152}. Let $p=p_0,p_1,\ldots, p_{2m}=p_0$ be the ordered 
 periodic points of $f$  in  $\ga\subset\partial S$, as in figure~\ref{BD}.
Let $L_p=L_0=L_{2m}$ and let $L_i$ be the branch of $p_i$ in $S\setminus \partial S$.
Then the $L_i$'s are connected components of $S\cap W^{s,u}(\ga)$. 
By \eqref{c115e}, $\ov{L_i}\subset S_0$ and 
 the $L_i$'s are all the components of $S\cap W^{s,u}(\ga)$.
Let $A_{2i},A_{2i+1}$ be the sectors adjacent to $p_i$ chosen so that $A_{2i-1}$ and $A_{2i}$
are adjacent to a connection between $p_{i-1}$ and $p_i$.

By lemma~\ref{ltb3}, $L_i$ accumulates in its adjacent sector $A_{2i+1}$.
Due to the connection, $L_i$ also accumulates on the sector $A_{2i+2}$,
 adjacent to $L_{i+1}$. Then $\ov{L_i}\cap L_{i+1}\ne\emptyset$.
 By the accumulation lemma~\ref{ACL} applied to $f^N$ where $N$ is a multiple of the periods of
 $L_i$ and $L_{i+1}$, we have that $L_{i+1}\subset \ov{L_i}$. Thus
 $$
 \ov{L_0}\supset \ov{L_1}\supset \ov{L_2}\supset\cdots \supset\ov{L_{2m}}=\ov{L_0}.
 $$
Therefore $\ov{L_i}=\ov{L_j}$ for every $(i,j)$.

\end{proof}

\bigskip

\color{black}

 \def\oth{{\ov{\theta}}}
 \def\drho{{\dot{\rho}}}
 \def\oom{{\ov{\om}}}
 \def\te{{\text{e}}}

 \newcommand{\spun}[1]{\text{span}\langle #1\rangle}

  \section{The extension to the boundary of the return map.}\label{bdrymap}

 In this section we study the extension to the boundary at a rotating boundary
  orbit of the return map to a component of a complete system of surfaces 
  of section a Reeb flow. On non rotating boundary orbits the return time is
  infinite and the return map is not defined.

 We first consider the case in which rotating boundary orbit $\Ga$ at the boundary
 of a component $\Si$ is irrationally elliptic. In that case we prove that 
 if the Floquet multipliers of the elliptic orbit are not roots of unity then
  the extension  to the boundary
 has no periodic point.

  Afterwards we deal with the case in which the periodic orbit $\Ga$
  at the boundary of the component $\Si$ is hyperbolic.
 We show that there exist a continuous extension of the Poincar\'e map
 to the boundary $\Ga$ of $\Si$. This extension is a Morse-Smale
 map in $\Ga$ with periodic points on $\Ga$ that, seen in $\Si$, are 
 saddle points. These periodic points do not correspond to other
 periodic orbits of the Reeb flow  but their {\it interior} invariant
 manifolds $W^s$, $W^u$ are the intersections of the invariant
 manifolds $W^u(\Ga)$, $W^s(\Ga)$ of the closed orbit $\Ga$ with the
 component $\Si$ and hence their intersections  with other
 invariant manifolds of the return map to $\Si$ are transversal if the Reeb flow is Kupka-Smale.

    \begin{figure}
   \includegraphics[scale=.35]{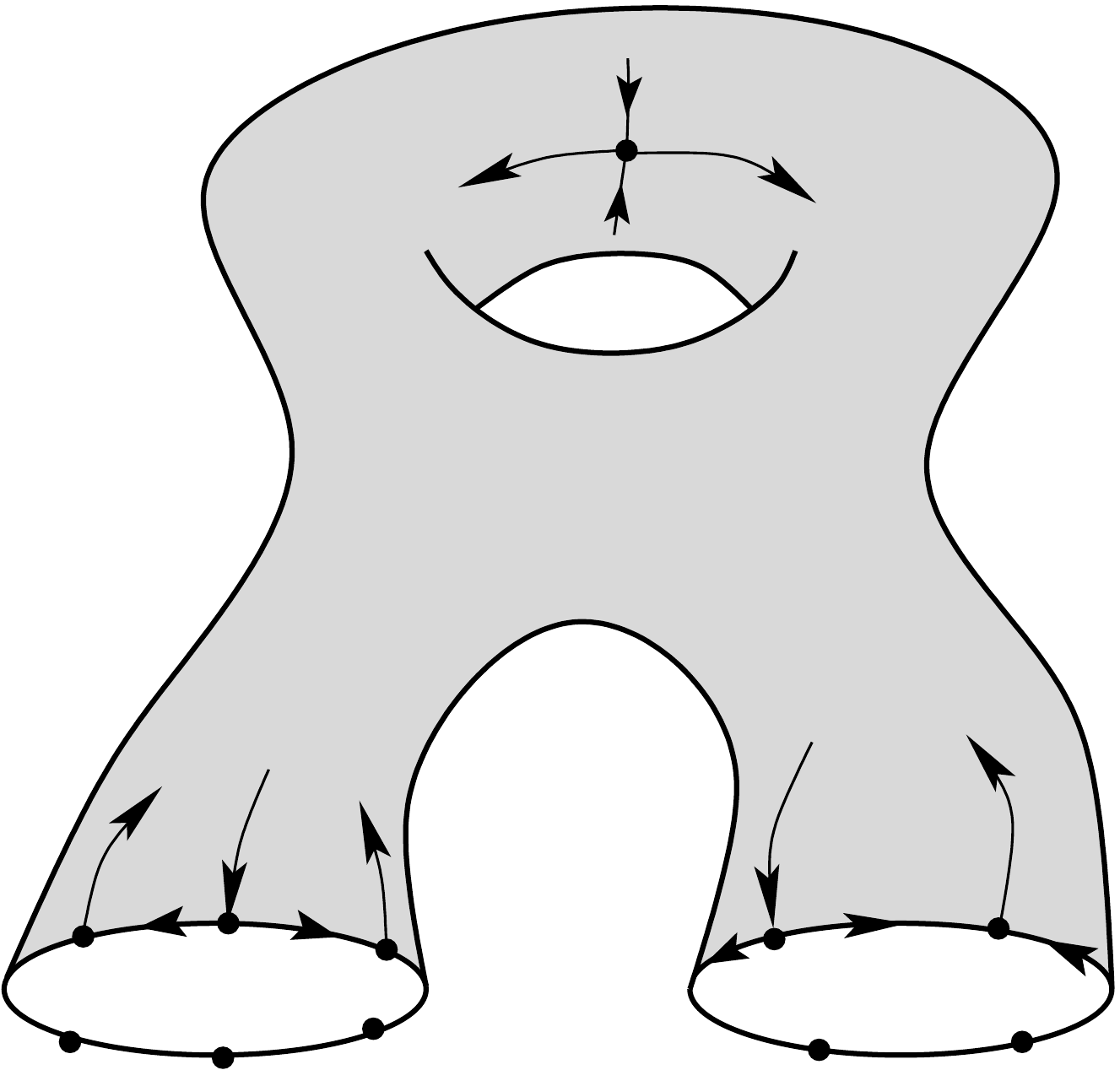}
   \caption{\footnotesize 
    The figure  shows the dynamics of the extension of the return map $P$
   to the surface of section $\Si$ with hyperbolic rotating boundary orbits. }
   \label{boundary}
   \end{figure}

 \bigskip

 \subsection{The elliptic case.}

 \begin{Proposition}\label{NPP}\quad
 
 If the binding periodic orbit $\Ga$ is elliptic and its Floquet multipliers are not roots of unity, then
 the extension of the return map to the boundary $\Ga=\partial\Si$ has no periodic points.
 \end{Proposition}

 \begin{proof}
 As we shall see in \S\ref{transv} the component
 $\Si$ has an asymptotic direction
 $e(t)\in\xi(\Ga(t))$, where $\xi$ is the contact structure.
 The direction of $e(t)$ turns more slowly than its movement $s\mapsto d\vr_s(e(t_0))$
 under the linearized Reeb flow. Then the extension $P$ of the return map
 to $\Ga=\partial\Si$ is given by $\Ga(t)\mapsto\Ga(b(t))$, where
 $$
 b(t)=t+\min\{\,s>0\;|\;\exists \la >0:\;  d\vr_{s}(e(t))=\la\,e(t+s)\,\}.
 $$

 If the extension $P$ has a periodic point at $\Gamma(t)$ then the subspace generated
 by $e(t)$ is invariant under $d\vr_T$, where $T$ is a multiple of the period of $\Ga$.
 Then some iterate of the Poincar\'e map of the periodic orbit $\Ga$ has an invariant
 1-dimensional subspace.  Since $\Ga$ is elliptic, this implies that the Poincar\'e map
 of $\Ga$ has an eigenvalue which is a root of unity. This is a contradiction.
 \end{proof}

\bigskip

\subsection{The hyperbolic case. Sketch of the proof.}\quad

\bigskip

Let $\Ga$ be a periodic orbit for the Reeb flow of $(N,\la)$.
The contact structure $\xi = \ker\la$ is a subspace transversal to the
Reeb vector field which is non integrable. The image of the exponential map
of a small ball $B_\de(0)\cap \xi(\Ga(t))$,
 $\Xi (\Ga(t))=\exp_{\Ga(t)}(\xi(\Ga(t))\cap B_\de(0))$
 is a system of transversal sections to the Reeb flow
 in a neighborhood of $\Ga$ which is tangent to the 
 contact structure at $\Ga$.

The picture of the return map to the surface  of section nearby  the 
boundary orbit and its extension to the boundary is clear when the 
flow is linear in a neighborhood of the  biding orbit $\Ga$, the 
contact structure $\xi$ is orthogonal to the vector field $X(\Ga)=\dot\Ga$
and the intersections, near the boundary $\partial\Si$, 
of the surface of section $\Si$ with the transversals
$\Xi$, are images under the exponential map of straight lines. 
In this case the set of images of  straight lines
in $\xi$ passing through $\Ga$ is a singular foliation invariant\footnote{
This does not happen on a contact flow but may happen for 
a reparametrization of the flow.} by the flow.

In the following paragraphs we show that the return map to $\Si$ is conjugate
to the situation described above. We choose a coordinate system in which
the strong invariant manifolds $W^{s}(\psi_tz)$, $W^{u}(\psi_t z)$ coincide with coordinate axes,
and the flow is linear. In these coordinates the surface of section turns around the
axis at least an angle $\pi$. We construct an open book decomposition $\cF$ 
of a neighborhood of $\Ga$ with spine $\Ga$, which is invariant under the
Reeb flow, whose intersections with the transversals $\Xi$ 
are straight lines  in these coordinates. 

Then we show that nearby $\Ga$ there is an isotopy $\Si_s$ between $\Si$ 
and a surface $\Si_1$, with the properties that for all $s$, $\Si_s$ is transversal 
to the Reeb vector field $X$ and that,
for the final surface,
$\Si_1 \cap \Xi(\Ga(t))$ is a leaf of $\cF \cap \Xi(\Ga(t))$ for all $t$. 
Then the return map is conjugated to the map that arises in 
 the simpler situation described above.

For this, in  \S\ref{tangency} we compute the  condition for a small 
cylinder, with $\Ga$ as one boundary component, to be tangent to the 
Reeb vector field $X$.  In  \S\ref{isot} we give sufficient 
conditions for such  isotopy to give surfaces transversal to $X$.
Finally, in  \S\ref{transv},  we use the  transversal 
approach of $\Si$ to $\Ga$ to prove that the isotopy is made by 
surfaces transverse to $X$.

\bigskip

 \subsection{Coordinates and preliminary equations.}\label{Coord}\quad
 \medskip
 
 Let $(N,\la)$ be a 3-dimensional contact manifold,
  let $X$ be its Reeb vector field, $\psi_t$ its Reeb flow 
 and $\xi=\ker\la$ its contact structure.
 Let $\Si$ be a surface of section for $X$ and $\Ga\subset\partial\Si$
 a rotating boundary periodic orbit. This means that there is a neighborhood
 $U$ of $\Ga$ in $\Si$ such that the first arrival times of $\psi_t$ and $\psi_{-t}$ 
 from $U$ to $\Si$ are bounded.
 
 From now on we assume that the boundary periodic orbit $\Ga$ is hyperbolic.
 For simplicity assume
  that the periodic orbit $\Ga$ has period $1$.
 To simplify the notation
 we shall also assume
  that $\Ga$ has 
  negative eigenvalues and that
 each local invariant manifold $W^s_\e(\Ga)$, $W^u_\e(\Ga)$ intersects
 $\Si$ in 
 a  neighborhood of $\Ga$ in 3 connected components. The other cases are similar.

 Since $\Ga$ has negative eigenvalues we can
 choose a smooth coordinate system $(x,y,z)$ near the periodic orbit
 $\Ga$ such that 
 $$
\Ga(t)=(0,0, t+\Z)\in\re^2\times \SS^1/\equiv, \qquad
 \SS^1=\re/\Z, \quad
(x,y,t+1)\equiv (-x,-y,t),
$$
 \begin{align*}
 W^{u}_\e(\Ga(t))&=\{(x,0, t\,\text{mod }1)\,:\,|x|<\e\;\},
 \\
  W^{s}_\e(\Ga(t))&=\{(0,y, t\,\text{mod }1)\,:\,|y|<\e\;\},
 \end{align*}
 and along the periodic orbit one has
 \begin{equation}\label{coords}
 e_u:=\tfrac{\partial}{\partial x}\in \E^u, \quad
 e_s:=\tfrac{\partial}{\partial y}\in \E^s, \quad
 d\eta(e_u,e_s)=1 \quad \text{ and } \quad\tfrac{\partial}{\partial z}=X,
 \end{equation}
 where $\E^s$, $\E^u$ are the stable and unstable subspaces for $\Ga$:
 \begin{align*}
 d\psi_s(\E^{s,u}(\Ga(t))) = \E^{s, u}(\Ga(t+s)),
 \quad \lV d\psi_1|{\E^s}\rV<1,\quad
 \lV d\psi_1|\E^u\rV >1.
 \end{align*}

 Consider  the derivative $DX(0,0,t)$ of the vector field
 along the orbit $\Ga(t)$. Since the subspaces $\E^s$, $\E^u$ are
 invariant under the linearized flow, we have that
 \begin{equation}\label{DX}
 DX(0,0,t)=\left[\begin{matrix}
 \la_t & 0 & 0 \\
 0 & \mu_t & 0 \\
 0 & 0 &  0
 \end{matrix}\right].
  \end{equation}
  The derivative of the flow at the periodic orbit
  $F(t):=d\vr_t(0,0,0)$ satisfies the differential equation
  $\dot{F} = DX(0,0,t)\; F$. Its solution is
  \begin{align*}
  d\vr_t\big[x_0\; e_u(\tau) &+ y_0\; e_s(\tau) + z_0\; X(0,0,\tau)\big] =
   \\
   &=   x_0\,e^{\int_\tau^{\tau+t} \la_s\,ds}\;e_u(t+\tau) +
   y_0\,e^{\int_\tau^{\tau+t} \mu_s\,ds} \;e_s(t+\tau)+ z_0\, X(0,0,t+\tau).
  \end{align*}
   Since $\vr_t$ preserves $d\eta$, 
   $d\eta[d\vr_t(e_u),d\vr_t(e_s)]\equiv 1$.
   This implies that $\mu_t=-\la_t$ for all $t$.

   Since $d\vr_t$ is 1-periodic in $t$, the unique invariant (i.e.
   1-periodic) subspace transversal to $X$ is given by
   $\xi=\spun{e_u, e_s}$, which necessarily coincides with
   the contact structure $\xi=\ker \eta$ along $\Ga$.
   Indeed, if in our coordinates
   the invariant transversal subspace $E_t$ at
   $(0,0,t)$ is given by   $E_t$:~$a_t\,x+b_t\,y+c_t\,z=0$,
   with 1-periodic functions $a_t$, $b_t$, $c_t$.
   Then for any $(x,y,z)$ such that
   $$
   a_\tau \, x + b_\tau\, y + c_\tau\, z =0,
   $$
   we must have that $d_{(0,0,\tau)}\vr_n(x,y,z)\in E_{\tau+n}$, i.e.
   $$
   a_{\tau+n} \,x\, e^{\int_0^n \la_s\;ds}+
   b_{\tau+n}\, y \, e^{-\int_0^n\la_s\;ds}+ c_{\tau+n}\, z =0,
   $$
   for all $n\in\Z$. 
   Since $a_t$, $b_t$, $c_t$ are 1-periodic,
   this implies that $a_\tau=b_\tau=0$ for all
   $\tau$.

   Thus the equation for the derivative of the flow
   $d\psi_t\vert_{\xi}$ restricted to the contact structure $\xi$
   over $\Ga$ is
   \begin{equation}\label{linpart}
    \dw = A(t)\cdot w,   \qquad
    A(t)=\begin{bmatrix} \la_t & 0 \\
          0 & -\la_t \end{bmatrix}.
   \end{equation}
   We write now this differential equation in polar coordinates.

   Let $w(t)= (x(t),y(t))= \zeta(t)\, u_{\be(t)}$
   be a solution of~\eqref{linpart},
   where $u_\be=(\cos\be,\,\sin\be)$.
   Then
   $w(t)= (x_0\, e^{\La_t}, y_0\, e^{-\La_t})$,
   $\La_t=\int_0^t \la_s\;ds$, $w(0)=(x_0,y_0)$. Hence
   $$
   \tan \be(t) = e^{-2 \La_t}  \tan\be(0).
   $$
     Differentiating this equation we obtain
 $$
 (\sec^2 \be)\,\dbe = - 2\la_t\,\tan\be,
 $$
 \begin{align}
 \dbe &= -2\la_t\,\cos^2\be\;\tan\be, \notag\\
 \dbe &= -\la_t\,\sin(2\be).\label{dbe}
 \end{align}
 Also,
 \begin{align}
 \zeta(t)^2 &=e^{2\La_t}x_0^2+e^{-2\La_t}y_0^2 \notag\\
 2 \,\zeta \,\dzeta &= 2\la_t\,\big(e^{2\La_t} x^2_0-e^{-2\La_t} y^2_0\big)
 \notag \\
 \zeta \,\dzeta &= \la_t\,\zeta^2\,(\cos^2\be-\sin^2\be)
 \notag \\
 \dzeta &= \la_t\,\zeta\,\cos(2\be).\label{dzeta}
 \end{align}

 The vector field $X$ satisfies
 $X(x,y,t)=(\la_t\, x,\;-\la_t\,y,\; 1)+ \cO(x^2+y^2)$.
 Let 
 $$\ga(t)=(\rho(t)\,u_{\a(t)},z(t))
 $$ 
 be a solution
 of $\dga=X(\ga)$. Then
 $\dga=(\drho\,u_\a+\rho\,\da\,u_\a^\perp,\;\dz)$,
 $u_\a^\perp:=  (-\sin\a,\cos\a)$, and
 \begin{align*}
 \drho &=\langle\dga,u_\a\rangle
      =\langle X(\ga), u_\a\rangle
      =\langle A(t)\,\ga, u_\a\rangle +\cO(\rho^2), \\
 \rho\,\da &= \langle\dga, u_\a^\perp\rangle
       =\langle X(\ga), u_\a^\perp\rangle
       =\langle A(t)\,\ga, u_\a^\perp\rangle + \cO(\rho^2).
 \end{align*}
 Therefore, using~\eqref{dzeta} and~\eqref{dbe},
 \begin{align}
 \drho &=\la_t\,\rho\,\cos(2\a)+\cO(\rho^2), \label{drho}
 \\
 \rho\,\da &=-\la_t\, \rho\,\sin(2\a)+\cO(\rho^2),
 \notag \\
 \da &=-\la_t\,\sin(2\a)+\cO(\rho).\label{da}
 \end{align}

  Writing $X=(f_1,f_2,f_3)$ let $Y=\tfrac 1{f_3}\, X$ and let $\phi_t$
  be the flow of $Y$. Then $\phi_t$ is the reparametrization of
  $X$ which preserves the solution $\Ga(t)$ and for which the
  foliation ``$z=$~constant'' is invariant. Observe that from~\eqref{DX},
  $\tfrac{\partial f_3}{\partial x}\vert_{(0,0,z)}
  =\tfrac{\partial f_3}{\partial y}\vert_{(0,0,z)}\equiv 0$.
  The vector field $Y$ is not
  a Reeb vector field of a contact form but it is smooth
  and along $\Ga$,
  $Y=X$ and $DY=DX$. In particular,
  the arguments above remain valid for $Y$.

  Consider the $\phi\,$-invariant foliation $\cF$ of
  $U_\e\setminus(\{{\mathbf 0}\}\times \SS^1)$, where
  $$
  U_\e:=\D_\e\times \SS^1,
  \quad
  \D_\e=\{z\in\re^2:|z|<\e\},
  \quad
   \SS^1=\re/\Z,
   $$
  whose leaves are
  \begin{equation}\label{folcf}
  \cF_t(u_{\a_i}):=\{\phi_t(r\,u_{\a_i},s) : r\in]0,\de[,\;s\in \SS^1\,\}\cap U_\e,
  \quad \a_i=\tfrac{\pi}{4} i, \quad i=0,1,\ldots,7,
  \end{equation} 
  with $0<\e\ll\de\ll 1$.
  This is a ``radial'' foliation which satisfies $\phi_s(\cF_t(\a_i))=\cF_{s+t}(\a_i)$.
  Observe that
  $\cF_s(u_0)=\cF_t(u_0)\subset W^u(\Ga)$ for all $s,t$
  and also $\cF_t(u_{\frac{\pi}2})$,~$\cF_t(u_{\frac{3\pi}2})\subset
  W^s(\Ga)$ and $\cF_t(u_{\pi})\subset W^u(\Ga)$.

  Let $\Si$ be a surface of section having $\Ga$ as a rotating boundary orbit.
   We will construct an isotopy
  of $\Si\cap U_\e$ along surfaces $\Si_\si$, $\si\in[0,1]$
   such that $\partial\Si_\si\cap \interior(U_\e)=\Ga$,
    $\Si_0=\Si$, $\Si_\si$ is transversal to $Y$
    and such that for all $\tau\in \SS^1$,
     $\Si_1\cap[z=\tau]$ is included in one
     leaf of $\cF$.

  \subsection{The tangency condition.}\label{tangency}\quad

  Consider an annular smooth surface $S$ in $U_\e$ with boundary
  $\partial S\cap \interior(U_\e)=\Ga$ which has a well defined
  limit tangent space at the points in $\partial S=\Ga$.
  Then, for $\e>0$ small enough,
   $S\cap U_\e$ is the image of a map $F:[0,\e[\times \SS^1\to U_\e$,
  $F(r,t)=(r\,u_{\th(r,t)},t)$, where $\theta$ is a $C^1$
  map, with continuous derivatives at $r=0$. We obtain now the
  conditions for $S$ to be tangent to the vector field $Y$ on
  $r>0$:

   \begin{Lemma}\label{Ltang}
 If the surface $F(r,t)=(r \,u_{\theta(r,t)} ,t)$
 is tangent to the vector field
 $Y$ at a point $(r,t)$ then
 $\th(r,t)$ satisfies
 \begin{equation}\label{tang}
  \th_t = -\la_t\,\big[\,\sin(2\th)+ r\,\th_r\,\cos(2\th)\,\big]+\cO(r)
  \qquad\text{ at } (r,t).
 \end{equation}
 \end{Lemma}

  \begin{proof}
  The tangent plane to $S$ is generated by
  \begin{align*}
  r\cdot\tfrac{\partial F}{\partial r}
  &= \big( r\,u_\th + r^2\,\th_r\,u_\th^\perp,\;0\big),
  \\
  \tfrac{\partial F}{\partial t}
  &= \big( r\,\th_t\,u_\th^\perp,\;1\big).
  \end{align*}
  Let $\ga(t)=\big(\rho(t)\,u_{\a(t)},t\big)$ be an orbit
  of the flow $\phi$ of $Y$. Then $\rho(t)$ and $\a(t)$
  also satisfy equations~\eqref{drho} and~\eqref{da}.
  The surface $S$ is tangent to the vector field $Y$
  at $r\,u_\th=\rho\, u_\a$ if and only if there  exists $a,\,b\in\re$
  such that
    \begin{align*}
    \big(\drho\,u_\a+\rho\,\da\,u_\a^\perp\,,\,1)
    &= a\;r\,\tfrac{\partial F}{\partial r}
    + b\,\tfrac{\partial F}{\partial t}
    \\
    &= a\,\big( r\,u_\th + r^2\,\th_r\,u_\th^\perp\,,\,0)
    + b\,\big(  r\,\th_t\,u_\th^\perp\,,\,1).
    \end{align*}
    In this case $b=1$ and, using~\eqref{drho} and~\eqref{da},
    \begin{gather}
    \drho = a\,r = \la_t\,\rho\,\cos(2\a) +\cO(\rho^2),
    \label{a}\\
    \rho\,\da = a\, r^2\,\th_r+r\,\th_t
    = -\rho\,\la_t\,\sin(2\a)+\cO(\rho^2).
    \label{thr}
    \end{gather}
    Since $\rho=r$ and $\a=\th$, from~\eqref{a} we get
    that
    $a= \la_t\,\cos(2\th)+\cO(r)$.
    Substituting $a$ in~\eqref{thr} we get ~\eqref{tang}.

    \end{proof}

 \subsection{The isotopy.}\label{isot}\quad

  Let $F:[0,\e[\times \SS^1\to U_\e$, $F(r,t)=(r\,u_{\th(r,t)},\,t)$
   be a local parametrization of the 
    surface of section $\Si=:\Si_0$.
   Let $\oth(t):=\th(0,t)$.
   Let $G:[0,\e[\times \SS^1\to U_\e$, $G(r,t)=(r\,u_{\om(r,t)},\,t)$
   be defined by $G([0,\e[,t)\in\F\cap(\re^2\times\{t\})$, where
   $\F$ is the leaf of the foliation $\cF$ in \eqref{folcf} 
   such that its tangent
   space at $(0,0,t)$ is $E=\spun{\,(0,0,1),\,(u_{\oth(t)},0)\,}$.

\newpage

 \begin{Lemma}\label{L>}\quad
 
 Let $\theta,\,\om:[0,\e[\times \SS^1\to \re$ be of class $C^1$.
 Write $\oth(t):=\th(0,t)$ and $\oom(t):=\om(0,t)$. 
 \newline
   For $\mu\in[0,1]$ write
 $\vr^\mu(r,t):=\mu\,\om(r,t)+(1-\mu)\,\th(r,t)$.
 Suppose that
 \begin{equation}\label{tht}
 \oth_t < -\la_t\,\sin(2 \oth)
 \quad\text{ and }
 \quad \oom(t)=\oth(t)
 \quad\text{ for all }\;t\in \SS^1.
 \end{equation}
 Then there is $\rho_0>0$ such that for all  $\mu\in[0,1]$,
 the surface $H^\mu(r,t)=\big(r\,u_{\vr^\mu(r,t)},t\big)$
 is transversal to the vector field $Y$ at all points
 $(r,t)$ with $0<r<\rho_0$, $t\in \SS^1$.
 \end{Lemma}

 \begin{proof}
  Since $\SS^1$ is compact, there exists $\e>0$ such that
 $$
 \oth_t < -\la_t \sin(2\oth)-3\e
 \quad\text{ for all }\;t\in \SS^1.
 $$
 Choose $\rho_1>0$ such that for all
 $0\le r<\rho_1$ and $t\in \SS^1$,
 \begin{gather*}
 |\la_t|\, r\, |\th_r|<\tfrac\e 4
 \,,\quad
 |\la_t|\, r\, |\om_r|<\tfrac\e 4\,,
 \\
 |\la_t|\,|\oth(t)-\th(r,t)|<\tfrac\e{16}
 \,,\quad
 |\la_t|\,|\oth(t)-\om(r,t)|<\tfrac\e{16}
 \quad\text{ and }
 \\
 \th_t<-\la_t\,\big[\sin(2\th)+r\,\th_r \cos(2\th)\big]-2\e,
 \\
 \om_t<-\la_t\,\big[\sin(2\om)+r\,\om_r \cos(2\om)\big]-2\e.
 \end{gather*}
 Then
 \begin{gather*}
 |\la_t|\,|\vr^\mu-\oth|<\tfrac\e{16}, \quad |\la_t|\,|\vr^\mu-\th|<\tfrac\e8,
 \quad |\la_t|\,|\vr^\mu-\om|<\tfrac\e8,
  \\
 |\la_t|\,|\sin(2\vr^\mu)-\sin(2\oth)|<\tfrac\e8
 \quad\text{ and}\\
 |\la_t|\,|\sin(2\vr^\mu)-\sin(2\th)|<\tfrac\e4,
 \quad
 |\la_t|\,|\sin(2\vr^\mu)-\sin(2\om)|<\tfrac\e4.
 \end{gather*}
 Hence
 $$
 |\la_t|\,\big|\sin(2\vr^\mu)-[\mu\,\sin(2\om)+(1-\mu)\,\sin(2\th)]\big|
 <\tfrac\e4.
 $$
 Also, for all $(r,t)\in[0,\rho_1[\times \SS^1$ and $\mu\in[0,1]$,
 since
 $\vr^\mu_r=\mu\,\om_r+(1-\mu)\,\th_r$, we have that
  $$
 |\la_t\, r \,\th_r \,\cos(2\th)|<\tfrac\e4,\quad
 |\la_t\, r \,\om_r \,\cos(2\om)|<\tfrac\e4
 \quad\text{ and }\quad
 |\la_t\, r \,\vr^\mu_r \cos(2\vr^\mu)|<\tfrac\e4.
 $$
 Thus, we have that
 \begin{align*}
 \vr^\mu_t
  &= \mu\,\om_t + (1-\mu)\,\th_t \\
 &<-2\e -\la_t\,\big[\mu\,\sin(2\om)+(1-\mu)\,\sin(2\th)\big]
             -\mu\,\la_t\,r\,\om_r \cos(2\om)-(1-\mu)\,\la_t\,r\,\th_r
         \cos(2\th)
         \\
       &< -2\e +\tfrac\e4-\la_t\,\sin(2\vr^\mu) +\tfrac\e4+\tfrac\e
       4
                 \\
       &< -2\e +\tfrac\e 4-\la_t\,\sin(2\vr^\mu) +\tfrac\e2
           - \la_t \, r\,\vr^\mu_r\,\cos(2\vr^\mu)
          +\tfrac\e4
          \\
       &< -\la_t\big[\sin(2\vr^\mu)
       +r\,\vr^\mu_r\,\cos(2\vr^\mu)\big]-\e.
 \end{align*}
 Then there is $0<\rho_0<\rho_1$ such that
 if $0<r<\rho_0$,
 $$
 \vr^\mu_t
 < -\la_t\big[\sin(2\vr^\mu)
       +r\,\vr^\mu_r\,\cos(2\vr^\mu)\big]
       +\cO(r),
 $$
 where $\cO(r)$ is from lemma~\ref{Ltang}. Then lemma~\ref{Ltang}
 finishes the proof.
 \end{proof}

\subsection{The transversality condition.}\label{transv}\quad

  The equation for the dynamics under $d\psi_t$
 of subspaces along the periodic orbit $\Ga(t)$ is~\eqref{dbe}.
 Then equation~\eqref{tht} just says that the limit direction $\oth$ of the
 surface of section at $\Ga$
turns more slowly than its iteration under the linearized Reeb flow.
Condition~\ref{dcsss}.\eqref{css3} in the definition of complete systems
implies \eqref{tht}.

We check in \S\ref{transvb} equation \eqref{tht}  for surfaces of section which 
are projections of pseudo holomorphic curves in a symplectization,
and in \S\ref{TBS} for Birkhoff annular surfaces of section.
The surfaces of section that we use 
in theorem~\ref{TC} are
obtained by topological surgery from Birkhoff annuli.
These surgeries mantain inequality~\eqref{tht}. 

Since a surface of section is transversal to the Reeb flow in its interior,
a weak inequality equation~\eqref{tht}
must hold at a rotating  boundary orbit.  
If needed one can  modify the surface nearby its
boundary binding orbit so that the asymptotic rotation of the surface in the boundary
is uniform with respect to the rotation of the flow, satisfying \eqref{tht}.

\subsection{The transversality condition for a Birkhoff annulus.}\label{TBS}\quad

 Let $M$ be a closed oriented riemannian surface, 
 $SM$ its unit tangent bundle and let $\phi_t:SM\hookleftarrow$
 be the geodesic flow of $M$.
 Let $\la$ be the Liouville form on $TM$:
 $$
 \la_{(x,v)}(\xi)
 =\langle v, d\pi(\xi)\rangle_x, 
 \qquad (x,v)\in TM,\quad \xi\in T_{(x,v)}TM.
 $$
 Let $\V=\ker d\pi$, and $\H=\ker K$ be the {\it vertical} and {\it horizontal} subspaces, where 
 \linebreak
 $K:TTM\to TM$ is the connection.
 The subbundle $\cN=\ker\la$ of  of $T(SM)$,
 \begin{equation}\label{cN}
 \cN=\ker\la=\{\,(h,w)\in T_{\theta}SM\subset\H\oplus\V\;|\;\langle h,\theta\rangle_{\pi(\th)}=0\;\},
 \end{equation}
 is invariant under the linearized geodesic flow $d\phi_t$, which is given by
 $$
 d\phi_t(J(0),\dot J(0))= (J(t),\dot J(t)) \in \cN\subset\H\oplus\V,
 $$
 where $t\mapsto J(t)$ is a Jacobi field along a geodesic 
 $c(t)=\pi J(t)$  which is orthogonal to $\dot c(t)$.
 The tangent space to the unit tangent bundle $SM$ is 
 $$
 T_\th SM=\langle \X(\th)\rangle\oplus\cN(\th),
 $$
  where $\X$
 is the geodesic vector field.

 \medskip

Let $\cJ:T_xM\hookleftarrow$ be the rotation of angle $+\frac\pi2$.
Given   a simple closed geodesic $\ga(t)$ parametrized with unit speed,
define its {\it Birkhoff annulus} by
$$
A(\dga):=\{\,(x,v):\exists t,\; x=\ga(t),\; \langle v,\cJ\dga(t)\rangle_{\ga(t)}\ge 0\,\}
$$
The interior of $A(\dga)$ is tranversal to the geodesic flow.
The  tangent space of the Birkhoff annulus at a boundary
 point $\pm\dga(t)$, is generated by the geodesic vector field 
 $\X(\dga(t))$  and the vertical direction $\V\cap \cN$.

   In order to obtain the  transversality condition
 $ \oth_t<-\la_t\sin(2\oth)$ it is enough to show that the 
 (vertical) limit tangent space of the Birkhoff annulus moves slower than the 
 movement of the vertical subspace under the derivative of the flow.
 This is done as follows:

 Let $J(t)\in T_{\ga(t)}M$ be an orthogonal Jacobi field.
 Since both $J$ and  $\dot J$ are  multiples of the orthogonal vector $\dga(t)^\perp$
  they can be regarded as scalar quantities. 
  When $(J,\dot J)$ is not horizontal, i.e. when $\dot J(t)\ne 0$,
   define $W(t)=J(t)/\dot J(t)$.
  From the Jacobi equation
  $$
  \ddot J + K\, J =0    \qquad \text{and} \qquad J = W\, \dot J,
  $$
 we get
 $$
 \dot J =\dot W \dot J - W KJ.
 $$
 Replacing $J= W \dot J$ when $\dot J\ne 0$ one obtains the Riccati equation
 \begin{equation}\label{riccati2}
 \dot W= K W^2+1.
 \end{equation}
 A solution $W(t)$ is the slope of the iteration under  $d\phi_t$ of a linear subspace, i.e. if
 $$
 \cW_0=\graph W(0)=\{(W(0) v,v)\in \H\oplus \V\,|\, v\in\H\}
 $$
 then $d\phi_t(\cW_0)=\graph(W(t))$. The subspace $\cW_0$ is the vertical subspace $\V$ precisely when $W(0)=0$.  In this case, from~\eqref{riccati2}  we have that $\dot W(0)= 1$. 
 If $V(t)$ is the slope of the vertical subspace $\cN\cap\V$, then $V(t)\equiv 0$ and $\dot V(0)=0$. This means that the iteration $W(t)$ of the vertical subspace under the linearized geodesic flow moves faster than the vertical subspace $V(t)$ (tangent to the Birkhoff annulus).

 \subsection{The transversality condition for finite energy surfaces. }
 \label{transvb}\quad

 In this section we prove that condition~\eqref{tht} holds for projections on $\SS^3$
 of pseudo holomorphic curves in the simplectization of a tight contact form on $\SS^3$.
 Then we can apply item~\eqref{B3} of  theorem~\ref{TB} to the complete system
 of surfaces of section of genus 0 obtained by Hofer, Wysocki, Zehnder in 
 \cite[Cor. 1.8]{HWZ1}, in order to obtain 
 Corollary~\ref{Ctight}.

 In this case the complete system is given by the
 rigid surfaces of the finite energy foliation. 
 Let $\Si$ be a rigid surface and let $\Ga$ be a boundary 
 periodic orbit of $\Si$ where the foliation is radial.
 The equation for the dynamics under $d\psi_t$
 of subspaces along a periodic orbit $\Ga(t)$ is~\eqref{dbe}.
So we want to prove that the limit direction $\oth$ of the
surface of section $\Si$ at $\Ga$
turns slower than its iteration under the linearized Reeb flow.

Recall that the contact structure $\xi$ is invariant under the
Reeb flow $\psi_t$. The linearized Reeb flow on $\xi$ satisfies
$v(t)=d\psi_t(v(0))$ where $v(0)\in\xi$ and
\begin{equation}\label{fder}
\dot v=DX(\psi_t(\Ga(0))) \cdot v = S(t) \,v.
\end{equation}
Here the matrix $S(t)=DX(\Ga(t))$ is symmetric on symplectic
linear coordinates in $\xi(\Ga(t))$ and $v(t)=\zeta(t)\,u_{\be(t)}$
satisfies~\eqref{dbe} and~\eqref{dzeta}.

From theorem 1.4 in~\cite{HWZ96} (where $S_\infty = -J_0\, S(t)$ and $J_0=J|_\xi$),
there is a periodic vector $e(t)=\varepsilon(t)\,u_{\oth(t)}\in\xi(\Ga(t))$
in the asymptotic direction of the rigid 
surface $\Si$ which satisfies the
(eigenvalue) equation
\begin{equation}\label{eigenval}
\dot e = S(t)\, e(t)+\mu\,Je
\end{equation}
with $\mu<0$ and $J:\xi\hookleftarrow$ an almost complex structure on $\xi$.

Our choice of coordinates~\eqref{coords} about $\Ga(t)$
is symplectic and the almost complex structure can be taken $J(x,y)=(-y,x)$
in these coordinates. Comparing equations~\eqref{fder} and~\eqref{eigenval}
at an initial condition for $v(t)$ such that $v(t_0)=e(t_0)$ we get that the rigid surface
$\Si$ turns slower than the linearized flow.

Indeed, in polar coordinates $r\,u_\be=v(t_0)=e(t_0)=\e\,u_\oth$, from~\eqref{fder}, at $t=t_0$ we have that
\begin{align*}
S(t_0)\,e &= \dot v = \dot r\,u_\be+r\,\dot \be\,u_\be^\perp.\\
\dot e &= \dot\e \,u_\oth+\e\,\dot\oth\,u_\oth^\perp,\\
&=S(t_0)\,e+\mu\,Je,
\end{align*}
where $Je= \e\, u_\oth^\perp$. In the component $u_\oth^\perp$ these equations are
$$
\e\,\dot\oth=\e\,\dot\be+\mu\,\e.
$$
From~\eqref{dbe}, $\dot\be=-\la_t\,\sin 2\be=-\la_t\,\sin2\oth$. Therefore
$$
\dot\oth=-\la_t\,\sin2\oth+\mu,
$$
with the (constant) eigenvalue $\mu<0$. This implies~\eqref{tht}.

 \subsection{The return map.}\quad

  Lemma~\ref{L>} gives an isotopy of the local surface of section $\Si_0$
  by surfaces $\Si_\mu$, 
  \linebreak
  $\mu\in[0,1]$, which are transversal to
  the vector field. Then the return map of the Reeb flow
  to the final surface $\Si_1$ is topologically conjugate
  to the return map to the surface
   $\Si_0$ in a neighbourhood
  of its boundary.
  Moreover,  the intersections $\Si_1\cap (\re^2\times\{t\})$
  are included in a leaf of the radial invariant foliation $\cF$.
  This implies that the return map to $\Si_1$ near the boundary $\partial\Si_1$
   preserves the
  foliation of $\Si_1$ given by the sections
  $\langle\Si_1\cap (\re^2\times\{t\})\rangle_{t\in \SS^1}$.
  The surface $\Si_1$ is parametrized by
  $$
  G(r,t)=(r\,u_{\om(r,t)},t).
  $$
  At its boundary points $(0,t)$, the surface $\Si_1$ has a well defined
  tangent plane generated by $(u_{\om(0,t)},0)$ and the Reeb vector
  field $X=(0,0,1)$. Here $\om(0,t)=\oom(t)=\oth(t)=\th(0,t)$
  is the same angular approach of the surface $\Si_0$, which
  by~\S\ref{transv} satisfies
  $$
  \oth_t<-\la_t\,\sin 2 \oth.
  $$

  Let $P:\Si_1\to\Si_1$ be the first return map to $\Si_1$.
  Then in coordinates $(r,t)$ given by the parametrization
  $G(r,t)$ we have that
  \begin{equation}\label{Poin}
  P(r,t)=\big( a(r,t),\, b(t)\big).
  \end{equation}
  Here $b(t)$ is given by the time in which the leaf
  $G([0,\e[,t)$ returns to $\Si_1$. This is the same as the time in
  which the derivative of the flow sends the tangent subspace
  $T_{(0,t)}\Si_1$ to $T_{(0,b(t))}\Si_1$. The equations for the
   derivative of the  flow in polar coordinates are~\eqref{dbe} and
   ~\eqref{dzeta}. Then $b(t)$ is determined by the minimal $b(t)>t$
   satisfying
   \begin{align}
   \be(t)=\oth(t), \qquad
   \be(b(t))=\oth(b(t)),
   \qquad
   \dbe=-\la_t\,\sin(2\be).
   \label{ebeta}
   \end{align}
   This is the first return map $Q$ of the flow $t\mapsto (t,\be(t))$ of
   the differential equation~\eqref{dbe} to the graph of $\oth$
   (see figures~\ref{fig1}, \ref{fig2}).

   \begin{figure} 
   \includegraphics[scale=.3]{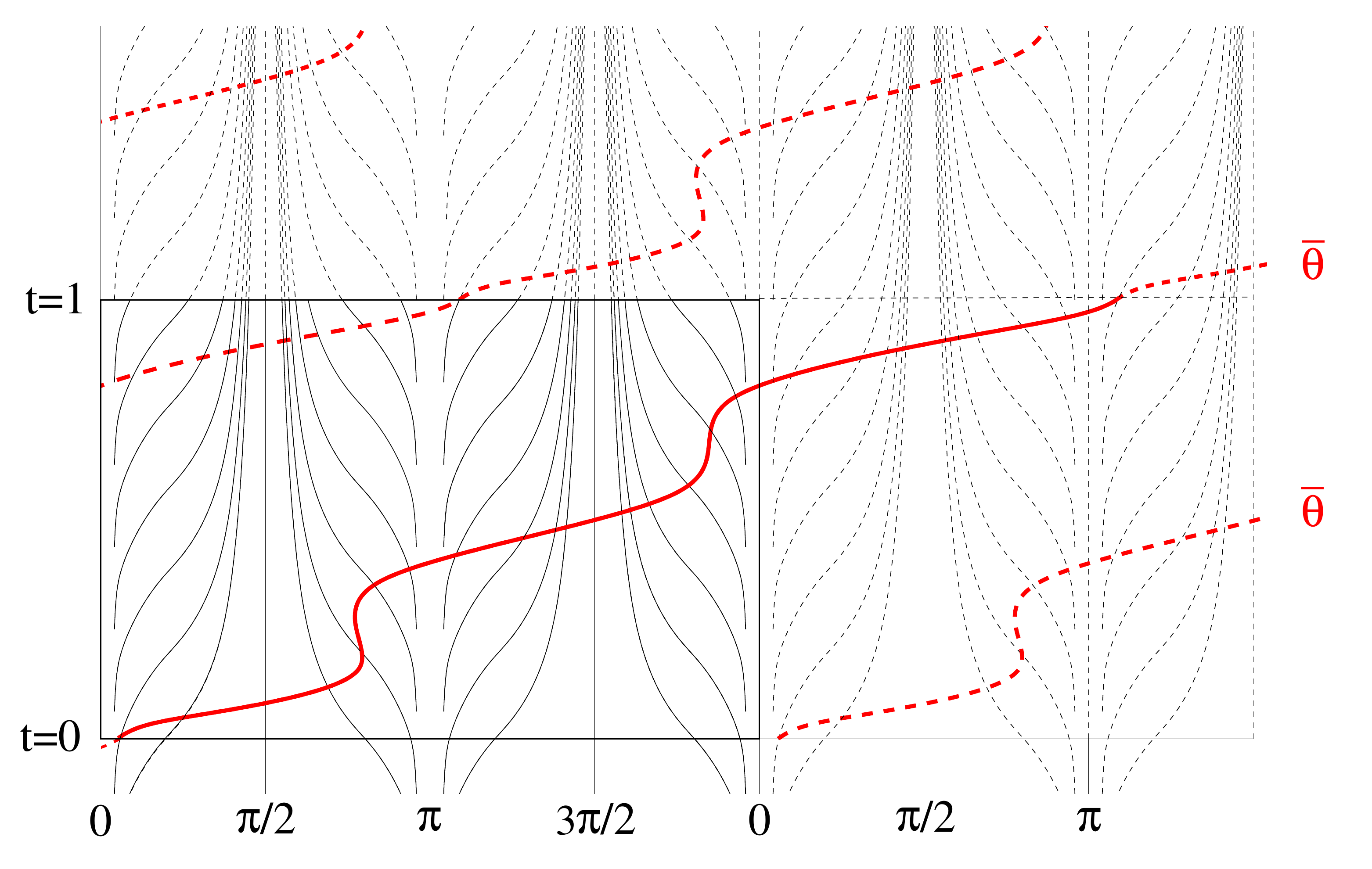}
   \caption{\footnotesize This figure shows the flow of the differential
   equation~\eqref{dbe} for $\be(t)$ describing the action of
   the derivative of the Reeb flow on 2-planes tangent to the
   periodic orbit $\ga$. It also shows the curve $(t,\oth(t))$,
   which corresponds to the movement of the limit tangent plane
   of the surface of section $\Si$ along its rotating boundary 
   periodic orbit $\ga$. The graph of $\oth$ is transversal to the
   flow of $\be$ and $\oth(t+1)=\oth(t)+{3\pi}$, but $\ov{\theta(t)}$ may not be monotonous.
   \newline 
   Observe that the first return map of the flow of~\eqref{dbe}
   to the graph of $\oth$ has  repelling periodic points at $\oth=0, \pi, 2\pi$
   and attracting periodic points at $\oth=-\frac{\pi}2,\frac{\pi}2,\frac{3\pi}2$.}
    \label{fig1}
  
    \hskip -5cm
   \includegraphics[scale=.25]{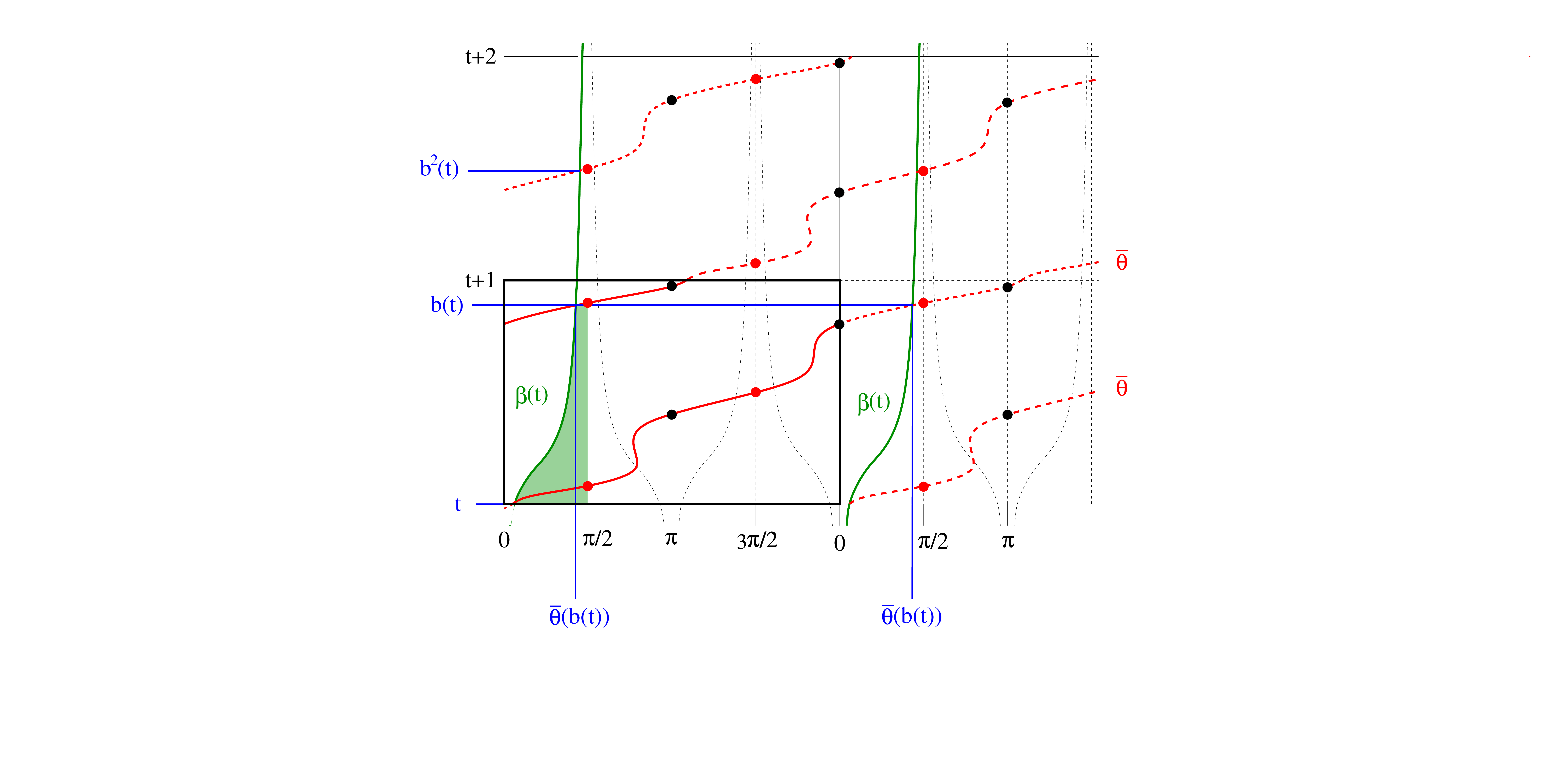}
   \vskip -7.5cm \hskip 9cm
   \includegraphics[scale=.35]{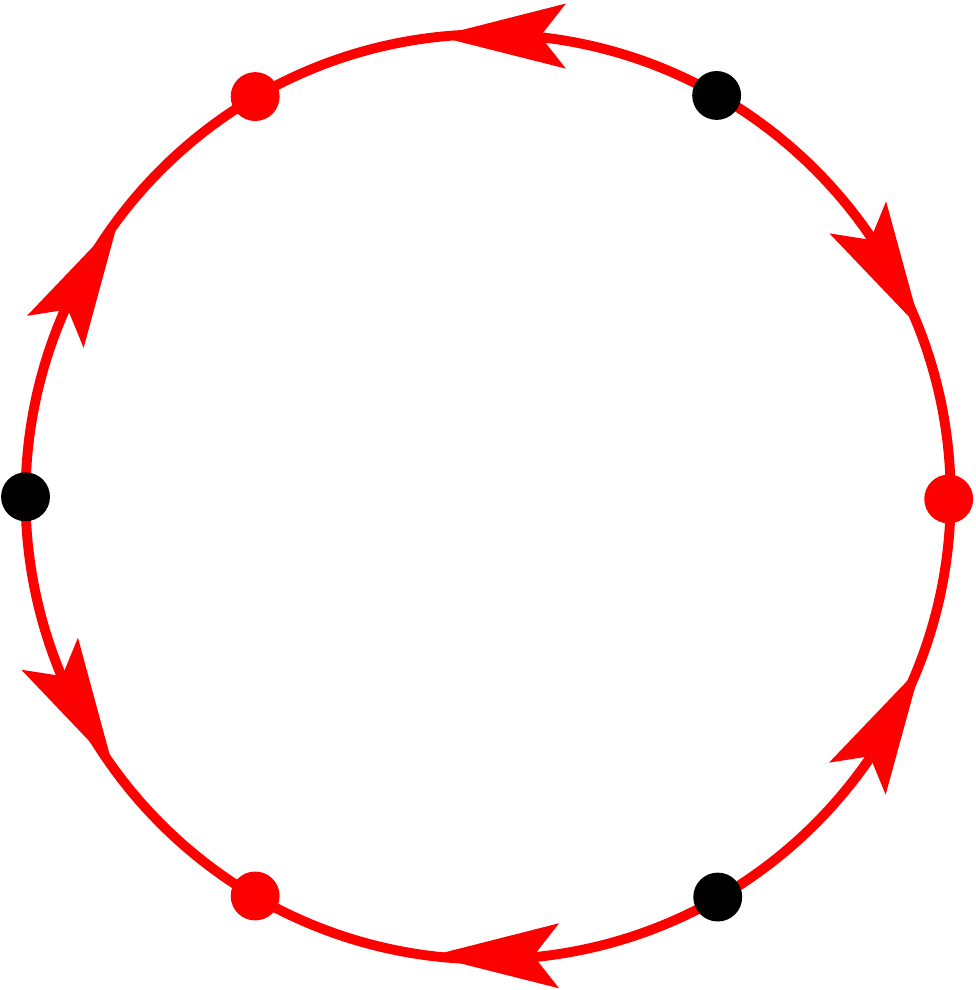}
   \vskip 2cm
      \caption{\footnotesize The figure at the left shows that the points $(t,\oth)$
   in the graph  $\cG(\oth)$ of $\oth$ with $\oth\notin \frac\pi 2 \Z~(\text{mod }3\pi)$ 
   are wandering under the return map. The figure at the
   right represents the dynamics of the return map $t\mapsto b(t)~(\text{mod }1)$.}
   \label{fig2}
   \end{figure}

   \begin{figure}[h]
   \includegraphics[scale=.25]{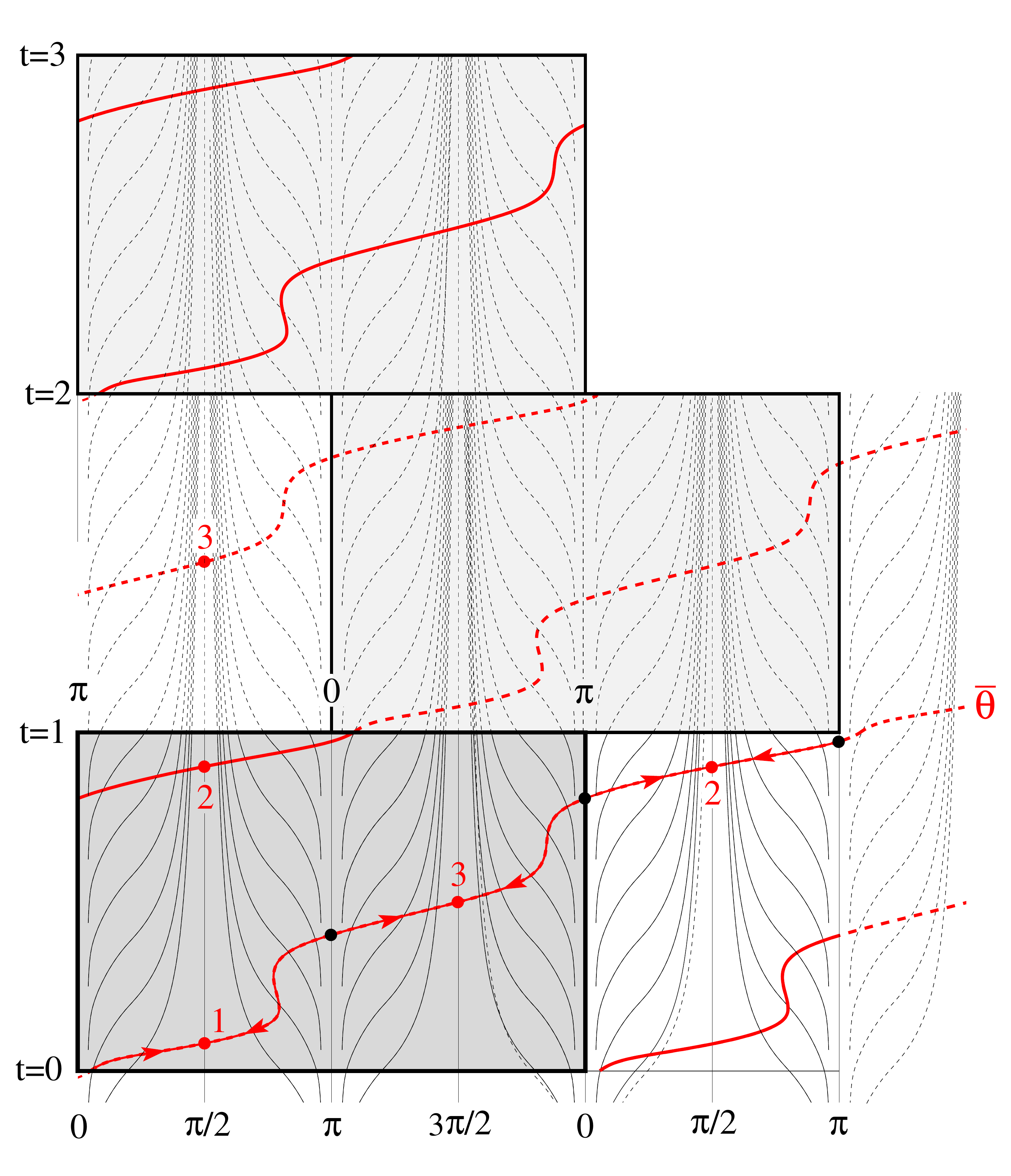}
   \caption{The dynamics at the boundary.
     }
   \label{flowv}

  \bigskip
   
     \parbox{13cm}{ 
      \footnotesize 
      \qquad
   The figure shows the dynamics of the extension to $\partial\Si$
   of the return map $P$ to the surface of section $\Si$ in a neighborhood
   of a rotating boundary component $\Ga\subset\partial \Si$, 
   when  the periodic orbit $\Ga$ is hyperbolic. 
   Here the map $P$
    has two periodic points in $\Ga=\partial\Si$ of
   period 3 which are saddles on $\Si$. The periodic points correspond 
   to the times in which the stable and unstable subspaces intersect the
   tangent space of the section $\Si$ at its boundary.

   \qquad 
   The vertical axis is the time parameter and the periodic orbit $\Ga$ which
   is supposed to have period 1. The three shadowed rectangles are copies 
   of the 2-torus formed by the periodic orbit (the time circle) and the one
   dimensional subspaces orthogonal to the periodic orbit, parametrized
   by their angle with one branch of the stable subspace.
   The movement of the stable subspace $\E^{s}$ is represented 
   by the angles $0$ and $\pi$ and the unstable subspace $\E^u$ by
   the angles $\frac\pi 2$ and $\frac{3\pi}2$.
     
     \qquad
        The periodic orbit has negative eigenvalues, then after one period
   the normal subspaces are identified by a rotation of angle $\pi$. 
   For example, the 0 branch of the  stable subspace $\E^{s}$ is
   identified with the $\pi$ branch of $\E^{s}$. This can be seen in the
   picture as a shift of length $\pi$ in the second shadowed square.
   The black lines are the dynamics of linear subspaces orthogonal to $\Ga$
   under the derivative of the flow, we will call it the projective flow.
   The subspaces converge to the  unstable
   subspace $\E^{u}$ in the future and to $\E^{s}$ in the past.
   
   \qquad
   The transversal lines are the graph of the asymptotic limit $\ov\theta(t)$ of the
   surface of section $\Si$. 
   We have assumed that this graph intersects three times $\E^s$, $\E^u$
   in one period. The dynamics of the extension 
   to the boundary in this figure is given by the return map of the projective 
   flow in the figure, to the graph of the asymptotic direction $\ov \th(t)$. The 
   periodic orbit corresponding to the unstable subspace $\E^u$ is shown in the
   figure with the numbers $1,~2,~3$, in the order of the orbit. It is an 
   attracting periodic orbit, and the stable subspace in $\partial\Si$ 
   gives a repelling periodic orbit. The extension of the return map to 
   $\partial \Si$ is Morse Smale.
   }
  \end{figure}

    The graph of $\oth$,  $\cG(\oth)=\{\,(t,\oth(t))\,|\,t\in \SS^1\,\}$, is transversal to the
    flow lines of ~\eqref{dbe}. 
    We are assuming that the angle $\ov\th(t)$ of $\Si$ turns $3\pi$ in one period
    $t\in[0,1]$. 
    The return map $Q:\cG(\oth)\to\cG(\oth)$ 
    under the flow of the differential equation~\eqref{ebeta}
    for $\be$ is the continuous extension of the return map of $\psi_t$
    to $\Si_1$ to the boundary $[r=0]\approx \cG(\oth)\approx \re/3\pi\Z\subset\partial\Si_1$.
    The return map to $\cG(\oth)$ has periodic points
    at  $\oth=0, \pi, 2\pi$ and at $\oth=-\frac{\pi}2,\frac\pi 2, \frac{3 \pi}2$. The
    periodic orbit at $\oth=0, \pi, 2\pi$ is a repellor and
    the periodic orbit at at $\oth=-\frac{\pi}2,\frac{\pi}2, \frac{3 \pi}2$
     is an attractor.
    Lemma~\ref{lnpp} shows that there are no other periodic points for the return map to
    $\cG(\oth)$.

 \begin{Lemma}\label{lnpp}
 There are no periodic points $(\tau,\oth(\tau))$ for the return map
 $Q:\cG(\oth)\to\cG(\oth)$
 $$
 Q(t,\oth(t))=\big(b(t),\oth(b(t))\big), \quad t\in\re/\Z, \quad \ov{\th}\in \re/3\pi \Z
 $$
  with $\oth(\tau)\notin\frac\pi2\,\Z$.
 The periodic orbit $\oth=0,\pi,2\pi$ is a repellor and the 
 periodic orbit 
 $\oth=-\frac\pi 2,\frac\pi 2,\frac{3\pi}2$ is an attractor.
 \end{Lemma}

\begin{proof}
Observe that since $t\mapsto \la_t$ is a 1-periodic function, the equation~\eqref{dbe} defines a 1-periodic flow $\phi_t(s,\be(s))=(s+t,\be(s+t))$ on $\re\times\re/2\pi\Z$.

Since the line $\oth=\frac\pi 2$ corresponds to the unstable subspace $\E^u$ of $\Ga$ and $t\mapsto\be(t)$ describes the dynamics of the linearized flow on 1-dimensional subspaces along $\Ga$,
we have that
\begin{equation}
\begin{aligned}\label{perio}
0<\be(0)<\pi \quad &\then \quad
\lim_{t\to+\infty}\be(t)=\tfrac\pi 2,
\\
\pi<\be(0)<2\pi \quad &\then \quad
\lim_{t\to+\infty}\be(t)=\tfrac{3\pi}2.
\end{aligned}
\end{equation}
This implies that the periodic orbit $\oth=-\frac\pi 2,\frac\pi 2,\frac{3\pi}2$  for $Q$ is an attractor. 
Similarly, the periodic orbit  $\oth=0, \pi, 2\pi$ is a 
repellor because it corresponds to the  stable subspace $\E^{s}$.

Suppose that there is a periodic point $(\tau,\oth(\tau))$ of the return map $Q$
with $\oth(\tau)\notin-\frac\pi 2 \Z$.  Then there is
$n\in\Z^+$ such that $Q(\tau,\oth(\tau))=(\tau+n,\oth(\tau))$. The solution $\be$ of~\eqref{dbe}
with $\be(\tau)= \oth(\tau)$ satisfies $\be(\tau+n)= \oth(\tau)$ and hence it is $n$-periodic.
This contradicts~\eqref{perio}.
\end{proof}

 \begin{figure}[h]
   \includegraphics[scale=.45]{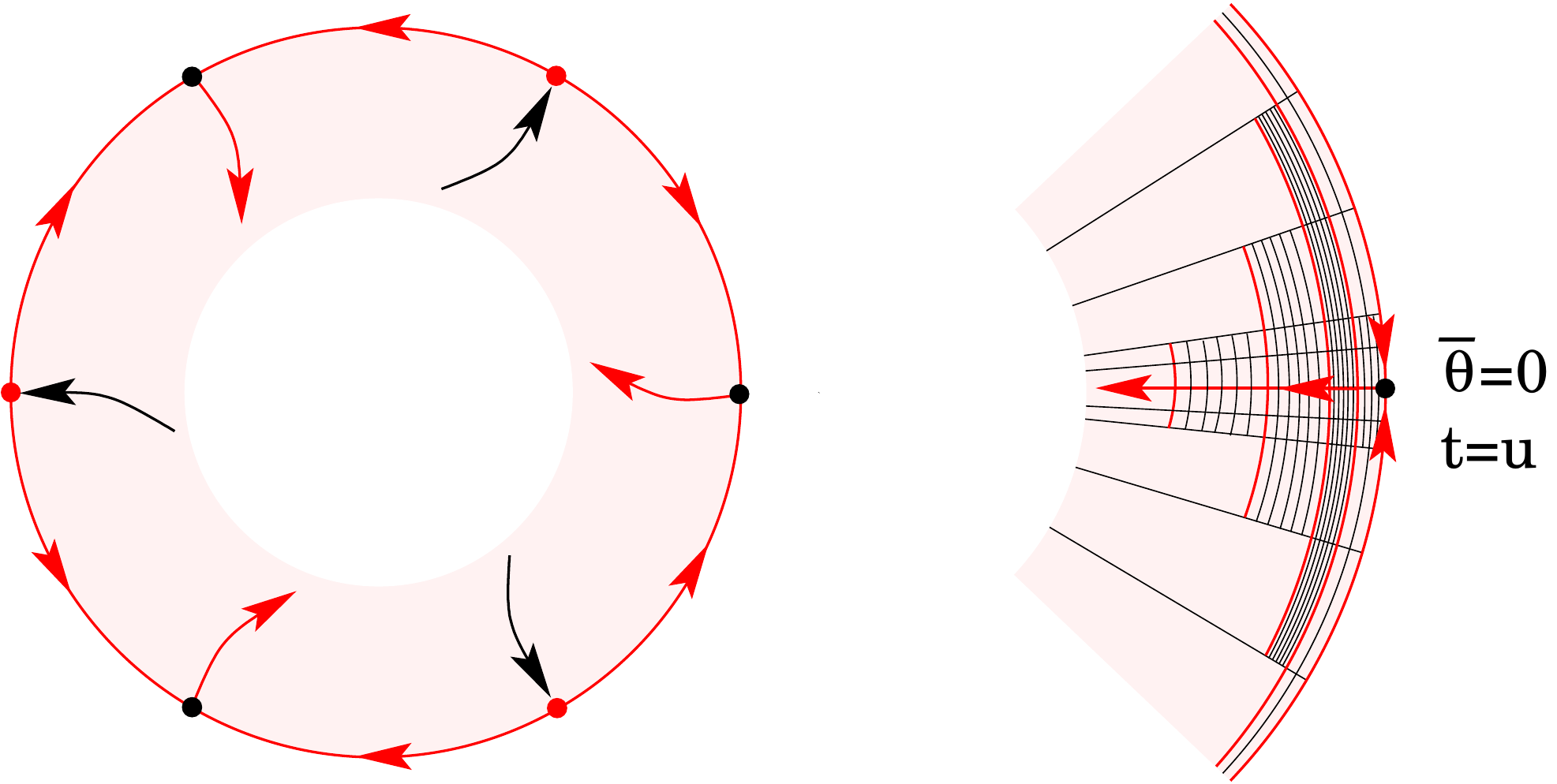}
   \caption{\footnotesize The figure at the left shows the dynamics of the return map $P$
   to the surface of section $\Si_1$. The figure at the right illustrates the construction of the
   conjugacy to a hyperbolic periodic point.}
   \label{sec1}
   \end{figure}

\quad
\newpage

 \begin{Proposition}\quad
 
The periodic points at a hyperbolic rotating boundary of a surface of section are saddles
for the return map.
 \end{Proposition}

\begin{proof}

 In a neighbourhood of the periodic orbit $\Ga$, the foliation of the surface $\Si_1$
 \linebreak
  whose
 leaves are $G(]0,\e[,t)$, $t\in[0,1]$ is invariant under the return map $P:\Si_1\to\Si_1$.
 Let 
 \linebreak
 $u,s\in[0,1]$ be such that $\oth(u)\in \{0,\pi\}$ and $\oth(s)\in\{\tfrac\pi2,\frac{3\pi}2\}$.
 Then $G(]0,\e[,u)$ and $G(]0,\e[,s)$ are components of $W^u(\ga)\cap\Si_1$ 
 and $W^s(\ga)\cap\Si_1$ respectively.
 For such  $u$'s, using formula~\eqref{Poin},
 the third iterate $P^3$ of return map $r\mapsto a(r,u)\in[0,\e[$, which 
 is the dynamics   in $W^u(\ga)\cap\Si_1$, is expanding with fixed
 point $r=0$ and on the components of  $W^s(\ga)\cap\Si_1$ it is a contraction with fixed point at the boundary
 of $\Si_1$.

   \begin{figure}[h]
   \includegraphics[scale=.3]{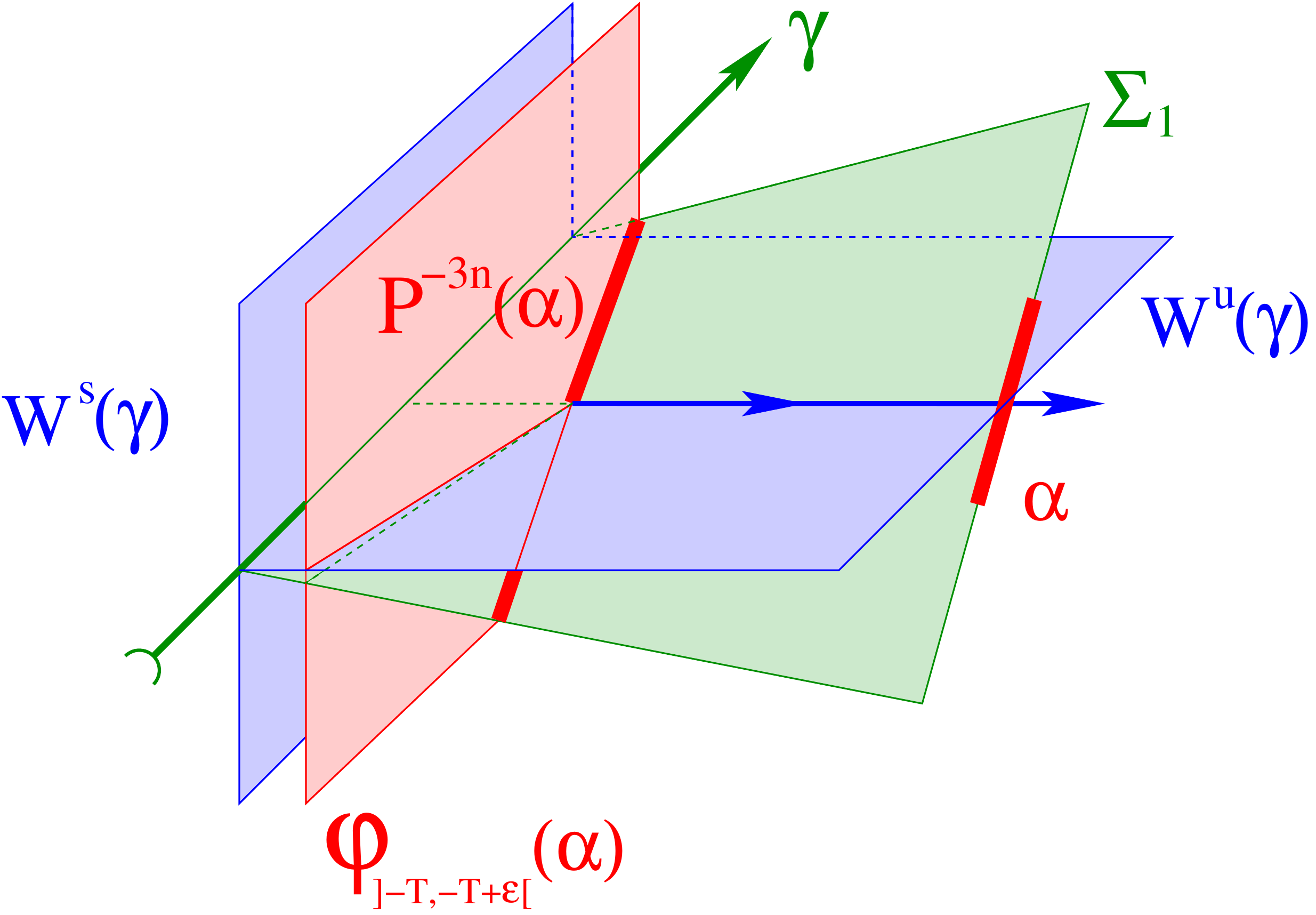}
   \caption{\footnotesize The figure shows that the iteration $P^{-3n}(\a)$ by the return map 
   to the section $\Si_1$ of a curve $\a\subset\Si_1$ which is transversal to the intersection
   $\Si_1\cap W^u(\ga)$.  The backward iteration of $\a$ under a long time $\vr_{]-T,-T+\e[}(\a)$
   is a surface which approaches the stable manifold $W^s(\ga)$. Its intersection with 
   $\Si_1$ is the return $P^{-3n}(\a)$, which in a small neighborhood of $\Si_1\cap W^u(\ga)$
   converges to the boundary $\partial\Si_1=\ga$.}
   \label{alfa}
   \end{figure}

 Consider a small curve $\a$ transversal to $W^u(\ga)\cap\Si_1$ as in figure~\ref{alfa}.  
 The inverse image $P^{-3}(\a)$ intersects a larger set 
 of leaves of the $P$-invariant foliation $\F=\cF\cap\Si_1$,
 this depends only on the dynamics of $b(t)$.
 Extend $\{\a,\,P^{-3}(\a)\}$ to a 1-dimensional foliation $\A$ on $\Si_1$ 
 between $\a$ and $P^{-3}(\a)$.
 By the $\la$-lemma, the backward flow $\phi_{-T}(\phi_{[0,\e]}(\a))$ of $\phi_{[0,\e]}(\a)$
 approaches in the $C^1$
 topology to the stable manifold $W^s(\ga)$.
 The intersection of $\phi_{-T}(\phi_{[0,\e]}(\a))$ with $\Si_1$ are 
  leaves of $P^{-3n}(\A)$, which  approach the boundary of $\Si$.
 Extend the foliation by iteration to a neighbourhood  $\cup_{n\in{\mathbb N}} P^{-3n}(\A)$ of the
 fixed point  at the boundary $r=0$, $t=u$. Use the foliations $\F$ and $\A$  as in figure~\ref{sec1},
  to construct a coordinate system in a neighbourhood of the fixed point $r=0$, $t=u$ which conjugates the dynamics to
 two sectors of a saddle fixed point. A similar construction can be made in a neighbourhood
 of the periodic points $r=0$, $t=s$.

\end{proof}

\section{The complete system for geodesic flows.}
\label{sbb}

    \begin{figure}
   \includegraphics[scale=.5]{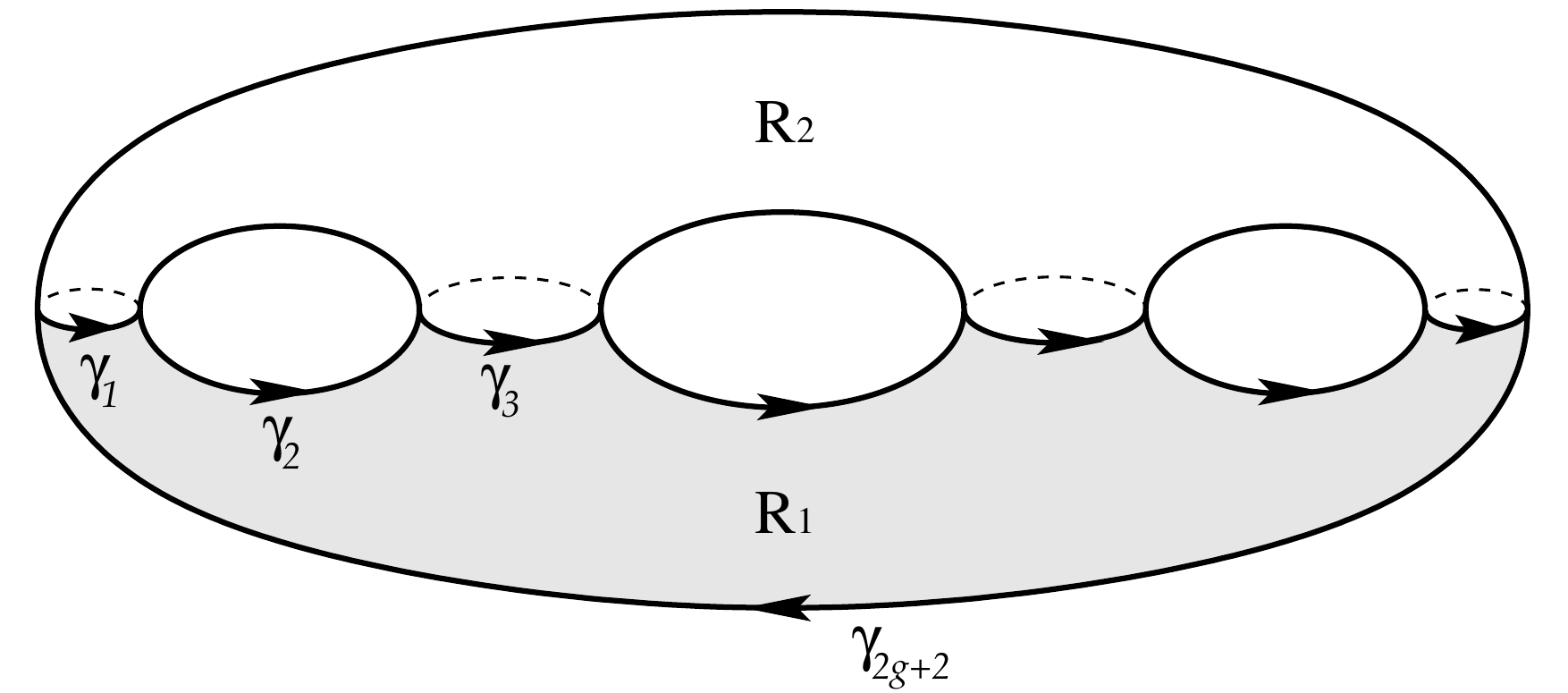}
   \caption{\footnotesize The minimizing geodesics in the homotopy classes of $\ga_i$
   separate $M$ into four simply connected regions $R_j$ and the smoothing of 
   their Birkhoff annuli 
   $$
   \cT=\partial R_1\cup\partial R_3, \qquad -\cT =\partial R_2\cup \partial R_4,
   $$ 
   are two embedded surfaces of section of genus 1. The orbits $\dga_i$, $-\dga_i$ are 
   simply covered rotating boundary 
   orbits for $\cT$ and $-\cT$. }
   \label{surface}
   \end{figure}

The set of ideas in this section descent from G. Birkhoff, notably
\cite{Birk3} section \S28.
By using an orientable double cover of $M$ if necessary, 
for theorem~\ref{TA} it is enough to assume that the surface $M$
is orientable.

We denote $SM=\{(x,v)\in TM : \rho(v,v)=1\}$ the unit tangent bundle,
$\pi:SM\to M$ the projection, $\phi_t$ the geodesic flow on $SM$
and $SA = \pi^{-1}(A)\cap SM$ for every $A\subset M$.

\subsection{Two surfaces of section of genus 1.}
\label{genus1}
\quad

Let $\ga_1,\ldots, \ga_{2g+2}$ be 
minimizing geodesics in the homotopy classes
of the curves shown in figure~\ref{surface}. 
We show now that they  divide the surface $M$ into four regions
$R_1,\ldots,R_4$ which are simply connected.

A bigon is a simply connected open subset of $M$ whose
boundary is two geodesic segments.
Two minimizing geodesics in their homotopy classes  can not form a bigon.
\linebreak
Then they must have minimal intersection number in their homotopy classes
\linebreak
c.f. \cite[Prop.~1.7]{FM}.
 Therefore 
 $$
 |\#(\ga_i\cap\ga_j)|=\de_{i,j-1}+\de_{i,j+1}
\quad  \text{ if }\quad i\ne j.
 $$ 
 Now $M\setminus(\ga_1\cup\ga_3\cup\cdots\cup \ga_{2g+1})$ is the union of
two surfaces $N_1$, $N_2$ of genus zero with $2g+1$ boundary components.
The segments $\ga_{2i}\cap N_j$ are curves connecting the boundary components
 $\ga_{2i-1}$ and $\ga_{2i+1}$. They form two simple closed curves bounding
 two simply connected regions $R_{2j-1}$, $R_{2j}$. 

Let $\cJ:TM\to TM$ be a linear map such that $(v,\cJ v)$ is an oriented orthonormal basis for every
unit vector $v$. 
Given an oriented simple closed geodesic $\ga$,  define the Brikhoff annulus of $\dga$
as
$$
A(\dga):=\{(x,v)\in SM\;|\;\exists t,\;  x=\ga(t),\; \langle v ,\cJ\dga(t)\rangle \ge 0\,\}.
$$
Then $A(\dga)$ is an annulus in $SM$ with boundaries $\dga$, $-\dga$ 
whose interior is  transversal to the geodesic flow. 
Because other  geodesics intersecting $\ga$ must be tranversal to $\ga$.

    \begin{figure}
   \includegraphics[scale=.25]{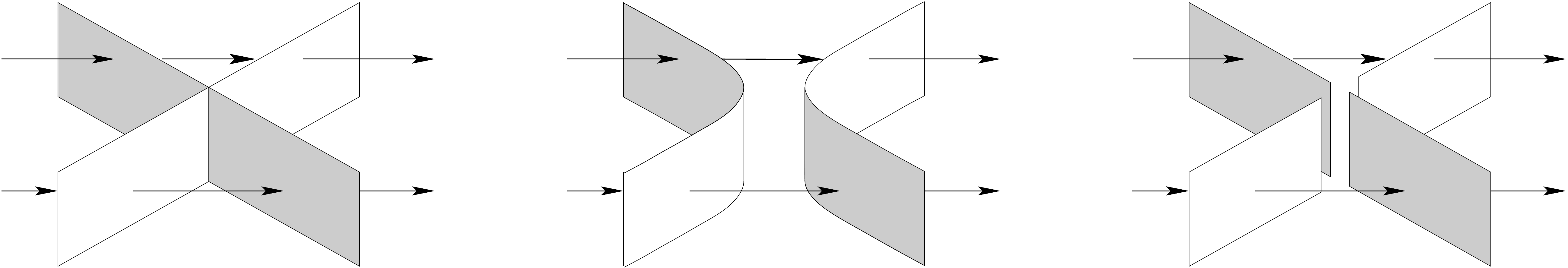}
   \caption{\footnotesize Fried surgery for a double crossing. 
   The new surface in the center does not self intersect and is transversal to the flow.
   The right figure shows that the surgery is obtained by cutting the surfaces along 
   two segments and gluing them. The gluing is uniquely determined by the contitions
   of transversality to the flow and non intersection.}
   \label{Fried1}
   \end{figure}

       \begin{figure}
   \includegraphics[scale=.25]{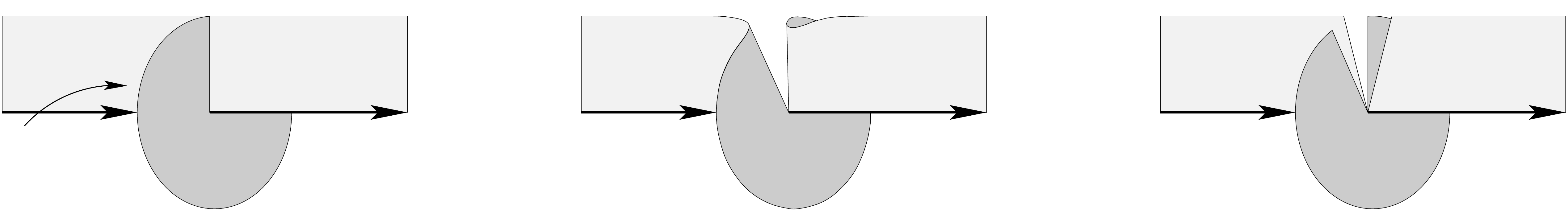}
      \caption{\footnotesize  Fried surgery for an intersection of a boundary orbit.
      The surgery is obtained by cutting along two segments and gluing. The gluing
      is uniquely determined by the flow. The resulting surface can be realized in 
      an arbitrarily small neighborhood of the original surfaces. At interior points
      it is the same surgery as in figure~\ref{Fried1}.
      }
      \label{Fried2}
   \end{figure}
  
We perform the Fried surgeries described in figures~\ref{Fried1}, \ref{Fried3}
to the collection of Birkhoff annuli $A(\dga_1),\ldots,A(\dga_{2g+2}), 
A(-\dga_1),\ldots,A(-\dga_{2g+2})$. Observe that there are not triple intersections 
of the interior of these  annuli because there are no triple intersections of their 
projected geodesics. We need to use the surgery in figure~\ref{Fried3} instead of
figure~\ref{Fried2} because the annuli $A(\dga_i)$, $A(-\dga_i)$ meet at their boundaries.

    \begin{figure}
   \includegraphics[scale=.245]{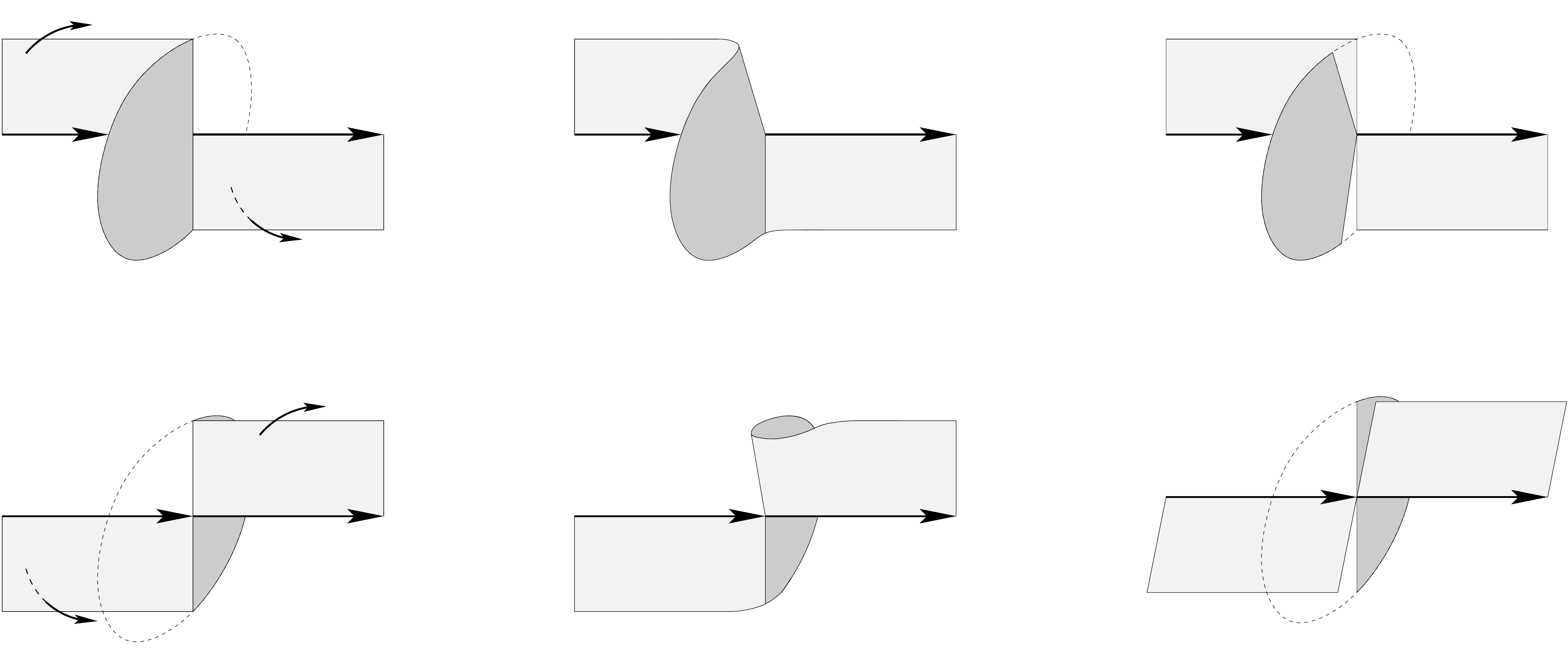}
   \caption{\footnotesize
    The Birkhoff annuli $A(\dga_i)$ and $A(-\dga_i)$ intersect at $\dga_i$ and $-\dga_i$, and the orbit 
    $\dga_i$ intersects transversely the annulus $A(\dga_{i+1})$ so it is necessary to perform the surgery in 
    figure~\ref{Fried2} twice.}
   \label{Fried3}
   \end{figure}

We prove that the result are two surfaces of section $S_1$, $S_2$, of genus~1,
each of them with the $4g+4$ 
boundary components $\{\dga_1,\ldots,\dga_{2g+2},-\dga_1,\ldots,-\dga_{2g+2}\}$. 
Observe that any orbit $\Ga$ with \begin{equation*}
\forall i \; \pi\Ga\ne \ga_i 
\quad\text{and}\quad
\pi\Ga\cap\cup_{i=1}^{2g+2}\ga_i\ne\emptyset
\end{equation*}
 intersects $S_1$ or $S_2$ transversely.

    \begin{figure}
   \includegraphics[scale=.4]{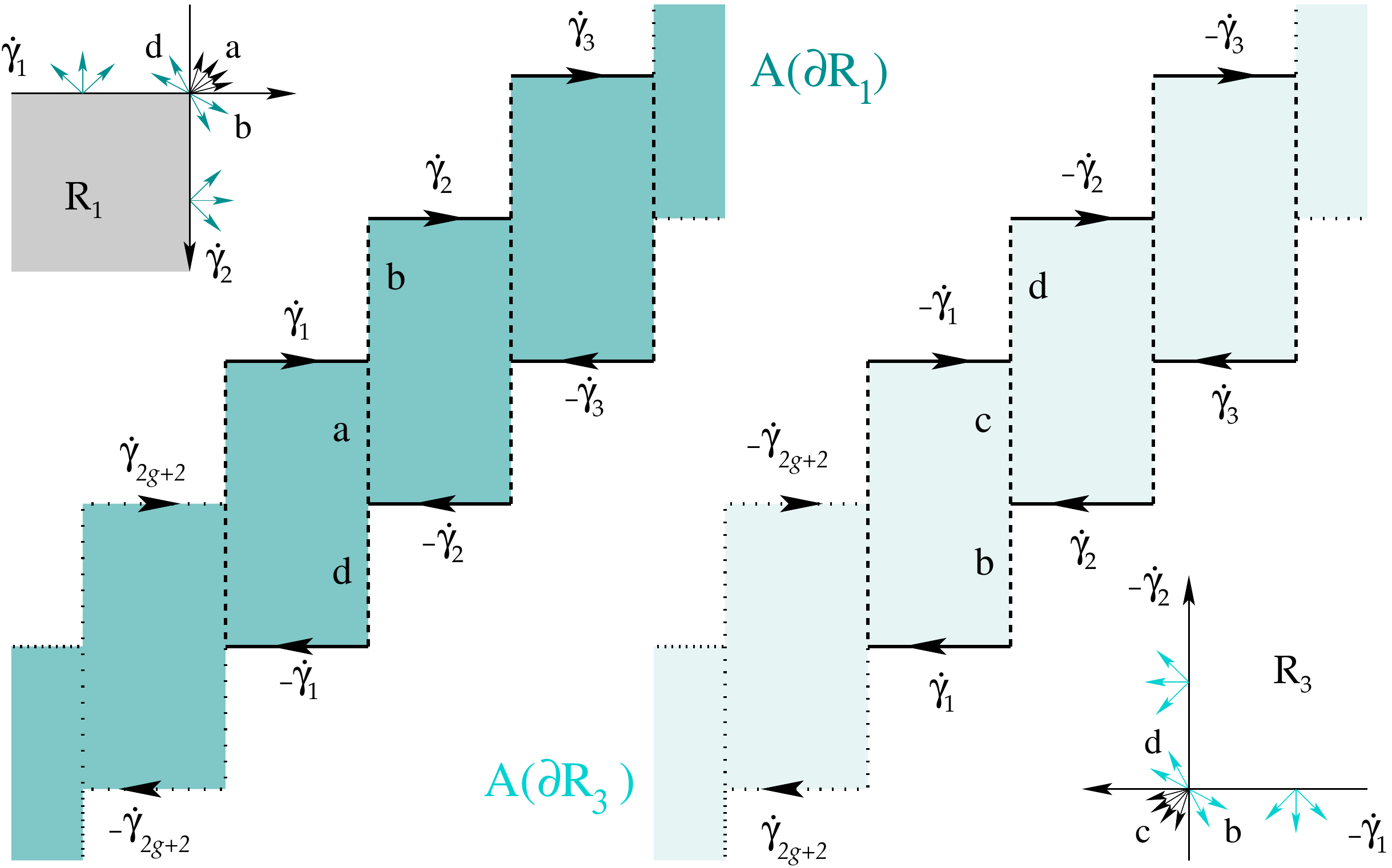}
   \caption{\footnotesize 
   The sets $A(\partial R_1)$ and $A(\partial R_3)$ are a collage 
   of half of the Birkhoff cylinders $A(\dga_i)$ and $A(-\dga_i)$ respectively.
   Both $A(\partial R_1)$ and $A(\partial R_3)$ are cylinders.}
      \label{BR1}
       \vskip 1cm
   \includegraphics[scale=.4]{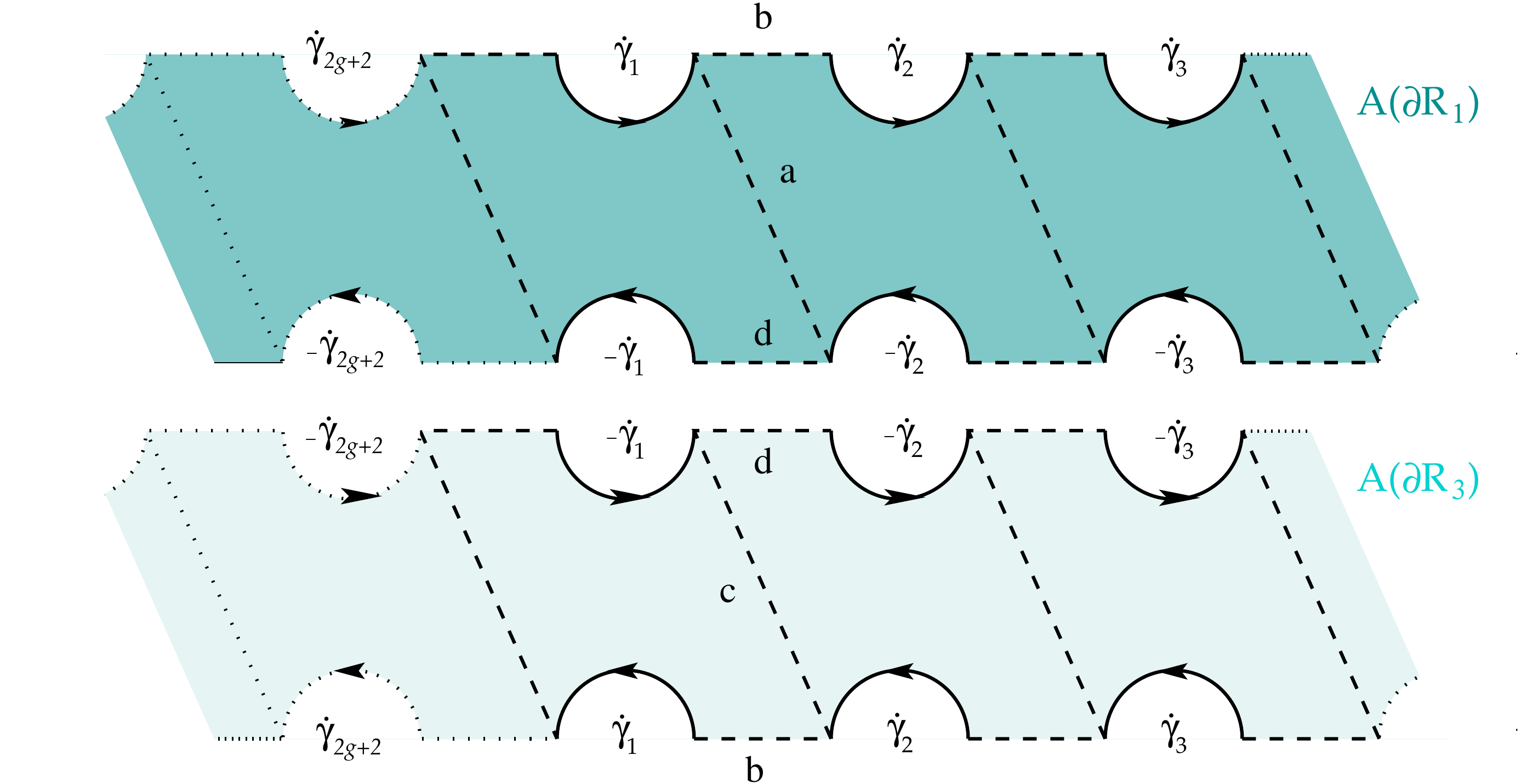}
   \caption{\footnotesize This is the same as figure~\ref{BR1} where the segments of 
   $\dga_i$, $-\dga_i$ has been drawn as curves and their adjacent vertical dashed 
   segments have been drawn horizontal. Each of the figures is a cylinder, glued at its sides.
   The two figures are glued at the horizontal dashed lines. They form a torus 
   $\cT=A(\partial R_1)\cup A(\partial R_3)$ 
   with $4g+4$ holes. The complete system contains another torus 
   $-\cT=A(\partial R_2)\cup A(\partial R_4)$ which corresponds
   to a similar construction using the regions $R_2$ and $R_4$ with the vectors opposite 
   to those of $\cT$. Both tori intersect pairwise at their boundaries.}
   \label{BR2}
   \end{figure}
   
Let $A(\partial R_i)$, $1\le i\le 4$ be the closure of the set of unit vectors based at $\partial R_i$
pointing outside of $R_i$. Then each $A(\partial R_i)$ is a cylinder whose boundary projects
to $\partial R_i$. Figure~\ref{BR1} shows the cylinders $A(\partial R_1)$, $A(\partial R_3)$
and how they are glued after performing the surgeries in figures~\ref{Fried1},~\ref{Fried3}.
Figure~\ref{BR3} shows how the surgeries of figure~\ref{Fried3} glue the segments 
$a$ and $b$ in figure~\ref{BR1}.
Then figure~\ref{BR2} is the same a figure~\ref{BR1} with the boundaries curved 
and rotated in order to show how the two cylinders $A(\partial R_1)$, $A(\partial R_3)$
glue after the surgery to form a torus $S_1$ with $4g+4$ holes.
Similarly $S_2$ is obtained from $A(\partial R_2)$ and $A(\partial R_4)$.

       \begin{figure}
   \includegraphics[scale=.45]{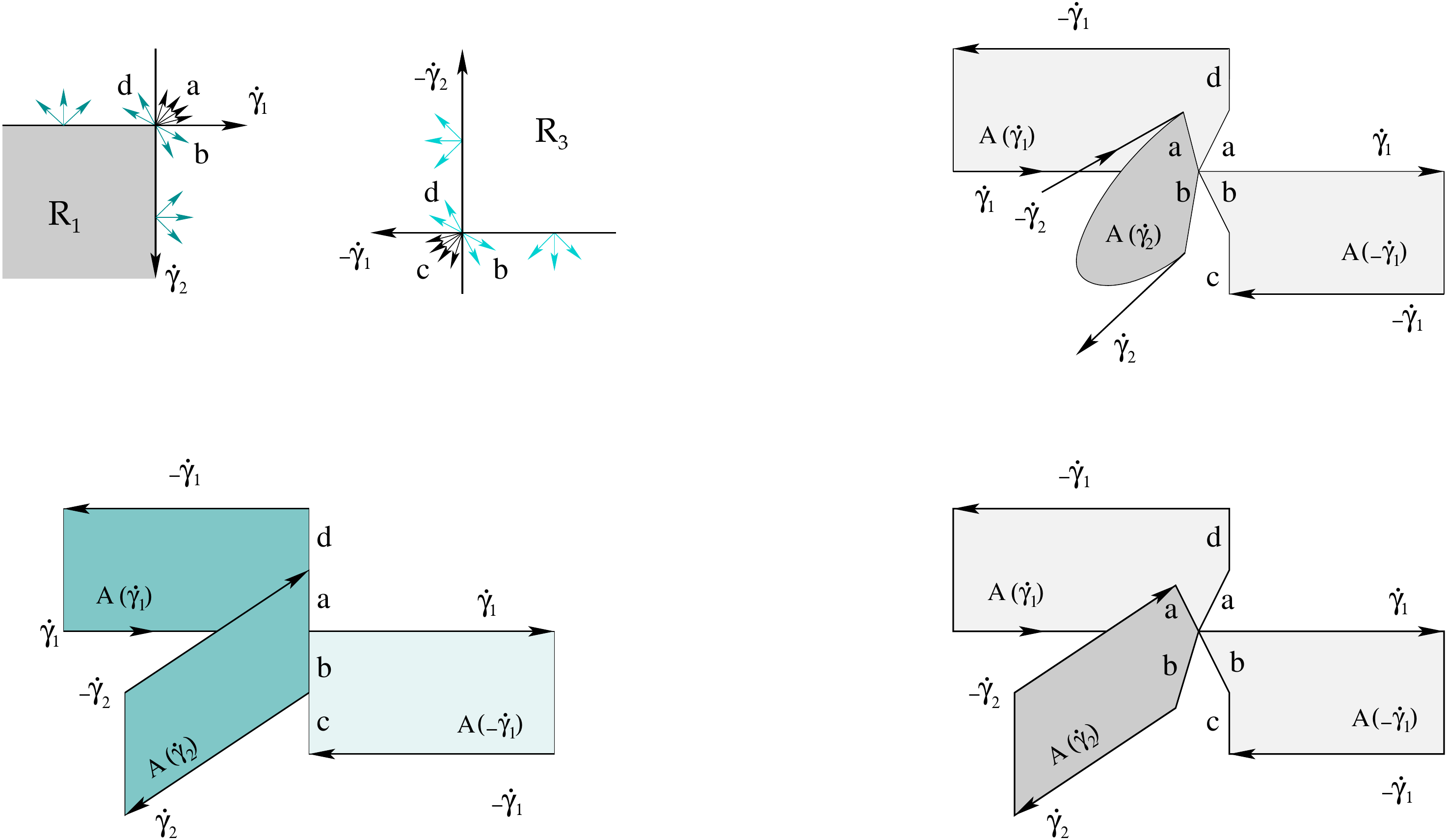}
   \caption{\footnotesize This figure shows in more detail how the annuli 
   $A(\dga_i)$, $A(\dga_{i+1})$, $A(-\dga_i)$ are glued in figure~\ref{BR1}
   after the
   surgery in figure~\ref{Fried3}.}
   \label{BR3}
   \end{figure}

       \begin{figure}
   \includegraphics[scale=.4]{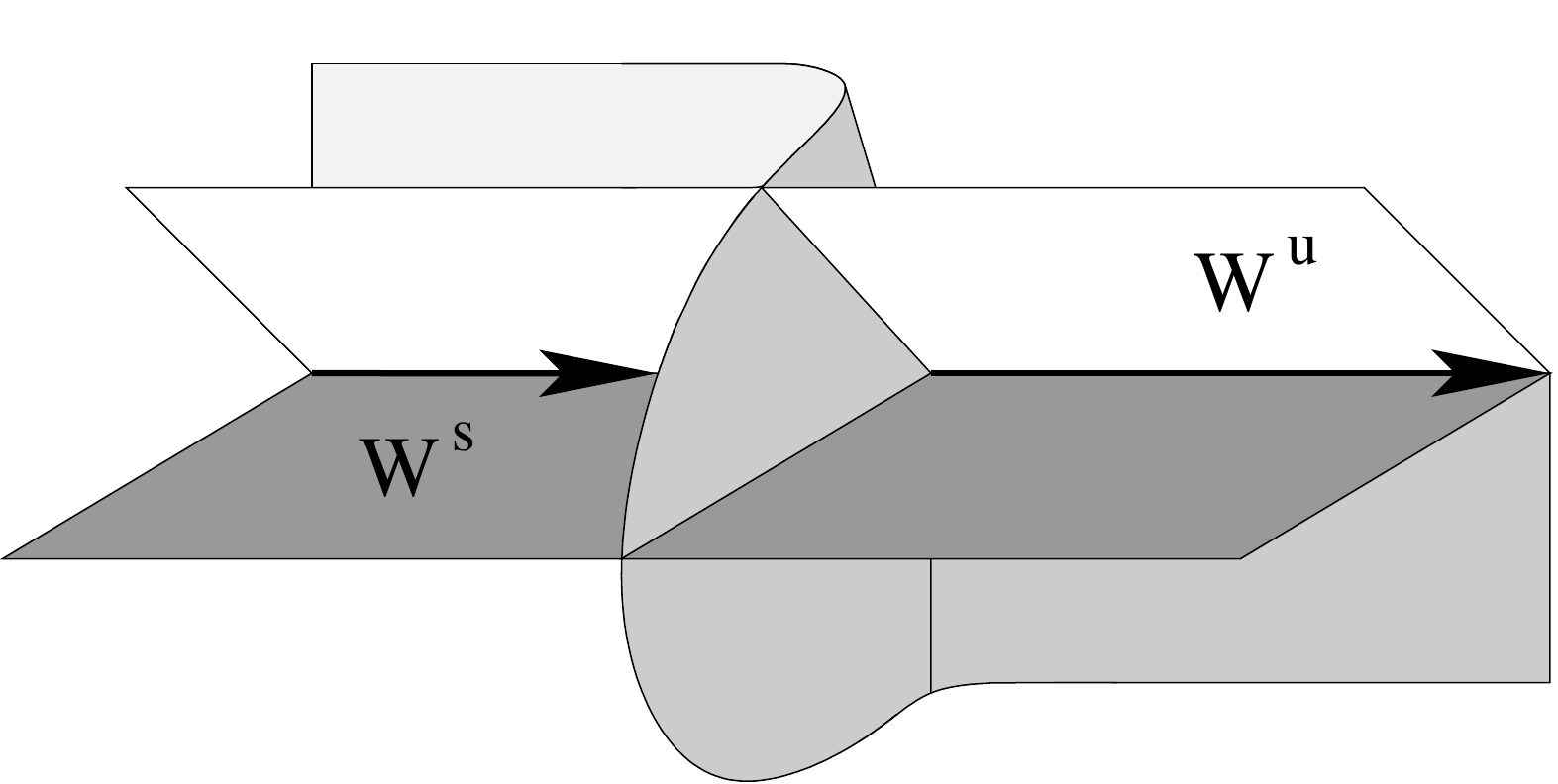}
   \caption{\footnotesize This figure shows how the local invariant manifolds
   $W^s_{loc}(\dot\ga_i)$, $W^u_{loc}(\dot\ga_i)$ intersect the surface 
   $S_j$ over the point $\ga_{i}\cap \ga_{i+1}$. The surface $S_j$
   stays near the Birkhoff annuli $A(\dga_i)$, $A(\dga_{i+1})$, $A(-\dga_i)$. The annulus 
   $A(\dga_i)$ is included in the vertical fibre $S\dga_i$. 
   The local manifolds do not intersect $S\dga_i\setminus\dga_i$
    because  $\ga_i$ 
   has no conjugate points.
   There are other intersections over the point $\ga_{i-1}\cap \ga_i$.
     }
   \label{intw}
   \end{figure}

 Now we prove that the boundary components $\dga_i$, $-\dga_i$
 are hyperbolic and that their local invariant manifolds
 $W^s_{loc}\cup W^u_{loc}$ intersect four times each
 section $S_1$, $S_2$.

 The geodesics $\ga_i$ are minimizers in their homotopy class.
 Since $M$ is a surface their multiples $\ga_i^n(t):=\ga_i(nt)$ are local
 minimizers, because a curve $\eta$ homotopic to $\ga_i^n$ contained in a small
 tubular neighborhood of $\ga_i$ can be separated into $n$ closed curves
 homotopic to $\ga_i$. Then the length $L(\eta)\ge n\cdot L(\ga_i)=L(\ga_i^n)$.
 This implies that the whole geodesic $\ga_i(t)$, $t\in\re$ has no conjugate points. 
 Since $\ga_i$ is non-degenerate, then it must be hyperbolic.
 Since $M$ is orientable $\ga_i$, is positive hyperbolic.

By section~\ref{TBS} the vertical subspace is not invariant.
Then its forward iterates $d\phi_t(V(\dga_i))$ must converge 
to the unstable subspace $E^u(\dga_i)$. But $d\phi_t(V(\dga_i))$ 
can not approach the vertical $V(\dga_i)$ because by section~\ref{TBS}
it would intersect the vertical non trivially, producing conjugate points.
Thus its limit $E^u(\dga_i)=\lim_{t\to+\infty}d\phi_t(V(\phi_{-t}(\dga_i)))$
satisfies 
\begin{equation}\label{EEV}
E^u(\dga_i)\cap V(\dga_i) =\{0\},
\quad \text{ and also } \quad
E^s(\dga_i)\cap V(\dga_i) =\{0\}.
\end{equation}

The tangent space to the Birkhoff annulus $A(\dga_i)$ is
$$
T_{\dga_i}A(\dga_i) = \langle X(\dga_i)\rangle\oplus V(\dga_i),
$$
where $X$ is the geodesic vector field. Then the invariant subspaces $E^s(\dga_i)$, $E^u(\dga_i)$
are bounded away from $T_{\dga_i} A(\dga_i)$. This implies that the local invariant manifolds
$W^s_{\e}(\dga_i)$, $W^u_{\e}(\dga_i)$ do not intersect the interior of $A(\dga_i)$.
Figure~\ref{intw} shows how each of the local invariant manifolds intersect once the surface
$S_j$ over the intersection $\ga_{i-1}\cap\ga_i$ and once more over $\ga_i\cap \ga_{i+1}$.   
By section~\ref{bdrymap} this gives four saddle periodic points for the return map 
at each boundary component $\dga_i$ or $-\dga_i$ of each surface of section $S_1$, $S_2$. 

Observe that there is $\ell >0$ and a neighborhood $N$ of $\partial S_j=\cup_i(-\dga_i \cup \dga_i)$
such that 
$$
\forall z\in N \quad \phi_{]0,\ell[}(z)\cap S_j \ne \emptyset 
\quad \&\quad
\phi_{]-\ell,0[}(z)\cap S_j \ne \emptyset. 
$$
Therefore the orbits $\pm\dga_i$ will be rotating boundary orbits for the sections $S_j$.

\subsection{Applications of the curve shortening flow.}\quad

Here we follow section \S 2 of \cite{CKMS}.
Let $(M,\rho)$ be an oriented riemannian surface.
Let $S^1=\re/\Z$. For an embedding $\ga:S^1\hookrightarrow M$, 
let $\nu_\ga$ be its positively oriented normal vector field  and let $k_\ga$ be the 
curvature of $\ga$. Let $\Emb(S^1,M)$ be the space of smooth embedded circles in  $M$
endowed with the $C^\infty$ topology. Let 
$$
L(\ga) =\int_{S^1} \lV \dga\rV_\rho
$$
be the length functional.
The {\it curve shortening flow}  is a continuous map 
$$
 \cU \longrightarrow \Emb(S^1,M), \qquad 
(s,\ga_0)\mapsto 
\Psi_s(\ga_0)=:\ga_s,
$$
defined on a maximal open neighborhood $\cU\subset [0,\infty[\times\Emb(S^1,M)$
of $\{0\}\times\Emb(S^1,M)$ by the following PDE:
$$
\partial_s\ga_s=k_{\ga_s}\nu_{\ga_s}.
$$
The following properties are proved in \cite{Grayson}, \cite{dpmms20}:
\begin{enumerate}[(i)]
\item $\Psi_0=id$ \quad and \quad$\Psi_s\circ\Psi_t=\Psi_{s+t}$ \quad  for all $s,t\ge 0$.
\item\label{p4}
$\Psi_s(\ga\circ\th)=\Psi_s(\ga)\circ\th$ \quad for all $\ga\in\Emb(S^1,M)$ and  $\th\in\Diff(S^1)$.
\item $\frac{d\,}{ds}L(\psi_s(\ga))\le 0$  for all $\ga\in \Emb(S^1,M)$, 
with equality if and only if the image of $\ga$ is a geodesic.
\item Given $\ga\in\Emb(S^1,M)$ let
$s_\ga=\sup\{\,s\ge 0\,|\,(s,\ga)\in \cU\,\}$. Then $s_\ga$ is finite if and only if 
$\Psi_s(\ga)$ converges to a constant when $s\to s_\ga$.
\end{enumerate}

A path-connected subset $U\subset M$ is {\it weakly convex}
if for any pair $x,y\in U$ that can be joined by an absolutely continuous curve in $U$
of length smaller than  the injectivity radius $\inj(M,g)$, the shortest geodesic
joining $x$ and $y$ is contained in $U$.  Another useful property of $\Psi_s$ is
that it preserves weakly convex sets, namely

\begin{enumerate}[(i)]
\setcounter{enumi}{4}
\item\label{p5} 
If $U\subset M$ is weakly convex then
$$
\ga\in\Emb(S^1,U) \quad \then \quad \forall s\in[0,s_\ga[ \quad \psi_s(\ga)\in\Emb(S^1,U).
$$
\end{enumerate}

This flow is used in \cite{CKMS}  to prove the following lemmata.

\begin{Lemma}[\cite{CKMS} lemma~2.1]\quad\label{L41}

Let $U\subseteq M$ be a weakly convex subset that is  not simply connected.
 Let $\cC\subset \Emb(S^1,U)$  be a connected component containing 
 loops that are non-contractible in $U$. Then, there exists a sequence $\ga_n\in\cC$
 converging in the $C^2$-topology to a simple closed geodesic $\ga$
 contained in $\ov U$ of length
 $$
 L(\ga)=\inf_{\zeta\in\cC}L(\zeta)>0.
 $$
 \end{Lemma}

 \pagebreak
 
 \begin{Lemma}[\cite{CKMS} lemma~2.2]\label{invconv}\quad
 
 If $U\subset M$ is  weakly convex and $K\subset SM$ is invariant by the geodesic flow
 (i.e. $\forall t\in\re$ $\phi_t(K)=K$)  and such that $\pi(K)\subset U$, then 
 any path-connected component of $U\setminus \pi(K)$ is weakly convex.
\end{Lemma}

A closed geodesic $\ga:S^1\to M$ is called a {\it waist} when any absolutely continuous curve 
$\zeta$ which is sufficiently $C^0$-close to $\ga$ satisfies $L(\zeta)\ge L(\ga)$.
By the argument before \eqref{EEV}, 
non degenerate waists are positive hyperbolic and have no conjugate points.

\bigskip
\bigskip

\begin{Lemma}\quad

A simple nondegenerate closed geodesic $\ga$ is a waist if and only if it has no conjugate points.
\end{Lemma}

\begin{proof}\quad

Suppose that $\ga$ is nondegenerate and has no conjugate points,  
we prove that it is a waist.
The converse is standard.
Consider the geodesic lagrangian $L:TM\to\re$  and hamiltonian
$H:T^*M\to\re$
\begin{equation}\label{LH}
L(x,v) =\tfrac 12 |v|_x^2, \quad 
H(x,p)=\sup_{v\in T_xM} \{p(v)-L(x,v)\},
\quad
H(x,p)=\tfrac 12 \,|p|_x^2.
\end{equation}
The Legendre transform $\cL(x,v)=\langle v,\cdot\rangle_x$ conjugates 
the geodesic flow to the hamiltonian flow of $H$ on the energy level
$H\equiv \tfrac 12$.
Also $\cL(\pi^{-1}\{x\})=\pi_*^{-1}\{x\}$ identifies the 
\linebreak
vertical fibers.
Observe that  $\ga$ must be positive hyperbolic.
Since~\eqref{EEV} holds in the hamiltonian flow
 there is a  neighborhood $U$ of $\ga$ where 
$W^s(\ga)\subset T^*M$ is a graph:
$$
T^*U\cap W^s(\ga)=\{\,(x,\om(x))\in T^*_xM:x\in U\,\}.
$$
Then $\om\in\La^1(U)$ is a 1-form on $U$ which is closed because $W^s(\ga)$
is a lagrangian submanifold. 
And  $(dx\wedge dp)|_{W^s(\ga)}\equiv 0$ 
 because $W^s(\ga)$ is tangent to the
Reeb vector field of $(H^{-1}\{\tfrac 12\},p\, dx)$.  
Since $H(x,\om(x))\equiv\tfrac 12$, equation~\eqref{LH} implies that
$$
\forall (x,v)\in TM\qquad \om(x)(v) \le L(x,v)+\tfrac 12.
$$
For $x\in\ga$, we have that $\om(x)=\cL(x,\dga)=\langle \dga,\cdot\rangle_x$. Therefore
$\om(\ga)\cdot\dga\equiv 1$.

Let $\eta$ be an absolutely continuous  curve 
$C^0$ close to $\ga$ in $U$ parametrized by arc length.
Then $L(\eta,\deta)\equiv \tfrac 12$ and 
$$
L(\ga)=\int_\ga\om=\int_\eta\om \le \int_\eta L+\tfrac 12 =L(\eta),
$$
where the second inequality holds because $\eta$ is homotopic to $\ga$ inside $U$.

\end{proof}

\newpage

We need the following min-max lemma.
These geodesic have conjugate points because minimax 
 critical points can not be local minima.

\begin{Lemma}[\cite{CKMS} lemma~{2.4}]\quad\label{L43}

Let $(M,\rho)$ be an orientable riemannian surface.
\begin{enumerate}[(i)]
\item\label{l431} If $A\subset M$ is an annulus bordered by two waists, then 
$\interior(A)$ contains a non contractible simple closed geodesic
with conjugate points.

\item\label{l432} If $D\subset M$ is a compact disk bounded by a waist, then 
$\interior(D)$ contains a simple closed geodesic 
with conjugate points.
\end{enumerate}

\end{Lemma}

\bigskip

\begin{Lemma}[\cite{dpmms20} lemma~5.9]\label{ll45}\quad

Let $(M,\rho)$ be a riemannian surface, and $\ga:[-T,T]\to M$ a geodesic arc
parametrized with unit speed whose interior $\ga|_{]-T,T[}$ contains a pair of 
conjugate points. Then there exists an open neighborhood $U\subset SM$
of $(\ga(0),\dga(0))$
such that, for each $(x,v)\in SU$, the geodesic $\zeta(t)=\exp_x(tv)$ intersets
$\ga$ for some $t\in[-T,T]$. 
\end{Lemma}

Lemma~\ref{ll45} implies the following corollary:

\begin{Corollary}\label{cco}
\quad

Let $(M,\rho)$ be an orientable riemannian surface and $\ga$ 
a simple closed geodesic with conjugate points.
\begin{enumerate}[(i)]
\item\label{cco1} There exists $T>0$ and an open neighborhood $V\subset SM$ of 
the lift $\dga$ such that, for each $z\in V$, the geodesic $\zeta(t):=\pi\circ\phi_t(z)$
intersects $\ga$ on some positive time $t_1\in ]0,T]$ and some negative time
$t_2\in[-T,0[$.
\item\label{cco2}
There exists $T>0$ and an open neighborhood $U\subset M$ of $\ga$
such that, for each $z\in SU$, the geodesic $\zeta(t):=\pi\circ\phi_t(z)$
intersects $\ga$ on some  $t\in[-T,T]$.
\end{enumerate}

\end{Corollary}

A geodesic polygon in a riemannian surface
 is a simple  closed curve which is a union of finitely many
distinct geodesic arcs that is not one closed geodesic. 
Observe that necessarily the geodesic arcs are transversal.
Therefore we have 

\bigskip

\begin{Remark}\quad\label{rpo}
\begin{enumerate}[(i)]
\it
\item\label{rpo1} If $P$ is a geodesic polygon then there exist a neighborhood $V\subset SM$ 
of the lift $\dot P$ 
and $\ell>0$ such that for every $z\in V$ and $\zeta(t):=\pi\phi_t(z)$, 
both geodesic arcs $\zeta|_{]0,\ell]}$, $\zeta|_{[-\ell,0[}$ 
 intersect $P$.
\item\label{rpo2} If $P$ is a geodesic polygon there exists a neighborhood $U\subset M$ of $P$ 
and $\ell>0$ such that for every $z\in SU$ and $\zeta(t):=\pi\phi_t(z)$, 
the geodesic arc $\zeta|_{[-\ell,\ell]}$ 
 intersects $P$.
\end{enumerate}
\end{Remark}

\bigskip

\subsection{Complementary Birkhoff annuli.}\quad

In this section we obtain a complete system of surfaces of sections 
for $(M,\rho)$ provided that all waists are nondegenerate (i.e. hyperbolic).
This is done by adding disjoint Birkhoff annuli to the surfaces obtained 
in section~\ref{genus1}. The Birkhoff annuli have genus 0, so for them
we don't need to check the condition in theorem~\ref{TB}.\eqref{B3}
on the number of intersections of the separatrices.

In fact some are Birkhoff annuli of waists which are in 
$\Kfix$ and other are Birkhoff annuli of minimax orbits
which have index 1 and are in $\Krot$. 
If these minimax orbits are hyperbolic,
then their Floquet multipliers are negative and their invariant 
subspaces $E^s$, $E^u$ intersect the vertical bundle $V=\ker d\pi$,
$\pi:TM\to M$,  twice along one period. So each local invariant manifold
$W^s(\ga)$, $W^u(\ga)$ intersects each Birkhoff  annuli $A(\dga)$, $A(-\dga)$
only once.

We prove the following.

\begin{Theorem}\label{TCBA}
\quad

Let $(M, \rho)$ be an orientable riemannian surface of genus $g$ 
with all its waists non degenerate.
There are a finite number of surfaces of section $\Si_1, \ldots, \Si_{2n}$ such that
\begin{enumerate}[(a)]
\item\label{tcba1}  
If $g=0$ then $\Si_1$, $\Si_2$ are the Birkhoff annuli of a minimax
         simple closed geodesic.
\item\label{tcba2}  
 If $g>0$, $\Si_1$, $\Si_2$ are the surfaces of genus 1 and $4G+4$ 
boundary components
described in subsection~\ref{genus1}.
\item\label{tcba3}  
$\Si_3,\ldots, \Si_{2n}$ are Birkhoff annuli of $n-1$ mutually disjoint simple closed geodesics.
\item\label{tcba4}  $\Si_3,\ldots, \Si_{2n}$ are disjoint from $\Si_1$, $\Si_2$.
\item\label{tcba5}  Every geodesic orbit intersects $\Si_1\cup\cdots\cup \Si_{2n}$.
\item\label{tcba51} Let $\Kfix$ be the union of the set of closed orbits without conjugate 
points in $\cup_{i=3}^{2n} \partial\Si_i$ and let $\Krot=\cup_{i=1}^{2n}\partial \Si_i\setminus \Kfix$.
There are $0<\ell<\infty$ and a neighborhood $\cU$ of $\Krot$ in $SM$ such that 
$$
\forall z\in \cU\quad
\phi_{]0,\ell[}(z)\cap \bSi\ne\emptyset
\quad \&\quad
\phi_{]-\ell,0[}(z)\cap \bSi\ne\emptyset,
\qquad \bSi:=\cup_{i=1}^{2n}\Si.
$$
\item\label{tcba6} If $\ga$ is a   geodesic with 
$\dga(]0,+\infty[) \cap \bSi
=\emptyset$
then $\dga(t)\in W^s(z_t)$ for some $z_t\in \Kfix$.
\item\label{tcba7} If $\ga$ is a  geodesic with 
$\dga(]\!-\!\infty,0[) \cap \bSi=\emptyset$
then $\dga(t)\in W^u(z_t)$ for some $z_t\in \Kfix$.
\end{enumerate}

\end{Theorem}

 The following proposition is proved in lemmas~3.8 and 3.7 in \cite{CKMS},
 using examples~3.2 and~3.3 in~\cite{CKMS}.
 
 \begin{Proposition}[\cite{CKMS} lemmas~3.8, 3.7]\label{pfin}
 \quad
 
 Let $(M,\rho)$ be a riemannian surface and let 
  $D\subset M$ be a simply connected open set whose boundary
 $\partial D=P$ is a geodesic polygon or a simple closed geodesic
 with conjugate points.
 Suppose that every simple closed geodesic 
 without conjugate points contained in $D$ is non-degenerate.
 Then every collection of mutually disjoint simple closed geodesics 
 contained in $D$ is finite.
 \end{Proposition}

A {\it corset} $(A,w)$ in $(M,\rho)$ is an annulus $A\subset M$ such that 
$\interior(A)$ contains a simple closed geodesic $w$ which is a waist and 
that the boundary components of $\partial A$ 
are either a polygon  or
a simple closed geodesic with conjugate points.
A {\it bowl}  is a disk $D\subset M$ whose boundary $\partial D$
is either a geodesic polygon 
or a simple closed geodesic with conjugate points.
We further require that corsets and bowls are connected components 
of the complement of finitely many geodesics.
Observe that by lemma~\ref{invconv}, corsets and  bowls are weakly convex.

\begin{Lemma}\label{lct}\quad

\begin{enumerate}
\item\label{lct1} If $(A,w_1)$ is a corset, $U=\interior A\setminus w_1$ and 
$\cap_{t\in\re}\phi_{-t}(SU)\ne\emptyset$; then there are two corsets $(A_1,w_1)$, $(A_2, w_2)$
with $A=A_1\cup A_2$ and $A_1\cap A_2 =\partial A_i\setminus \partial A$, $i=1,2$.

\item\label{lct2} If $D$ is a bowl, $V=D\setminus\partial D$ and $\cap_{t\in\re}\phi_{-t}(SV)\ne \emptyset$;
then there is a corset $(A,w)$ and a bowl $B$ such that $D= A\cup B$
and $A\cap B =\partial B =\partial A\setminus\partial D$.
\end{enumerate}

\end{Lemma}

    \begin{figure}
   \includegraphics[scale=.3]{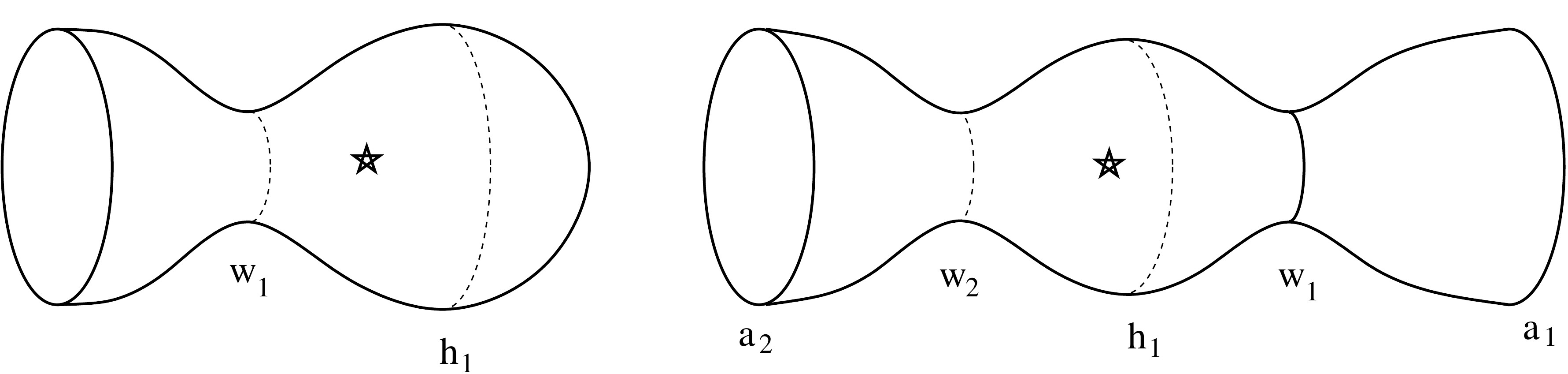}
   \caption{\footnotesize 
   These are examples of a decomposition of a bowl and a corset in lemma~\ref{lct} 
   when there is a new invariant subset projecting in its interior.
   The star $\star$ marks a point in the projection $\pi\La$ of the new invariant subset 
   and determines the homotopy class of the new waist.
  }
   \label{boset1}
   \end{figure}

\begin{proof}\quad

\eqref{lct1}. Write $\La:=\cap_{t\in\re}\phi_{-t}(SU)$.
By corollary~\ref{cco}.\eqref{cco2} and remark~\ref{rpo}.\eqref{rpo2}
there is a neighborhood $N$ of $\partial A$ such that 
 $\La\cap SN=\emptyset$.
 Since $U\cap w_1=\emptyset$ we have that 
 $\pi\La\not\subset w_1$.
 By lemma~\ref{invconv} any path-connected component of 
 $A\setminus (w_1\cup \pi\La)$ is weakly convex. 
Let $a_2$ be the connected component of $\partial A$
 which is included in  a connected component of $A\setminus w_1$
 which intersects $\pi\La$. Let $a_1$ be the other component of $\partial A$.
 Let $W$ be the connected component of $A\setminus (w_1\cup \pi\La)$ which 
 contains $a_2$. 
 Observe that $a_2$ is not homotopic to $w_1$ in $W$.

 Let $x\in\pi\La\setminus w_1$, and $\e>0$ with $d(x,w_1)>\e$.
 We claim that if a closed curve $\ga\subset W$ is homotopic to $a_2$ inside $W$
 then there is $y_\ga\in\ga$ such that $d(y_\ga,w_1)\ge\e$.
 For if 
 $$
 \ga\subset B(w_1,\e):= \{z\in A:d(z,w_1)<\e\},
 $$
  then $\ga$ is homotopic to $w_1$
 inside $A\setminus\{x\}$. Thus $\ga$ is non homotopic to $a_2$
 inside $A\setminus\{x\}$. Then $\ga$ is non homotopic to $a_2$ inside
 $W\subset A\setminus\{x\}$. A contradiction.
 Consequently, if $\eta$ is a $C^0$ limit of curves $\ga_n\subset W$ homotopic to
 $a_2$ inside $W$, then
 \begin{equation}\label{etaw1}
 \eta\ne w_1.
 \end{equation}
 
 Let $\cC$ be the connected component of $\Emb(S^1,W)$ containing
 a curve homotopic to $a_2$.
 By lemma~\ref{L41}  there is a sequence $\ga_n\in\cC$ 
 converging in the $C^2$ topology to a simple closed geodesic $w_2$ in $\ov W$
 of length $L(w_1)=\inf_{\zeta\in \cC}L(\zeta)>0$.
 Then $w_2$ is  a waist and by \eqref{etaw1},
 $w_2\ne w_1$.
 Since $w_1$, $w_2$ are waists in the same homotopy class 
 in $A$, we have that $w_1\cap w_2=\emptyset$ (c.f. \cite[Prop. 1.7]{FM}).
 Since $w_2\subset \ov W\subset A$, 
 if $w_2\cap a_2\ne\emptyset$ then $w_2$ and $a_2$ would be tangent geodesics
 (segments)
 and hence the same geodesic. But this is not possible because $a_2$ has conjugate points
 or is a polygon and $w_1$ is a waist.
  
  In the annulus  $A$ the curves 
 $a_1$, $w_1$, $w_2$, $a_2$ are all disjoint
 and homotopic. 
 There is an annulus $A(w_1,w_2)$ in $A$ with boundaries $w_1$ and $w_2$.
 By lemma~\ref{L43}.\eqref{l431} there is a non contractibe simple closed geodesic $h$ 
 with conjugate points in $\interior(A(w_1,w_2))$. 
 In particular $h$ is disjoint and homotopic to $w_i$, $a_i$, $i=1,2$.
 Denote the annuli $A_1:=A(a_1,h)$, $A_2=A(h,a_2)$ with boundaries $(a_1,h)$, $(h,a_2)$
 respectively. Then $(A_1,w_1)$, $(A_2,w_2)$ are the desired corsets.
 
 \eqref{lct2}. Write $\La:=\cap_{t\in\re}\phi_{-t}(SV)$.
 By corollary~\ref{cco}.\eqref{cco2} and remark~\ref{rpo}.\eqref{rpo2}
there is a neighborhood $N$ of $\partial D$ such that 
 $\La\cap SN=\emptyset$.   Let $W$ be the connected component 
 of  $D\setminus \pi\La$ which contains $\interior(N)$.
 Observe that $\partial D$ is non-contractible in $\ov W$.
 Let $\cC$ be a connected component in $\Emb(S^1,W)$ containing 
 curves homotopic to $\partial D$. By lemma~\ref{L41} apllied to $\cC$, there is a waist
 $w_1$ in $\ov W$. The waist $w_1$ bounds a disk $D_1$ in $D$.
 By lemma~\ref{L43}.\eqref{l432} there is a  simple  closed geodesic $h_1$
 in $\interior (D_1)$ with conjugate points. Let $B_1$ be the disk in $D$ with $\partial B_1= h_1$.
 Let $A_1=A(\partial D, h_1)$ be the annulus with boundary $\partial D\cup h_1$.
 Then $B_1$ is a bowl and $(A_1,w_1)$ is a corset with disjoint interiors 
 and  $D=A_1\cup B_1$ as required.
  
\end{proof}

\noindent{\bf Proof of theorem~\ref{TCBA}:}

If $M=\SS^2$ let $\Si_1,\Si_2$ be the Birkhoff annuli of a simple closed minimax geodesic
$m$
and let $R_1,R_2$  be the two disks bounded by $m$.
Otherwise let $R_1,\ldots, R_4$
be the disks in $M$ and  
$\Si_1$, $\Si_2$ be the surfaces of section of genus 1 obtained in 
subsection~\ref{genus1}.

Given and open subset $V\subset SM$ define the {\it forward trapped set}
$\trap_+(V)$ and the {\it backward trapped set} $\trap_-(V)$ as
$$
\trap_{\pm}(V)=\{\, z\in SM\,:\; \exists \tau\quad  \forall t>\tau\quad \phi_{\pm t}(z)\in V \,\}.
$$

\newpage

\begin{claim}\label{c410}\quad

For each $i=1,\ldots ,\{2,  4\}$ there are finitely many corsets 
$(A_1^i,w_1^i),\ldots,$ $(A_{m_i}^i,w_{m_i}^i)$ and a
bowl $B_{m_i}^i$ with disjoint interiors such that 
$R_i=A^i_1\cup\cdots\cup A^i_{m_i}\cup B^i_{m_i}$ and letting
\begin{gather}
K_i:=\partial R_i\cup \textstyle\cup_{j=1}^{m_i}(\partial A^i_j\cup w^i_j)\cup\partial B^i_{m_i},
\label{dK}\\
\label{ibdy}
\ga\text{ geodesic} \quad\ga(0)\in R_i \quad\then\quad \ga(\re)\cap K_i\ne \emptyset,\\
\trap_{\pm} SU_i\subset\textstyle \bigcup_{j=1}^{m_i} W^{s,u}(\dot w^i_j)\cup W^{s,u}(-\dot w^i_j),
\qquad
U_i:= R_i\setminus K_i.
\label{trui}
\end{gather}
\end{claim}

    \begin{figure}
   \includegraphics[scale=.35]{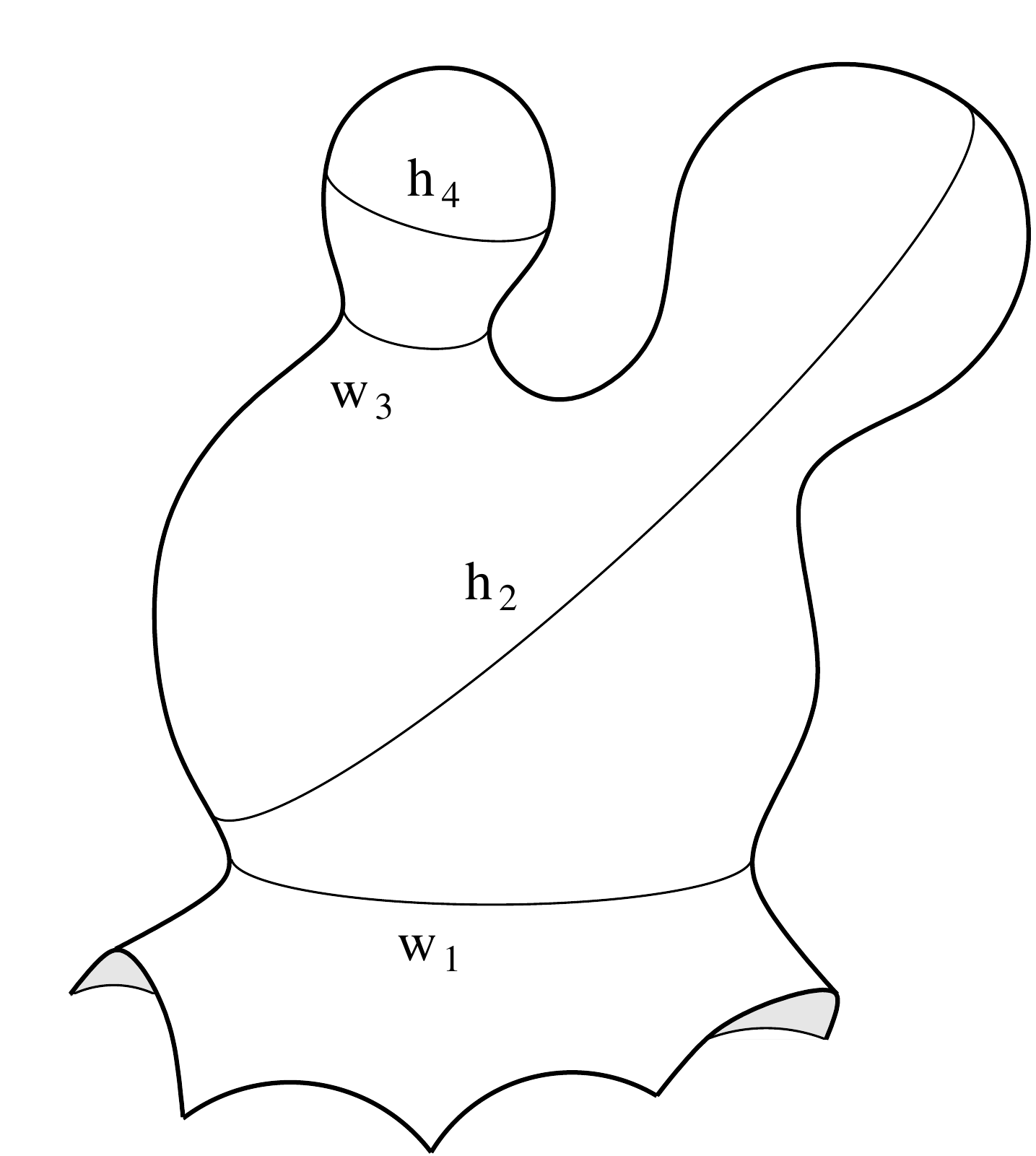}
   \caption{\footnotesize 
   This is a possible outcome of a decomposition $R_i=A_1\cup A_2\cup B_2$.
   The decomposition depends on the order in which invariant remaining subsets are given.
   The subindices are the order in which the closed orbits are found in this event.}
   \label{mickey}
   \end{figure}

Assume claim~\ref{c410} holds. 
Let $\Si_3,\ldots, \Si_{2n}$ be the collection of the two Birkhoff annuli of the 
geodesics in $K_i\cap \interior R_i$ whenever 
$\trap_\pm S(R_i\setminus\partial R_i)\ne\emptyset$.
Then \ref{TCBA}.\eqref{tcba1}-\eqref{tcba4} hold.
Also \eqref{ibdy} implies~\ref{TCBA}.\eqref{tcba5}.
By~\eqref{trui} we have that 
$$
\trap_\pm (SM\setminus\cup_{j=1}^{2n}\Si_j)\subset
\textstyle \bigcup_{i=1}^{\{2,4\}}\bigcup_{j=1}^{m_i}  
W^{s,u}(\dot w^i_j)\cup W^{s,u}(-\dot w^i_j).
$$
This implies~\ref{TCBA}.\eqref{tcba6} and~\ref{TCBA}.\eqref{tcba7}.

We have that $\Kfix=\cup_{i,j}\{\dot w^i_j,-\dot w^i_j\}$
and $\Krot =\cup_{i=1}^{2n}\partial\Si_i\setminus\Kfix$.
Since the orbits in $\Krot$ are either in a polygon or 
have conjugate points,
corollary~\ref{cco}.\eqref{cco1} and remark~\ref{rpo}.\eqref{rpo1}
imply~\ref{TCBA}.\eqref{tcba51}.

We now prove claim~\ref{c410}.
Observe that the disks $R_1,\ldots, R_{\{2,4\}}$ are bowls.
Recall that the surfaces $\Si_1,\Si_2$ can be constructed inside an arbitrarily
small neighborhood of $S(\cup_{i=1}^{\{2,4\}}\partial R_i)$.

If $\cap_{t\in\re}\phi_t(S(\interior R_i))\ne\emptyset$,
by lemma~\ref{lct}.\eqref{lct2} 
we can add a corset $(A_1,w_1)$ and a bowl $B_1$ with $R_i=A_1\cup B_1$,
$A_1\cap B_1=\partial B_1=\partial A_1\setminus\partial R_i=:b_1$.
Observe that $\{ \dot w_1,\dot b_1\}$ is a set of pairwise disjoint simple closed geodesics 
in $R_i$. Inductively, suppose we have corsets $(A_j,w_j)$, $j=1,\ldots, m$ and a bowl $B_m$
with disjoint interiors and $R_i=A_1\cup\cdots\cup A_m\cup B_m$.
Let $U=  \cup_{j=1}^m(\interior A_i - w_j)\cup\interior B_m$.
If $\La:=\cap_{t\in\re}\phi_t(SU)\ne\emptyset$ then either 
$\pi \La\cap \interior B_m\ne\emptyset$ or $\pi\La\cap (\interior A_j- w_j)\ne \emptyset$ for some $j$.
Therefore either $\cap_{t\in\re}\phi_t(S(\interior B_m))\ne\emptyset$ or
$\cap_{t\in\re}\phi_t(S(\interior{A_j}- w_j))\ne\emptyset$.
By lemma~\ref{lct} there is a corset $(A_{m+1},w_{m+1})$, and a bowl $B_{m+1}$ or 
a corset $(A'_j,w_j)$, with either  $A_j=A'_j\cup A_{m+1}$ or $B_m=A_{m+1}\cup B_{m+1}$
where the sets in the unions have disjoint interiors. In any case we obtain a new decomposition
$$
R_i=A_1\cup\cdots\cup A_{m+1}\cup B_{m+1},
$$
where $(A_j,w_j)$ are corsets and $B_{m+1}$ is a bowl, all with disjoint interiors.
The process can continue as long as 
\begin{equation}\label{conu}
\cap_{t\in\re}\phi_t(SU)\ne\emptyset, \quad U= \cup_{j=1}^{m+1}(\interior A_i - w_j)\cup\interior B_{m+1}.
\end{equation}
Here the closed simple geodesics $\{w_j\}_{j=1}^{m+1}$ are mutually disjoint waists and contained in $R_i$.
By proposition~\ref{pfin} this process must stop.
Then for each $R_i$, $i=1,\ldots,\{2,4\}$ there is $m=:m_i-1$ for which condition~\eqref{conu}
does not hold. This implies~\eqref{ibdy}.

 Let $U_i=R_i\setminus K_i$ be from~\eqref{trui}. Suppose that $z\in\trap_+(SU_i)$ then its $\om$-limit
 $$
 \om(z) := \textstyle\bigcap_{t>T}\ov{\phi_{[T,+\infty[}(z)}
 $$
 is an invariant set with projection $\pi(\om(z))\subset \ov{U_i}$.
 Since condition~\eqref{conu} does not hold for $U=U_i$ we have that 
 $\pi(\om(z))\cap U_i=\emptyset$. Therefore $\pi(\om(z))$ is a connected component 
 of $K_i$ in \eqref{dK}.
 By corollary~\ref{cco}.\eqref{cco1} and remark~\ref{rpo}.\eqref{rpo1} the forward orbit 
 of $z$ can not approach the boundary orbits in  $S(\partial A_j)$ or $SB_m$  
 without intersecting $SK_i$. Thus $\pi(\om(z)) = w^i_j$ for some $1\le j\le m_i$.
 Since $\om(z)$ is invariant, this implies that $\om(z)=\pm\dw^i_j$.
 This proves~\eqref{trui}.
 
 \qed


\begin{thebibliography}{10}

\bibitem{Andersson}
K.~G. Andersson, \emph{Poincar\'{e}'s discovery of homoclinic points}, Arch.
  Hist. Exact Sci. \textbf{48} (1994), no.~2, 133--147.

\bibitem{Birk3}
George~D. Birkhoff, \emph{Dynamical systems with two degrees of freedom},
  Trans. Amer. Math. Soc. \textbf{18} (1917), no.~2, 199--300.

\bibitem{Birk2}
\bysame, \emph{Dynamical systems}, American Mathematical Society, Providence,
  R.I., 1966.

\bibitem{CY1}
Chong-Qing Cheng and Jun Yan, \emph{Existence of {D}iffusion {Orbits} in {\sl a
  priori} {U}nstable {H}amiltonian {S}ystems}, Journal of Differential Geometry
  \textbf{67} (2004), 457--518.

\bibitem{aclarke1}
Andrew Clarke, \emph{Generic properties of geodesic flows on analytic
  hypersurfaces of {E}uclidean space}, Discrete Contin. Dyn. Syst. \textbf{42}
  (2022), no.~12, 5839--5868.

\bibitem{geod}
Gonzalo Contreras, \emph{Geodesic flows with positive topological entropy,
  twist maps and hyperbolicity}, Annals of Math. \textbf{172} (2010), no.~2,
  761--808.

\bibitem{CKMS}
Gonzalo Contreras, Gerhard Knieper, Marco Mazzucchelli, and Benjamin~H. Schulz,
  \emph{Surfaces of section for geodesic flows of closed surfaces}, Preprint
  arXiv:2204.11977, 2022.

\bibitem{CM1}
Gonzalo Contreras and Marco Mazzucchelli, \emph{Proof of the ${C}^2$ stability
  conjecture for geodesic flows of closed surfaces}, Preprint arXiv.org
  2109.10704, 2021.

\bibitem{CM2}
\bysame, \emph{Existence of {B}irkhoff sections for {K}upka-{S}male {R}eeb
  flows of closed contact 3-manifolds}, Geom. Funct. Anal. \textbf{32} (2022),
  no.~5, 951--979.

\bibitem{CP2}
Gonzalo Contreras-Barandiar\'an and Gabriel Paternain, \emph{Genericity of
  geodesic flows with positive topological entropy on {$S^2$}}, Jour. Diff.
  Geom. \textbf{61} (2002), 1--49.

\bibitem{dpmms20}
G.~De~Philippis, M.~Mazzucchelli, M.~Marini, and S.~Suhr, \emph{Closed
  geodesics on reversible {F}insler 2-spheres}, to appear in {J}. {F}ixed
  {P}oint {T}heory {A}ppl., 2020.

\bibitem{DLS1}
Amadeu Delshams, Rafael de~la Llave, and Tere~M. Seara, \emph{A geometric
  approach to the existence of orbits with unbounded energy in generic periodic
  perturbations by a potential of generic geodesic flows of {${\bf T}\sp 2$}},
  Comm. Math. Phys. \textbf{209} (2000), no.~2, 353--392.

\bibitem{FM}
Benson Farb and Dan Margalit, \emph{A primer on mapping class groups},
  Princeton Mathematical Series, vol.~49, Princeton University Press,
  Princeton, NJ, 2012.

\bibitem{fakri}
Bassam Fayad and Rapha\"{e}l Krikorian, \emph{Herman's last geometric theorem},
  Ann. Sci. \'{E}c. Norm. Sup\'{e}r. (4) \textbf{42} (2009), no.~2, 193--219.

\bibitem{FLC}
John Franks and Patrice Le~Calvez, \emph{Regions of instability for non-twist
  maps}, Ergodic Theory Dynam. Systems \textbf{23} (2003), no.~1, 111--141.

\bibitem{gage}
Michael~E. Gage, \emph{Curve shortening on surfaces}, Ann. Sci. \'{E}cole Norm.
  Sup. (4) \textbf{23} (1990), no.~2, 229--256.

\bibitem{genecand}
Chr. Genecand, \emph{Transversal homoclinic orbits near elliptic fixed points
  of area-preserving diffeomorphisms of the plane}, Dynamics reported, Dynam.
  Report. Expositions Dynam. Systems (N.S.), vol.~2, Springer, Berlin, 1993,
  pp.~1--30.

\bibitem{GLS}
Marian Gidea, Rafael de~la Llave, and Tere M-Seara, \emph{A general mechanism
  of diffusion in {H}amiltonian systems: qualitative results}, Comm. Pure Appl.
  Math. \textbf{73} (2020), no.~1, 150--209.

\bibitem{Grayson}
Matthew~A. Grayson, \emph{Shortening embedded curves}, Ann. of Math. (2)
  \textbf{129} (1989), no.~1, 71--111.

\bibitem{HWZ96}
H.~Hofer, K.~Wysocki, and E.~Zehnder, \emph{Properties of pseudoholomorphic
  curves in symplectisations. i: Asymptotics}, Ann. Inst. Henri Poincar\'e,
  Anal. Non Lin\'eaire \textbf{13} (1996), no.~3, 337--379, correction ibid.
  15, No.4, 535-538 (1998).

\bibitem{HWZ1}
\bysame, \emph{Finite energy foliations of tight three-spheres and
  {H}amiltonian dynamics}, Ann. of Math. (2) \textbf{157} (2003), no.~1,
  125--255.

\bibitem{irie1}
Kei Irie, \emph{Dense existence of periodic {R}eeb orbits and {ECH} spectral
  invariants}, J. Mod. Dyn. \textbf{9} (2015), 357--363.

\bibitem{JMM}
M.~Juvan, A.~Malni\v{c}, and B.~Mohar, \emph{Systems of curves on surfaces}, J.
  Combin. Theory Ser. B \textbf{68} (1996), no.~1, 7--22.

\bibitem{KW2}
Gerhard Knieper and Howard Weiss, \emph{{$C\sp \infty$} genericity of positive
  topological entropy for geodesic flows on {$S\sp 2$}}, J. Differential Geom.
  \textbf{62} (2002), no.~1, 127--141.

\bibitem{LC2}
Patrice Le~Calvez, \emph{\sl {{\'E}}tude {T}opologique des {A}pplications
  {D}\'eviant la {V}erticale}, Ensaios Matem\'aticos, {\bf 2}, Sociedade
  Brasileira de Matem\'atica, Rio de Janeiro, 1990.

\bibitem{Mat9}
{John N.} Mather, \emph{Invariant subsets for area preserving homeomorphisms of
  surfaces}, Mathematical analysis and applications, Part B, Academic Press,
  New York-London, 1981, pp.~531--562.

\bibitem{Mat15}
John~N. Mather, \emph{Arnol'd diffusion. {I}. {A}nnouncement of results},
  Sovrem. Mat. Fundam. Napravl. \textbf{2} (2003), 116--130, translation in
  {J}. {M}ath. {S}ci. ({N}.{Y}.) 124 (2004), no. 5, 5275–5289.

\bibitem{Moser}
J{\"u}rgen Moser, \emph{Stable and random motions in dynamical systems},
  Princeton University Press, Princeton, N. J., 1973, With special emphasis on
  celestial mechanics, Hermann Weyl Lectures, the Institute for Advanced Study,
  Princeton, N. J, Annals of Mathematics Studies, No. 77.

\bibitem{Ol3}
Fernando Oliveira, \emph{On the generic existence of homoclinic points},
  Ergodic Theory Dynam. Systems \textbf{7} (1987), no.~4, 567--595.

\bibitem{OC2}
Fernando Oliveira and Gonzalo Contreras, \emph{The {I}deal {B}oundary and the
  {A}ccumulation {L}emma}, Preprint arXiv.2205.14738, 2022.

\bibitem{OC1}
\bysame, \emph{No elliptic points from fixed prime ends}, Preprint
  arXiv:2205.14768, 2022.

\bibitem{pixton}
Dennis Pixton, \emph{Planar homoclinic points}, J. Differential Equations
  \textbf{44} (1982), no.~3, 365--382.

\bibitem{PoincareI}
H.~Poincar\'{e}, \emph{Les m\'{e}thodes nouvelles de la m\'{e}canique
  c\'{e}leste. {T}ome {I}}, Gauthier-Villars, Paris, 1892, Solutions
  p\'{e}riodiques. Non-existence des int\'{e}grales uniformes. Solutions
  asymptotiques.

\bibitem{PoincareIII}
\bysame, \emph{Les m\'{e}thodes nouvelles de la m\'{e}canique c\'{e}leste.
  {T}ome {III}}, Gauthier-Villars, Paris, 1899, Invariant int\'{e}graux.
  Solutions p\'{e}riodiques du deuxi\`eme genre. Solutions doublement
  asymptotiques.

\bibitem{XiaZhang}
Zhihong Xia and Pengfei Zhang, \emph{Homoclinic intersections for geodesic
  flows on convex spheres}, Dynamical systems, ergodic theory, and probability:
  in memory of {K}olya {C}hernov, Contemp. Math., vol. 698, Amer. Math. Soc.,
  Providence, RI, 2017, pp.~221--238.

\bibitem{Zehnder1}
E.~Zehnder, \emph{Homoclinic points near elliptic fixed points}, Comm. Pure
  Appl. Math. \textbf{26} (1973), 131--182. \MR{345134}

\end{thebibliography}

\def\cprime{$'$} \def\cprime{$'$} \def\cprime{$'$} \def\cprime{$'$}
\providecommand{\bysame}{\leavevmode\hbox to3em{\hrulefill}\thinspace}
\providecommand{\MR}{\relax\ifhmode\unskip\space\fi MR }
\providecommand{\MRhref}[2]{%
  \href{http://www.ams.org/mathscinet-getitem?mr=#1}{#2}
}
\providecommand{\href}[2]{#2}

\end{document}